\documentclass[12pt]{amsart}
\usepackage{mathrsfs}
\usepackage{euscript}

\def\today{September 19, 2011}
\catcode`\@=11
\def\@evenfoot{\rule{0pt}{20pt}[\today] \hfill [{\tt \jobname.tex}]}
\def\@oddfoot{\rule{0pt}{20pt}{[\tt \jobname.tex}]\hfill [\today]}
\catcode`\@=13
\newtheorem{theorem}{Theorem}
\newtheorem{proposition}[theorem]{Proposition}
\newtheorem{lemma}[theorem]{Lemma}

\newtheorem*{conjecture}{Conjecture}
\newtheorem{corollary}[theorem]{Corollary}

\def\CH{{\it CH}}

\def\Tot{{\rm Tot\/}} 

\def\Coend{{\it Coend}}
\def\spanext(#1,#2,#3,#4,#5){
\unitlength 1.5mm
\linethickness{0.4pt}
\begin{picture}(19,6)(3.5,11)
\thicklines
\put(11,13){\vector(-1,-1){3.2}}
\put(15,13){\vector(1,-1){3.2}}
\put(7.7,8){\makebox(0,0)[r]{$#1$}}
\put(18.3,8){\makebox(0,0)[l]{$#2$}}
\put(13,14){\makebox(0,0)[b]{$#3$}}
\put(9,12.5){\makebox(0,0)[r]{\scriptsize $#4$}}
\put(18,12.5){\makebox(0,0)[l]{\scriptsize $#5$}}
\end{picture}
}

\def\eqv{
{
\unitlength=.5pt
\begin{picture}(60.00,20.00)(-10.00,2.00)
\thicklines
\put(40.00,15.00){\vector(2,-1){0.00}}
\put(40.00,5.00){\vector(2,1){0.00}}
\qbezier(0.00,5.00)(20.00,-5.00)(40.00,5.00)
\qbezier(0.00,15.00)(20.00,25.00)(40.00,15.00)
\end{picture}}
}
\def\0-span(#1,#2,#3){\spanext(#1,#2,#3,s,t)}
\def\I{{\EuScript I}}
\def\dvespanextnahore(#1,#2,#3,#4,#5,#6,#7,#8,#9){
\unitlength 1.7mm
\begin{picture}(20,12)(3.5,5.5)
\thicklines
\put(11,13.5){\vector(-1,-1){3.5}}
\put(15,13.5){\vector(1,-1){3.5}}
\put(7,8){\makebox(0,0)[cc]{$#3$}}
\put(19,8){\makebox(0,0)[cc]{$#4$}}
\put(7,6.5){\vector(0,-1){5.5}}
\put(19,6.5){\vector(0,-1){5.5}}
\put(9,6){\vector(4,-3){8}}
\put(17,6){\vector(-4,-3){8}}
\put(7,0){\makebox(0,0)[t]{$#1$}}
\put(19,0){\makebox(0,0)[t]{$#2$}}
\put(13,14){\makebox(0,0)[b]{$#5$}}
\put(6,4){\makebox(0,0)[r]{\scriptsize $#6$}}
\put(15,6){\makebox(0,0)[cc]{\scriptsize $#6$}}
\put(11.5,6){\makebox(0,0)[cc]{\scriptsize $#7$}}
\put(20.5,4){\makebox(0,0)[l]{\scriptsize $#7$}}
\put(8,12){\makebox(0,0)[rb]{\scriptsize $#8$}}
\put(18,12){\makebox(0,0)[lb]{\scriptsize $#9$}}
\end{picture}
}

\def\dvespanext(#1,#2,#3,#4,#5,#6,#7){
\dvespanextnahore(#1,#2,#3,#4,#5,#6,#7,#6,#7)
}

\def\spanjennahore(#1,#2,#3,#4,#5){
\posun{3.5}
\unitlength 1.7mm
\begin{picture}(20,12)(3.5,5.5)
\thicklines
\put(11,13.5){\vector(-1,-1){3.5}}
\put(15,13.5){\vector(1,-1){3.5}}
\put(7,8){\makebox(0,0)[cc]{$#3$}}
\put(19,8){\makebox(0,0)[cc]{$#4$}}
\put(7,6.5){\vector(0,-1){5.5}}
\put(19,6.5){\vector(0,-1){5.5}}
\put(9,6){\vector(4,-3){8}}
\put(17,6){\vector(-4,-3){8}}
\put(7,0){\makebox(0,0)[t]{$#1$}}
\put(19,0){\makebox(0,0)[t]{$#2$}}
\put(13,14){\makebox(0,0)[b]{$#5$}}
\put(8,12){\makebox(0,0)[rb]{\scriptsize $s$}}
\put(18,12){\makebox(0,0)[lb]{\scriptsize $t$}}
\end{picture}
}

\def\jarca#1#2#3#4{
\unitlength 4mm \posunm{1.2}
\linethickness{0.4pt}
\begin{picture}(16,4)(1,9.25)
\thicklines
\put(8.5,12){\vector(1,0){5}}
\put(7,8){\vector(2,1){6.6}}
\put(6,11){\vector(0,-1){3}}
\put(6,12){\makebox(0,0)[cc]{$#1$}}
\put(15,12){\makebox(0,0)[cc]{$#2$}}
\put(6,7){\makebox(0,0)[cc]{$#3$}}
\put(5.5,9.5){\makebox(0,0)[rc]{\scriptsize $#4$}}
\put(11,9){\makebox(0,0)[cc]{\scriptsize $\gamma$}}
\put(11,12.25){\makebox(0,0)[bc]{\scriptsize $\cong$}}
\end{picture}
}

\def\pampeliska#1#2#3{
\posunm{.5}
\unitlength 1.6mm
\linethickness{0.4pt}
\ifx\plotpoint\undefined\newsavebox{\plotpoint}\fi 
\begin{picture}(19,6.6)(6,10)
\thicklines
\put(11,14){\line(1,-1){4}}
\put(15,10){\line(1,1){4}}
\put(15,10){\line(-1,2){2}}
\put(15,10){\line(0,1){0}}
\put(15,10){\line(0,-1){3}}
\put(11,14.5){\makebox(0,0)[bc]{\scriptsize $#1$}}
\put(13,14.5){\makebox(0,0)[bc]{\scriptsize $#2$}}
\put(19,14.5){\makebox(0,0)[bc]{\scriptsize $#3$}}
\put(15,10){\makebox(0,0)[cc]{\scriptsize $\bullet$}}
\put(16,14){\makebox(0,0)[cc]{\scriptsize $\cdots$}}
\end{picture}
}

\def\tens{\mbox{$\odot\hskip-0.63em\star\hskip0.3em$}}
\def\II#1{{1_{#1}}}

\def\Fact_#1{{\EuScript Fct}_{#1}(O)}\def\Deltaalg{{\Delta_{\it alg}}}

\def\vlra{{\hbox{$-\hskip-1mm-\hskip-2mm\longrightarrow$}}}

\def\vvlra{{\hbox{$-\hskip-2mm-\hskip-2mm-\hskip-2mm-\hskip-2mm-\hskip-2mm-\hskip-2mm-\hskip-2mm-\hskip-2mm\longrightarrow$}}}
\def\UJarka{{\mathcal U}\Jarka}\def\eAss{\underline{{\mathcal As}}}
\def\fAss{\underline{{\mathcal Ass}}}
\def\rada#1#2{{#1,\ldots,#2}}
\def\BTree{{\it BTree}}
\def\ATree{{\it ATree}}
\def\calC{{\mathcal C}}
\def\calK{{\mathcal K}}
\def\euo{\mathscr{E}} 
\def\calL{{\mathcal L}}
\def\ttD{{\mathtt D}}
\def\funny{\EuScript{S}p_2}
\def\Und{{\mathcal U}}
\def\id{{\it id}}\def\Id{{\it Id}}
\def\nubar{{\overline \nu}}\def\mubar{{\overline \mu}}\def\ubar{{\overline u}}
\def\ttD{{\mathtt D}}
\def\1-span(#1,#2,#3,#4,#5){\dvespanext(#1,#2,#3,#4,#5,s,t)}
\def\boxx{{\hskip .2em \Box}}\def\wt{\widetilde}
\def\posun#1{{\raisebox{-#1em}{\rule{0pt}{0pt}}}}
\def\posunm#1{{\raisebox{-#1cm}{\rule{0pt}{0pt}}}}
\def\Chain{\EuScript Chain}\def\Set{\EuScript Set}
\def\calUV{{\mathcal U}\calA} 
\def\Gl{{{\EuScript G}\rm l}(\cat)}
\def\sfG{{\mathsf G}}\def\ttE{{\mathtt E}}\def\ttD{{\mathtt D}}\def\ot{\otimes}
\def\bbU{{\mathbb U}}
\def\duo{\mathscr{D}} 

 \def\Tam{{\mathcal T}\hskip-0.2em{am}} 
\def\ttF{{\mathtt F}}\def\ttG{{\mathtt G}}\def\tte{{\mathtt e}}
\def\coll#1{\{#1(n)\}_{n\geq 0}}
\def\Rada#1#2#3{#1_{#2},\dots,#1_{#3}}\def\Ob{{\rm Ob}}
\def\Jarka{\EuScript{J}}
\def\sfF{{\mathsf F}}\def\ttM{{\mathtt M}}     \def\ttN{{\mathtt N}}                     \def\bfhom{{\mathbf {hom}}}
\def\cat{\mathscr{C}}\def\calA{{\mathcal A}}
\def\id{{\it id}}
\def\End{\hbox{${\mathcal E}${\hskip -.1em\it nd\hskip .1em}}}
\def\EE{{\mathcal E}}\def\bfRhom{{\mathbf {Rhom}}}
\def\glb#1#2#3#4{{\rm glb}\left(#1,#2;#3,#4\right)}
\def\Cat{{\EuScript Cat}}\def\Mon{{\EuScript Mon}}

\def\Ol{\widetilde l}\def\Of{\widetilde f}\def\Og{\widetilde g}
\def\Oh{\widetilde h}\def\Ol{\widetilde l}\def\Sp{{\it Sp}}
\def\Ohf{\widetilde {\it hf}}\def\Olg{\widetilde{\it lg}}
\def\ll{\hbox{$\left( \posun{.4} \right. \hskip -.0em$}}
\def\rr{\hbox{$\hskip -.0em \left. \posun{.4} \right)$}}
\def\spn{span}\def\calU{{\mathcal U}}

\def\Gr{{\EuScript Grph}}

\def\cases#1#2#3#4{
                  \left\{
                         \begin{array}{ll}
                           #1,\ &\mbox{#2}
                           \\
                           #3,\ &\mbox{#4}
                          \end{array}
                   \right.
}

\def\jaruskaext#1#2#3#4#5#6{
\unitlength 3mm 
\thicklines
\begin{picture}(22,8)(5.5,1)
\thicklines
\put(11,0){
\put(19,7){\vector(1,-1){2}}
\put(17,7){\vector(-1,-1){4}}
\put(2.5,2){\vector(1,0){8}}
\put(15,4){\vector(1,0){5.5}}
\put(12,2){\makebox(0,0)[cc]{${\wt C}$}}
\put(22,4){\makebox(0,0)[cc]{${\wt C}$}}
\put(12.25,4){\line(1,0){.75}}
\put(15.73,8){\vector(1,0){1}}
\put(7.93,8){\line(1,0){.9}}
\put(9.73,8){\line(1,0){.9}}
\put(11.53,8){\line(1,0){.9}}
\put(13.33,8){\line(1,0){.9}}
\put(15.13,8){\line(1,0){.9}}
\put(18,8){\makebox(0,0)[cc]{$\ttG_C$}}
\put(6,2.3){\makebox(0,0)[b]{\scriptsize $\Oh$}}
\put(17.3,4.3){\makebox(0,0)[b]{\scriptsize $\Ol$}}
\put(15,6){\makebox(0,0)[cc]{\scriptsize $s$}}
\put(20,7){\makebox(0,0)[cc]{\scriptsize $t$}}
}
\put(6,7){\vector(-1,-1){4}}
\put(8,7){\vector(1,-1){2}}
\put(19,7){\vector(1,-1){2}}
\put(17,7){\vector(-1,-1){4}}
\put(2.25,2){\vector(1,0){8}}
\put(15,4){\vector(1,0){5.5}}
\put(1,2){\makebox(0,0)[cc]{$#1$}}
\put(12,2){\makebox(0,0)[cc]{$#2$}}
\put(22,4){\makebox(0,0)[cc]{$#2$}}
\put(11,4){\makebox(0,0)[cc]{$#1$}}
\put(12.25,4){\line(1,0){.75}}
\put(15.73,8){\vector(1,0){1}}
\put(7.93,8){\line(1,0){.9}}
\put(9.73,8){\line(1,0){.9}}
\put(11.53,8){\line(1,0){.9}}
\put(13.33,8){\line(1,0){.9}}
\put(15.13,8){\line(1,0){.9}}
\put(7,8){\makebox(0,0)[cc]{$#3$}}
\put(18,8){\makebox(0,0)[cc]{$#4$}}
\put(6,2.3){\makebox(0,0)[b]{\scriptsize $#5$}}
\put(17.3,4.3){\makebox(0,0)[b]{\scriptsize $#6$}}
\put(4,6){\makebox(0,0)[cc]{\scriptsize $s$}}
\put(15,6){\makebox(0,0)[cc]{\scriptsize $s$}}
\put(9,7){\makebox(0,0)[cc]{\scriptsize $t$}}
\put(20,7){\makebox(0,0)[cc]{\scriptsize $t$}}
\end{picture}
}

\def\pullback#1#2#3#4#5#6{
\posun{3}
\unitlength 2.8mm
\begin{picture}(16,6.5)(2,7.57)
\thicklines
\put(9.5,12){\vector(-1,-1){3}}
\put(12.5,12){\vector(1,-1){3}}
\put(6.5,7){\vector(1,-1){3}}
\put(15.5,7){\vector(-1,-1){3}}
\thinlines
\put(11,10){\line(-1,1){1.3}}
\put(11,10){\line(1,1){1.3}}
\put(11,3){\makebox(0,0)[cc]{$#1$}}
\put(8,5){\makebox(0,0)[rt]{\scriptsize $#2$}}
\put(14,5){\makebox(0,0)[lt]{\scriptsize $#3$}}
\put(6,8){\makebox(0,0)[cc]{$#4$}}
\put(16,8){\makebox(0,0)[cc]{$#5$}}
\put(11,13){\makebox(0,0)[cc]{$#6$}}
\end{picture}
}

\def\jaruska#1#2#3#4#5#6{
\unitlength 3mm 
\thicklines
\begin{picture}(22,8)(0,2)
\thicklines
\put(6,7){\vector(-1,-1){4}}
\put(8,7){\vector(1,-1){2}}
\put(19,7){\vector(1,-1){2}}
\put(17,7){\vector(-1,-1){4}}
\put(2.5,2){\vector(1,0){8}}
\put(15,4){\vector(1,0){5.5}}
\put(1,2){\makebox(0,0)[cc]{$#1$}}
\put(12,2){\makebox(0,0)[cc]{$#2$}}
\put(22,4){\makebox(0,0)[cc]{$#2$}}
\put(11,4){\makebox(0,0)[cc]{$#1$}}
\put(12.25,4){\line(1,0){.75}}
\put(15.73,8){\vector(1,0){1}}
\put(7.93,8){\line(1,0){.9}}
\put(9.73,8){\line(1,0){.9}}
\put(11.53,8){\line(1,0){.9}}
\put(13.33,8){\line(1,0){.9}}
\put(15.13,8){\line(1,0){.9}}
\put(7,8){\makebox(0,0)[cc]{$#3$}}
\put(18,8){\makebox(0,0)[cc]{$#4$}}
\put(6,2.3){\makebox(0,0)[b]{\scriptsize $#5$}}
\put(17.3,4.3){\makebox(0,0)[b]{\scriptsize $#6$}}
\put(4,6){\makebox(0,0)[cc]{\scriptsize $s$}}
\put(15,6){\makebox(0,0)[cc]{\scriptsize $s$}}
\put(9,7){\makebox(0,0)[cc]{\scriptsize $t$}}
\put(20,7){\makebox(0,0)[cc]{\scriptsize $t$}}
\end{picture}
}

\def\globe#1#2#3#4{\raisebox{-1.2em}{\rule{0pt}{0pt}}
{
\unitlength=.05em
\begin{picture}(115.00,20.00)(-20.00,18.00)
\put(40.00,4.00){\makebox(0.00,0.00)[t]{\scriptsize$#4$}}
\put(40.00,43.00){\makebox(0.00,0.00)[b]{\scriptsize$#3$}}
\put(78.00,25.00){\makebox(0.00,0.00)[l]{$#2$}}
\put(2,25.00){\makebox(0.00,0.00)[r]{$#1$}}
\put(70.00,20.00){\vector(1,1){0}}
\put(70.00,30.00){\vector(1,-1){0}}
\qbezier(10.00,30.00)(20.00,40.00)(40.00,40.00)
\qbezier(40.00,40.00)(60.00,40.00)(70.00,30.00)
\qbezier(40.00,10.00)(60.00,10.00)(70.00,20.00)
\qbezier(10.00,20.00)(20.00,10.00)(40.00,10.00)
\end{picture}}
}

\def\ssglobe#1#2#3#4{\raisebox{-1.2em}{\rule{0pt}{0pt}}
{
\unitlength=.03em
\begin{picture}(115.00,20.00)(-20.00,18.00)
\put(40.00,4.00){\makebox(0.00,0.00)[t]{\scriptsize${}_{#4}$}}
\put(40.00,43.00){\makebox(0.00,0.00)[b]{\scriptsize${}_{#3}$}}
\put(78.00,25.00){\makebox(0.00,0.00)[l]{\scriptsize$#2$}}
\put(2,25.00){\makebox(0.00,0.00)[r]{\scriptsize$#1$}}
\put(70.00,20.00){\vector(3,2){0}}
\put(70.00,30.00){\vector(3,-2){0}}
\qbezier(10.00,30.00)(20.00,40.00)(40.00,40.00)
\qbezier(40.00,40.00)(60.00,40.00)(70.00,30.00)
\qbezier(40.00,10.00)(60.00,10.00)(70.00,20.00)
\qbezier(10.00,20.00)(20.00,10.00)(40.00,10.00)
\end{picture}}\hskip .2em
}

\def\kategorie#1#2{
{
  \unitlength=.05em
\begin{picture}(50.00,40.00)(0.00,10.00)
\thicklines
\put(50.00,30.00){\makebox(0.00,0.00)[l]{\scriptsize$#1$}}
\put(0.00,30.00){\makebox(0.00,0.00)[r]{\scriptsize$#2$}}
\put(3,0){
\put(30.00,10.00){\vector(-1,-1){0}}
\qbezier(40.00,30.00)(40.00,20.00)(30.00,10.00)
\qbezier(30.00,50.00)(40.00,40.00)(40.00,30.00)
}
\put(-3,0){
\put(20.00,10.00){\vector(1,-1){0}}
\qbezier(10.00,30.00)(10.00,20.00)(20.00,10.00)
\qbezier(20.00,50.00)(10.00,40.00)(10.00,30.00)
}
\end{picture}}
}

\theoremstyle{definition}
\newtheorem{definition}[theorem]{Definition}
\newtheorem{example}[theorem]{Example}
\newtheorem{convention}[theorem]{Convention}
\newtheorem{remark}[theorem]{Remark}
\newtheorem{observation}[theorem]{Observation}

\title[Centers and homotopy centers]{Centers and homotopy centers in enriched monoidal categories.}
\author[M.~Batanin and M.~Markl]{Michael Batanin and Martin~Markl}
\thanks{The first author acknowledges  the financial
support of Scott Russel Johnson Memorial Foundation, Max Plank
Institut f\"{u}r Mathematik and Australian Research Council (grant
No.~DP1095346). The second author was supported by the grant GA \v CR
201/08/0397 and by the Academy of Sciences of the Czech Republic,
Institutional Research Plan No.~AV0Z10190503.}

\catcode`\@=11
\address{Macquarie University, NSW 2109, Australia}
\email{michaelbatanin@mq.edu.au}
\address{Mathematical Institute of the Academy, {\v Z}itn{\'a} 25,
         115 67 Prague 1, The Czech Republic}
\email{markl@math.cas.cz}
\catcode`\@=13

\keywords{Monoidal categories, center, Hochschild complex, operads, Deligne's conjecture} 
\subjclass{Primary 18D10,18D20, 18D50, secondary 55U40, 55P48.}

\begin{document}
\bibliographystyle{plain}

\begin{abstract}
We consider a theory of centers and homotopy centers of monoids in
monoidal categories which themselves are enriched in duoidal
categories. The duoidal categories (introduced by Aguillar and
Mahajan under the name $2$-monoidal categories) are categories with
two monoidal structures which are related by some, not
necessary invertible, coherence morphisms.  Centers of monoids in this sense
include many examples which are not `classical.' In particular, the
$2$-category of categories is an example of a center in our
sense. Examples of homotopy center (analogue of the classical Hochschild
complex) include the $\mathbf{Gray}$-category $\mathbf{Gray}$ of
$2$-categories, $2$-functors and pseudonatural transformations and
Tamarkin's homotopy $2$-category of $dg$-categories, $dg$-functors and
coherent $dg$-transformations.

\end{abstract}

\maketitle
\setcounter{tocdepth}{1}
{\small \tableofcontents}

\baselineskip 17pt plus 2pt minus 1 pt

\section{Introduction}

This paper grew up from our attempts to comprehend a construction by D.~Tamarkin
in~\cite{tamarkin:CM07} which answers a question: what do
$dg$-categories, $dg$-functors, and coherent up to all higher
homotopies $dg$-transformations, form? In the process we discovered
that the most natural language which allows easy development of such a
construction is a generalization to the enriched categorical context of the
classical Hochschild complex theory for algebras. Enrichment,
however, should be understood in a more general sense.  In this paper
we therefore want to set up some basic definitions and constructions
of the proposed theory of enrichment and the corresponding theory of
centers and the Hochschild complexes. In the sequel
\cite{batanin-berger-markl:homotopycenter} we will consider
homotopical aspects of the theory of the Hochschild complexes. The higher
dimensional generalization of our theory will be addressed in yet
another paper.

Classically, the Hochschild complex of an associative algebra can be
understood as its derived or homotopical center.  Our theory
generalizes this point of view by extending the notions of center and
homotopy center to a much larger class of monoids. Of course, the
classical center construction and the Hochschild complex are special
cases of the center in our sense. But, perhaps, the most striking
feature of our theory is that the $2$-category of categories is an
example of our center construction as well.  An example of a
homotopy center is then the symmetric monoidal closed category
$\mathbf{Gray}$ of $2$-categories, $2$-functors and pseudonatural
transformations \cite{GPS}.  Tamarkin's homotopy 2-category of
$dg$-categories, $dg$-functors and their coherent natural
transformations is also an example of the homotopy center.  In some
philosophical sense, we have here a new understanding of the center as a
universal method for (higher dimensional) enrichment.
Other nontrivial examples of duoidal categories and centers are
presented in the lecture notes of Ross Street concerning invariants of
$3$-dimensional manifolds~\cite{Street}.

Let us now provide more detail about where we enrich. Classically, we
can enrich over any monoidal category $\duo.$ However, monoidal
$\duo$-enriched categories make sense only if $\duo$ has some degree
of commutativity, more precisely, we need $\duo$ to be braided. It was
observed by Forcey in \cite{Forcey} that we can slightly weaken this
requirement. It is enough for $\duo$ to be $2$-fold monoidal in his
sense.  Even Forcey's conditions can be weaken. It is enough for
$\duo$ to be $2$-monoidal in Aguillar-Mahajan sense \cite{AM}.  In our
paper we call such a $\duo$ a {\it duoidal category}.\footnote{This 
terminology was proposed by Ross Street and we found it very convenient. 
The terminology of \cite{AM} suffers from the existence of a
similarly sounding terminology of Balteanu, Fiedorowicz, Schw\"anzl, and
Vogt \cite{BFSV}, and Forcey \cite{Forcey}.} In a duoidal category we
have two tensor products with the corresponding unit objects making $\duo$
a monoidal category in two different ways. In addition, we require that
these two tensor products are related by a not necessary invertible
middle interchange law and that the unit objects also satisfy some interesting
coherence relations.

Let $\duo$ be such a duoidal category. To incorporate the theory of the
Hochschild complex, we also assume that $\duo$ itself is enriched over
a base closed symmetric monoidal category $V.$ In this situation one
can consider a monoidal category $\calK$ enriched in $\duo$ (and the
underlying category of $\calK$ is a monoidal $V$-category). Then we
could define a monoid in $\calK$ in a usual way as an object equipped
with a unit and an associative multiplication. But, unexpectedly, such
a monoid notion is in general not correct -- it amounts to a monoid in
the underlying monoidal category.  We introduce a new notion of a
monoid in $\calK$ by adding  more unitary operations which makes all
theory nontrivial (all these operations coincide if $\duo$ is a braided
monoidal category, so their multitude is not visible in the classical theory). Then for
a monoid $\ttM$ in $\calK$ we define a cosimplicial object in $\duo$
which is an analogue of the classical Hochschild cosimplicial complex of an
algebra.  If we fix a cosimplicial object $\delta$ in $V$, we can take
a kind of geometric realization of the cosimplicial Hochschild
complex. This is an object $\CH_{\delta}(\ttM,\ttM)$ from $\duo$ and
this is our definition of the $\delta$-center of a monoid. If $\delta$ is
the constant object $\delta_n = I$ (the unit of $V$), then the
$\delta$-center is called the center.  Notice that the $\delta$-center of a
monoid lives in $\duo$, not in $\calK.$ When $V, \duo$ and $\calK$ have
compatible model structures, we also define a~homotopy center
$\CH(\ttM,\ttM)$ of $\ttM$ by using an appropriate cofibrant and
contractible $\delta$ and a~fibrant replacement of $\ttM$.

In the classical theory we know that the center of a monoid is a
commutative monoid. We have an analogue of this statement in our
settings: the center of a monoid in $\calK$ is a duoid (double monoid
in the terminology of \cite{AM}) in $\duo.$ The homotopy version of
this statement is the following: there is a canonical action of a
contractible $2$-operad on the homotopy center of a monoid in $\calK
.$ This is an analogue of Deligne's conjecture for the classical
Hochschild complex. Tamarkin's main result from \cite{tamarkin:CM07}
is a special case of this Deligne's conjecture applied to a particular
monoid in a monoidal category $\Jarka(O,\Chain)$ constructed in
Section~\ref{sec:tamark-compl-funct-1}.  The classical Deligne's
conjecture follows from this statement by a theorem from
\cite{batanin:conf} if $\duo$ is a symmetric monoidal category.  In
this paper we set up a version of the theory of $2$-operads which
allows a precise formulation of such a statement. A proof of this form
of Deligne's conjecture will be given
in~\cite{batanin-berger-markl:homotopycenter}.

Finally, let us say a few words about possible further directions.
One interesting and almost obvious possibility is to replace duoidal
categories by $n$-oidal categories. An $n$-oidal category is a
category with $n$ monoidal structures related by interchange morphisms
and various coherence morphisms between unit objects which satisfy
some coherence relations~\cite{AM}.  Many results of our paper admit
more or less obvious generalization to the $n$-oidal case. In particular,
we can consider $n$-oids in $n$-oidal categories and centers and
homotopy centers of $n$-oids.

\begin{conjecture}[$(n+1)$-oidal Deligne's conjecture]
There is a canonical action of a contractible $(n+1)$-operad on the
homotopy center of an $n$-oid $\ttN$ which lifts the $(n+1)$-oid
structure on the center of $\ttN.$
\end{conjecture}

Analogously to Tamarkin's theorem, this conjecture answers a question:
what do $n$-categories enriched in a symmetric monoidal model category
$V$ form?  This conjecture should imply also the $n$-dimensional form
of the classical Deligne conjecture \cite{Kontsevich} via the results
of~\cite{batanin:AM08,batanin:conf}. We hope to address the proof of
these conjectures in the near future.

Another very interesting direction is a construction of the so-called
semistrict $n$-categories. In the theory of higher dimensional
categories it is highly desirable to have some sort of a minimal model 
of the theory of weak $n$-categories.  Many important statements in higher
category theory, like the equivalences amongst almost all definitions of weak
$n$-categories, or the Grothendieck hypothesis on algebraic models of
$n$-homotopy types \cite{C}, will follow naturally once we have at
hands a well developed theory of semistrict $n$-categories.  So far,
however, a good notion of semistrict $n$-category is known only for
$n\le 3.$ For $n=2$, it is the category of strict $2$-categories. For
$n=3$ it is the category of $\mathbf{Gray}$-categories
\cite{GPS}. Both these categories are examples of enrichment over a
closed symmetric monoidal category which comes from our homotopical
center construction (see Examples \ref{2cat} and \ref{Graycat}).
Combining the results of \cite{BCW1} with the approach of our paper, we
hope to be able to construct an analogue of the Gray tensor product for
all dimensions and therefore a good theory of semistrict
$n$-categories. This is currently a work in progress with M.~Weber and
D.-C.~Cisinski~\cite{BCW2}.

\noindent
{\bf Acknowledgment.}  
We would like to express our thanks to  C.~Berger, D.-C.~Cisinski, 
S.~Lack, R.~Street, D.~Tamarkin and M.~Weber for many enlightening  discussions. 

\section{Monoidal $V$-categories and duoidal $V$-categories}

We fix from the beginning a complete and cocomplete closed symmetric
monoidal category $(V,\otimes,I)$. Its underlying category is denoted
$\calU V.$ For objects $X,Y$ of an $V$-enriched category $\calA$, we
denote by $\calA(X,Y)\in V$ the enriched hom and by
$\calU\calA(X,Y) := \calU V\big(I,\calA(X,Y)\big)$ the set of homomorphism in the
underlying category.

\subsection{Monoidal $V$-categories.}
It is classical~\cite{eilenberg-kelly} that $V$-categories, $V$-functors
and $V$-natural transformations form a $2$-category $\Cat(V).$
Moreover, this $2$-category is a symmetric monoidal $2$-category with
respect to the tensor product $\times_V$ of $V$-categories:
\begin{align*}
Ob(K\times_V L) &:= Ob(K)\times Ob(L),
\\
(K\times_V L)\big((X,Y),(Z,W)\big) &:= K(X,Z)\otimes_V L(Y,W).
\end{align*}
The unit for this tensor product is the category $1$ which has one
object $*$ and $1(*,*) = I$.

When $V=\Set$ we will use the notation $\Cat$ for $\Cat(V).$ The
underlying category functor provides then a symmetric lax-monoidal
2-functor
\[
\calU: \Cat(V)\to \Cat.
\]

Recall that, for any monoidal $2$-category or, more generally, for a
monoidal bicategory there exists a concept of a pseudomonoid, i.e.~of an
object equipped with a coherently associative multiplication and a
coherent unit which generalizes the notion of a monoidal category
\cite{Lack}.

\begin{definition} 
A {\em monoidal $V$-category\/} is a pseudomonoid in $\Cat(V)$.
\end{definition}

So, the definition is the usual definition of a monoidal category, but
we require the tensor product to be a $V$-functor and the coherence constraint
to be $V$-natural.

\begin{definition} 
A lax-monoidal $V$-functor between monoidal
$V$-categories 
$K = (K,\boxx_K, e_K)$ and $L= (L,\boxx_L,e_L)$ consists of 
\begin{itemize} \item[(i)] a $V$-functor $F:K \to L,$  
\item[(ii)]  
a $V$-natural transformation
$$\phi: F(X)\boxx_L F(Y) \to F(X\boxx_K Y)$$
and a morphism
$$ \phi_e: e_L\to F(e_K)$$ \end{itemize}
which satisfy the usual coherence conditions.
\end{definition}

A lax-monoidal functor is called {\it strong monoidal } if $\phi$ and $\phi_e$ are isomorphisms and it is called {\it strict monoidal } if they are identities.  

Monoidal $V$-categories, lax-monoidal $V$-functors and their monoidal
$V$-transformations form a $2$-category $1\Cat_{lax}(V).$ It is a
monoidal $2$-category with respect to the tensor product~$\times_V.$
Analogously, we have monoidal $2$-subcategories
$$1\Cat_{strict}(V)\subset 1\Cat(V)\subset 1\Cat_{lax}(V)$$ of strict
monoidal and strong monoidal functors.

\subsection{Duoidal $V$-categories}

\begin{definition} 
\label{sec:spn-mono-categ-1}
A {\em duoidal $V$-category} is a pseudomonoid in $1\Cat_{lax}(V)$. 
Explicitly, 
a  duoidal $V$-category is a quintuple $\duo = 
(\duo,\boxx_0,\boxx_1,e,v)$ such that
\begin{itemize}
\item [(i)]
$(\duo,\boxx_0,e)$ and $(\duo,\boxx_1,v)$ are monoidal $V$-categories, equipped with 
\item[(ii)]
 a $V$-natural interchange transformation
\[
(X \boxx_1 Y) \boxx_0 (Z \boxx_1 W) \to 
(X \boxx_0 Z) \boxx_1 (Y \boxx_0 W),
\]
\item[(iii)]
a map 
\begin{equation*}
\label{e}
e \to e\boxx_1 e, 
\end{equation*}
\item[(iv)] 
a map
\begin{equation*}
\label{v}
v\boxx_0 v \to v,
\end{equation*}
\item[(v)] 
and a map 
\[
e\to v.
\]
\end{itemize}
The above data enjoy the coherence properties listed 
in~\cite[Definition~2.1]{AM}, namely the {\em associativity\/} meaning
that the diagrams
\def\diagone#1#2#3#4#5#6{
{
\unitlength=1.000000pt
\begin{picture}(180.00,104.00)(0.00,-10.00)
\thicklines
\put(200.00,30.00){\vector(0,-1){20.00}}
\put(-20.00,30.00){\vector(0,-1){20.00}}
\put(200.00,70.00){\vector(0,-1){20.00}}
\put(-20.00,70.00){\vector(0,-1){20.00}}
\put(73.00,0.00){\vector(1,0){34.00}}
\put(73.00,80.00){\vector(1,0){34.00}}
\put(200.00,0.00){\makebox(0.00,0.00){$#6$}}
\put(200.00,40.00){\makebox(0.00,0.00){$#4$}}
\put(200.00,80.00){\makebox(0.00,0.00){$#2$}}
\put(-20.00,0.00){\makebox(0.00,0.00){$#5$}}
\put(-20.00,40.00){\makebox(0.00,0.00){$#3$}}
\put(-20.00,80.00){\makebox(0.00,0.00){$#1$}}
\end{picture}}
}
\[
\diagone{\big((A \boxx_1 B)\boxx_0 (C\boxx_1D)\big)\boxx_0(E\boxx_1F)}%
        {(A \boxx_1 B)\boxx_0 \big((C\boxx_1D)\boxx_0(E\boxx_1F)\big)}%
        {\big((A \boxx_0 C)\boxx_1 (B\boxx_0D)\big)\boxx_0(E\boxx_1F)}%
        {(A \boxx_1 B)\boxx_0 \big((C\boxx_0E)\boxx_1(D\boxx_0F)\big)}%
{\big((A\boxx_0 C) \boxx_0E\big) \boxx_1 \big((B \boxx_0 D)\boxx_0 F \big)}%
{\big(A\boxx_0 (C \boxx_0E)\big) \boxx_1 \big(B \boxx_0 (D\boxx_0 F )\big)}
\]
\[
\diagone{\big((A\boxx_1 B) \boxx_1C\big)\boxx_0\big((D\boxx_1E)\boxx_1F \big)}
        {\big(A\boxx_1(B\boxx_1C)\big)\boxx_0\big(D\boxx_1(E\boxx_1 F)\big)}
        {\big((A \boxx_1 B)\boxx_0 (D\boxx_1E)\big)\boxx_1(C\boxx_0F)}
        {(A \boxx_0 D)\boxx_1 \big((B\boxx_1C)\boxx_0(E\boxx_1F)\big)}
{\big((A \boxx_0 D)\boxx_1 (B\boxx_0E)\big)\boxx_1(C\boxx_0F)}
        {(A \boxx_0 D)\boxx_1 \big((B\boxx_0E)\boxx_1(C\boxx_0F)\big)}
\]
commute, and the {\em unitality\/} meaning the commutativity of
\def\diagtwo#1#2#3#4{
{
\unitlength=1.000000pt
\begin{picture}(200.00,59.00)(-45.00,-6.00)
\thicklines
\put(100.00,0.00){\makebox(0.00,0.00){$#4$}}
\put(-10.00,0.00){\makebox(0.00,0.00){$#3$}}
\put(100.00,40.00){\makebox(0.00,0.00){$#2$}}
\put(-10.00,40.00){\makebox(0.00,0.00){$#1$}}
\put(10.00,0.00){\vector(1,0){37.00}}
\put(100.00,30.00){\vector(0,-1){20.00}}
\put(-10.00,10.00){\vector(0,1){20.00}}
\put(25.00,40.00){\vector(1,0){22.00}}
\end{picture}}
}
\def\diagtwobis#1#2#3#4{
{
\unitlength=1.000000pt
\begin{picture}(200.00,59.00)(-45.00,-6.00)
\thicklines
\put(100.00,0.00){\makebox(0.00,0.00){$#4$}}
\put(-10.00,0.00){\makebox(0.00,0.00){$#3$}}
\put(100.00,40.00){\makebox(0.00,0.00){$#2$}}
\put(-10.00,40.00){\makebox(0.00,0.00){$#1$}}
\put(10.00,0.00){\vector(1,0){37.00}}
\put(100.00,10.00){\vector(0,1){20.00}}
\put(-10.00,10.00){\vector(0,1){20.00}}
\put(47.00,40.00){\vector(-1,0){22.00}}
\end{picture}}
}
\[
\diagtwo{e \boxx_0(A\boxx_1B)}{(e\boxx_1 e) \boxx_0 (A\boxx_1 B)}%
        {A\boxx_1B}{(e \boxx_0A)\boxx_1(e\boxx_0 B)}
\diagtwo{(A\boxx_1B)\boxx_0e}{(A\boxx_1 B)\boxx_0(e\boxx_1 e) }%
        {A\boxx_1B}{(A \boxx_0e)\boxx_1(B\boxx_0 e)}
\]
\[
\diagtwobis{v \boxx_1(A\boxx_0B)}{(v\boxx_0 v) \boxx_1 (A\boxx_0 B)}%
        {A\boxx_0B}{(v \boxx_1A)\boxx_0(v\boxx_1 B)}
\diagtwobis{(A\boxx_0B)\boxx_1v}{(A\boxx_0 B)\boxx_1(v\boxx_0 v) }%
        {A\boxx_0B}{\hphantom{.}(A \boxx_1v)\boxx_0(B\boxx_1 v).}
\]
In the above diagrams, $A, \ldots,F$ are objects of $\duo$ and the
arrows are induced by the structure operations of $\duo$ in an obvious
way. Moreover, we require the units $e$, $v$ to be {\em compatible\/}
in the sense that $v$ is a monoid in $(\duo,\boxx_0,e)$ and $e$ a
comonoid in $(\duo,\boxx_1,v)$.
\end{definition}

\begin{remark}
Observe that (v) is redundant as the interchange map (ii) with $A=D =
e$ and $B = C =v$ gives exactly (v).
\end{remark}

\begin{definition} A duoidal category $\duo$ is called 
{\em strict\/} if both monoidal categories
$(\duo,\boxx_0,e)$ and $(\duo,\boxx_1,v)$ are strict monoidal categories.
\end{definition}

\begin{example} Pseudomonoids in $1\Cat(V)$ are the same as  braided monoidal $V$-categories \cite{JS}.  Any braided monoidal $V$-category can be considered as a duoidal $V$-category in which two tensor products and two units coincide.   
\end{example}

\begin{example} 
The iterated $2$-monoidal categories of
Balteanu-Fiedorowicz-Schw\"anzl-Vogt \cite{BFSV} are strict duoidal
categories for which $e=v.$
\end{example}

\begin{example} 
Forcey's $2$-fold monoidal categories are duoidal categories for which
(iii) and (iv) are isomorphisms. 
\end{example}

\begin{example} 
If $\duo$ is a duoidal $V$-category, then its underlying category
$\calU(\duo)$ is a duoidal $\Set$-category which we simply call a duoidal
category. The duoidal categories are exactly the $2$-monoidal
categories in the original sense of \cite{AM}.
\end{example}

\begin{definition} 
A {\em lax-duoidal $V$-functor\/} between duoidal $V$-categories   $\duo = 
(\duo,\boxx_0,\boxx_1,e,v)$ and $\duo' = 
(\duo',\boxx'_0,\boxx'_1,e',v')$ consists of 
\begin{itemize} 
\item[(i)] 
a $V$-functor $F:\duo \to \duo',$  
\item[(ii)]  
$V$-natural transformation
$$\phi: F(A)\boxx'_0 F(B) \to F(A\boxx_0 B)$$
and a morphism
$$ \phi_e: e'\to F(e)$$
which makes $F$ a lax-monoidal functor from $
(\duo,\boxx_0,e)$ to $
(\duo',\boxx'_0,e')$,
\item[(iii)]  
a $V$-natural transformation
$$\gamma: F(A)\boxx'_1 F(B) \to F(A\boxx_1 B)$$
and a morphism
$$ \gamma_v: v'\to F(v)$$
which makes $F$ a lax-monoidal functor from $
(\duo,\boxx_1,v)$ to $
(\duo',\boxx'_1,v')$
\end{itemize}
and which enjoy coherence properties from 
\cite[Definition~6.44]{AM}. Namely, we require the commutativity of
the diagrams
\def\diagonemod#1#2#3#4#5#6{
{
\unitlength=1.000000pt
\begin{picture}(180.00,104.00)(-10.00,-10.00)
\thicklines
\put(200.00,30.00){\vector(0,-1){20.00}}
\put(-20.00,30.00){\vector(0,-1){20.00}}
\put(200.00,70.00){\vector(0,-1){20.00}}
\put(-20.00,70.00){\vector(0,-1){20.00}}
\put(53.00,0.00){\vector(1,0){74.00}}
\put(80.00,80.00){\vector(1,0){53.00}}
\put(200.00,0.00){\makebox(0.00,0.00){$#6$}}
\put(200.00,40.00){\makebox(0.00,0.00){$#4$}}
\put(200.00,80.00){\makebox(0.00,0.00){$#2$}}
\put(-20.00,0.00){\makebox(0.00,0.00){$#5$}}
\put(-20.00,40.00){\makebox(0.00,0.00){$#3$}}
\put(-20.00,80.00){\makebox(0.00,0.00){$#1$}}
\end{picture}}
}
\[
\diagonemod{\big((F(A) \boxx_1' F(B)\big)\boxx_0'\big((F(C) \boxx_1'F(D)\big)}%
        {F(A \boxx_1B)\boxx_0' F(C \boxx_1D)}%
        {\big((F(A) \boxx_0' F(C)\big)\boxx_1'\big((F(B) \boxx_0'F(D)\big)}%
        {F\big((A \boxx_1B)\boxx_0 (C \boxx_1D)\big)}%
        {(F(A \boxx_0 C)\boxx_1'(F(B \boxx_0 D)}%
        {F\big((A \boxx_0C)\boxx_1 (B \boxx_0D)\big)}%
\]
\def\diagthree#1#2#3#4#5{
{
\unitlength=1.000000pt
\thicklines
\begin{picture}(175.00,43.00)(-35.00,-6.00)
\put(0.00,30.00){\vector(0,-1){20.00}}
\put(95.00,0.00){\makebox(0.00,0.00){$#5$}}
\put(0.00,0.00){\makebox(0.00,0.00){$#4$}}
\put(95.00,40.00){\makebox(0.00,0.00){$F(#3)$}}
\put(40.00,40.00){\makebox(0.00,0.00){$F(#2)$}}
\put(0.00,40.00){\makebox(0.00,0.00){$#1$}}
\put(20.00,0.00){\vector(1,0){40.00}}
\put(95.00,10.00){\vector(0,1){20.00}}
\put(55.00,40.00){\vector(1,0){15.00}}
\put(8.00,40.00){\vector(1,0){17.00}}
\end{picture}}
}
\def\diagthreebis#1#2#3#4#5{
{
\unitlength=1.000000pt
\thicklines
\begin{picture}(175.00,43.00)(-15.00,-6.00)
\put(0.00,10.00){\vector(0,1){20.00}}
\put(95.00,0.00){\makebox(0.00,0.00){$#5$}}
\put(0.00,0.00){\makebox(0.00,0.00){$#4$}}
\put(95.00,40.00){\makebox(0.00,0.00){$F(#3)$}}
\put(40.00,40.00){\makebox(0.00,0.00){$F(#2)$}}
\put(0.00,40.00){\makebox(0.00,0.00){$#1$}}
\put(20.00,0.00){\vector(1,0){40.00}}
\put(95.00,10.00){\vector(0,1){20.00}}
\put(70.00,40.00){\vector(-1,0){15.00}}
\put(8.00,40.00){\vector(1,0){17.00}}
\end{picture}}
}
\[
\diagthree{e'}e{e\boxx_1e}{e'\boxx_1'e'}{F(e)\boxx_1'F(e)}
\diagthreebis{v'}v{v\boxx_0v}{v'\boxx_0'v'}{F(v)\boxx_0'F(v)}
{
\unitlength=1.000000pt
\begin{picture}(50.00,40.00)(10.00,-6.00)
\thicklines
\put(50.00,0.00){\makebox(0.00,0.00){\hphantom{,}$v'$,}}
\put(0.00,0.00){\makebox(0.00,0.00){$e'$}}
\put(50.00,40.00){\makebox(0.00,0.00){$F(v)$}}
\put(0.00,40.00){\makebox(0.00,0.00){$F(e)$}}
\put(10.00,0.00){\vector(1,0){30.00}}
\put(50.00,10.00){\vector(0,1){20.00}}
\put(0.00,10.00){\vector(0,1){20.00}}
\put(14.00,40.00){\vector(1,0){23.00}}
\end{picture}}
\]
where $A,B,C,D$ are objects of $\duo$ and the meaning of the arrows is clear.

\end{definition}

We call a lax-duoidal functor {\it strong } if
$\phi,\phi_e,\gamma,\gamma_v$ are isomorphisms. A strong duoidal
functor is {\it strict} if these isomorphisms are identities.

\begin{definition}[\cite{AM}, Definition~6.46] 
A duoidal transformation $\phi:F\to G$ between two lax-duoidal
functors is a natural transformation between $F$ and $G$ as functors,
which is a monoidal transformation with respect to two lax-monoidal
structures on $F$ and $G$.
\end{definition}

Duoidal categories, lax-duoidal (strong, strict) duoidal functors and
their duoidal transformations form a $2$-category $2\Cat_{lax}(V)$
($2\Cat(V), 2\Cat_{strict}(V)$).

\subsection{Duoids in duoidal $V$-categories}

The following definition coincides with the definition of a double
monoid given in \cite{AM}.

\begin{definition} 
A {\em duoid\/}  in a duoidal $V$-category $\duo$ is a lax-duoidal $V$-functor
\[
\ttD:1\to \duo.
\]
A {\em morphism\/} between duoids is a duoidal transformation between
corresponding duoidal lax-functors.
\end{definition}
It is easy to see that a duoid $\ttD$ is given by an object $\ttD\in \duo$ together with 
\begin{itemize} 
\item[(i)] a structure of a monoid 
$$ \ttD\boxx_0\ttD\to \ttD \ , \  e\to \ttD$$
with respect to the first monoidal structure, and
\item[(ii)] a structure of a monoid
$$\ttD\boxx_1\ttD \to \ttD \ , \ v\to \ttD$$
with respect to the second monoidal structure.
\end{itemize}
This data should satisfy the  following conditions: 
\begin{itemize}
\item[($\star$)] 
The map $v\to \ttD$ is a monoid morphism with respect to the first
structure and
\item[($\star\star$)] 
the diagram
\begin{center}
{\unitlength=.8mm
\begin{picture}(100,36)(10,0)
\thicklines
\put(40,35){\makebox(0,0){\mbox{$
(\ttD\boxx_1\ttD)\boxx_0 (\ttD\boxx_1\ttD)$}}}
\put(40,30){\vector(0,-1){27}}
\put(62,35){\vector(1,0){15}}
\put(62,-2){\vector(1,0){15}}
\put(88,31){\vector(1,-1){11}}
\put(88,2){\vector(1,1){11}}
\put(85,35){\makebox(0,0){\mbox{$\ttD\boxx_0 \ttD$}}}
\put(102,16.5){\makebox(0,0){\mbox{$\ttD$}}}
\put(40,-2){\makebox(0,0){\mbox{$ (\ttD\boxx_0\ttD)\boxx_1 (\ttD\boxx_0\ttD)$}}}
\put(85,-2){\makebox(0,0){\mbox{$ \ttD\boxx_1 \ttD$}}}
\end{picture}
}
\end{center}
\end{itemize}
\noindent commutes.

\begin{example} 
If $\duo$ is a braided monoidal category then a duoid in $\duo$ is the
same as a commutative monoid.
\end{example}

\begin{example} 
The second unit $v\in \duo$ is a duoid in $\duo$ with the first monoid
structure given by the canonical morphism $v\boxx_0 v\to v$ and with the
second monoid structure given by the canonical isomorphism $v\boxx_1
v\to v.$
\end{example}

\begin{example} 
Duoids in the duoidal category $\funny(\cat,\calK)$ (see Section
\ref{sec:spn-categories}) are $\calK$-enriched $2$-categories whose
$1$-truncation is the category $\cat.$
\end{example}

Any lax-duoidal $V$-functor $F:\duo\to\duo'$ maps duoids in $\duo$ to
duoids in $\duo'$.  If $\ttD$ is a duoid in a duoidal $V$-category
$\duo$ then $\ttD$ is also a duoid in the underlying duoidal category
$\calU(\duo)$. We will use the notation $u(\ttD)$ for this duoid and will
call it the {\em underlying duoid\/} of $\ttD$.

\subsection{Coherence for duoidal $V$-categories.}

Duoidal $V$-categories are algebras of a $2$-monad on  the $2$-category $\Cat(V).$ The forgetful $2$-functor
$ 2\Cat(V) \to \Cat(V)$ reflects equivalences by \cite{BKP}, hence, the following definition makes sense:

\begin{definition}  
A strong duoidal $V$-functor $F$ is a {\em duoidal equivalence\/} if
it is a $V$-equivalence of the underlying $V$-categories.  Two duoidal
$V$-categories are called {\em duoidal equivalent\/} if there is a
duoidal equivalence between them.
\end{definition}

\begin{theorem}
\label{semistrict} 
Every duoidal $V$-category is duoidal equivalent to a strict duoidal $V$-category.
\end{theorem}

Before we prove the theorem, we introduce the following auxiliary terminology.
Let us call a {\it double monoidal $V$-category} a category $\cat$
equipped with two monoidal structures $(\cat,\boxx_0,e)$ and
$(\cat,\boxx_1,v)$ without any relations between two structures.  We
call a double monoidal $V$-category {\it strict} if both monoidal
structures are strict.  Likewise, we have a notion of double monoidal
strong functor and equivalence between double monoidal $V$-categories.

\begin{lemma}
\label{doublemonoidal} Any double monoidal $V$-category is equivalent
to a strict  double monoidal $V$-category.

\end{lemma}
\begin{proof} 
Any strict monoidal $V$-category is an algebra of the nonsymmetric
operad $M = \{M(n)\}_{n \geq 0}$ in $\Cat(V)$ such that $M(n)$ is, for
each $n$, the terminal category. A monoidal $V$-category is an algebra
of another operad $M_c$ in $\Cat(V).$ There is an operadic map
$\pi:M_c\to M$ recalled below which is an adjoint $V$-equivalence in
each arity (operadic weak equivalence). One can prove that $M_c$ is a
cofibrant resolution of $M$ in the category of nonsymmetric operads in
$\Cat(V)$ equipped with a model structure developed in \cite{Muro} and
\cite{LackV}. Here we consider the model structure on $V$ for which weak
equivalences are isomorphisms. These data allow to prove the coherence
result for monoidal $V$-categories (using bar-construction, for
example).
 
More generally, one can prove by the same method that given a weak
equivalence $\xi:A\to B$ of $\Cat(V)$-operads, every $A$-algebra is
equivalent to an algebra of the form $\xi^*(X)$, where $\xi^*$ is the
restriction functor induced by $\xi$.

\begin{remark}
One can prove that the adjunction between categories of algebras
induced by $\xi$ is in fact a Quillen equivalence.
\end{remark}

Observe now that a
double monoidal $V$-category is an algebra of the operad $M_c\coprod
M_c$ and a strict double monoidal $V$-category is an algebra of
$M\coprod M.$ Therefore, the lemma will be proved if we establish that
coproduct $\pi\coprod\pi$ is a weak equivalence of operads. This
follows from the following explicit description of this coproduct.
  
{\it A bicolored binary planar tree} is a planar tree whose vertices
have valencies three or one and have two colors white and black. A
planar tree $l$ without vertices (and, therefore, without coloring) is
also considered as a binary bicolored tree.  Let $\BTree$ be the set
of isomorphism classes of bicolored binary trees with $n$-leaves.
The sequence $\BTree := \{\BTree(n)\}_{n \geq 0}$ is an operad in
$Set$ -- the free operad on two $0$-operations and two binary
operations. A subtree $S$ of a bicolored binary tree $T$ is called
{\em monocolored\/} if all its vertices have the same colors. Any monocolored
subtree belongs to a unique maximal monocolored subtree.
 
{\it An alternating bicolored planar tree} is a planar tree whose
vertices have valencies one or greater or equal than three and have
two colors -- white and black. It also must satisfy the following
condition: there is no edge connecting two vertices of the same color.
The tree $l$ is also considered as an alternating bicolored tree. Leaves
and roots of an alternating bicolored tree inherit the color by the
following rule: a leaf (a root) has white (black) color if the unique
vertex to which the leaf (the root) is attached has white (black)
color.

Let $\ATree(n)$ be the set of isomorphism classes of 
alternating bicolored trees with
$n$-leaves. The sequence $\ATree := \{\ATree(n)\}_{n \geq 0}$ forms 
an operad. The operadic
multiplication is given by grafting if we graft a tree to a leaf of
another tree and the color of this leaf is different from the color
of the root. In the case the colors coincide, we graft and contract the
edge which has the endpoints of the same color. The unit of this
operad is $l.$
 
There is an obvious operadic map $F:\BTree \to \ATree$. For a bicolored
binary tree $T$, the tree $F(T)$ is obtained by contracting all maximal
monocolored subtrees of $T$ to corollas and preserving the colors.
 
We make $\ATree$ a $\Cat(V)$-operad by considering $\ATree(n)$ as a
discrete $V$-category.  Likewise, we make $\BTree$ a $\Cat(V)$-operad
by requiring that we have a unique isomorphism between two bicolored
binary trees if and only if their imaged under $F$ coincide.  It is
easy to see that $F$ is indeed a weak equivalence of
$\Cat(V)$-operads. Indeed, one can easily check that $M_c\coprod M_c \simeq
\BTree$ and $M\coprod M\simeq \ATree$ by considering generators and
relations in these operads. Moreover, $F\simeq \pi\coprod\pi$, which
completes the proof.
\end{proof}

\begin{proof}[Proof of Theorem~\ref{semistrict}.]
Let $\duo$ be a duoidal category. It has an underlying double
monoidal category $D$. By Lemma~\ref{doublemonoidal}, one can find
a strict double monoidal category $D'$ and a double monoidal
equivalence $F:D \to D'.$ Using this equivalence one can transport the
duoidal structure from $\duo$ to $D'$ without altering the tensor
products and units in $D'.$ In this way we obtain a strict duoidal
category $\duo'$ and $F$ is lifted to a duoidal equivalence $F':\duo
\to \duo'$.
\end{proof}

\begin{remark}Our proof of Lemma~\ref{semistrict} was based on the fact that the
coproduct of two weak equivalences of $\Cat(V)$-operads is again a weak
equivalence. This statement is a `non-abelian' version of the
K\"unneth formula for augmented dg-operads proved as Theorem~21
of~\cite{markl:ha}.  
\end{remark}

\section{$\duo$-categories and monoidal $\duo$-categories }
\label{sec:duo-categ-mono}

If $\duo$ is a duoidal category we will denote $\Cat(\duo)$ the
$2$-category of $(\duo,\boxx_0,e)$-enriched categories. It was observed
by Forcey \cite{Forcey} that $\Cat(\duo)$ can be equipped with a
monoidal structure. The tensor product $\times_1$ of two
$\duo$-categories $\calK$ and $\calL$ is given by the cartesian product
on the objects level and
\[
(\calK\times_1\calL)((X,Y),(Z,W)) = \calK(X,Z)\boxx_1 \calL(Y,W),
\mbox { for $a,c \in S$, $b,d \in P$.}
\]
 The unit for this tensor product is the category $\mathbf 1_v$ which has one object $*$ and 
 ${\mathbf 1}_v(*,*) = v.$ 
 
 \begin{definition}\label{mono}
 A monoidal $\duo$-category $\calK = (\calK,\odot,\eta)$ is a pseudomonoid in the monoidal $2$-category $(\Cat(\duo),\times_1,{\mathbf 1_v}).$ 
\end{definition} 

So we have a $\duo$-functor $\odot : \calK \times_1 \calK \to \calK$
fulfilling the expected associativity up to a $\duo$-natural
transformation, and a $\duo$-functor $\eta: {\mathbf 1_v} \to \calK$.
By abusing notations we will denote $\eta$ the value of $\eta$ on the
unique object of $\mathbf 1_v.$

A pseudomonoid  structure therefore implies the existence of a monoid morphism 
\begin{equation}
\label{vaction}
v\to \calK(\eta,\eta)
\end{equation} 
and interchange morphisms
 $$\calK(X,Y)\boxx_1\calK(Z,W)\to \calK(X\odot Z,Y\odot W)$$
 satisfying various coherence conditions.
 
Every $\duo$-category $\calK$ has an underlying $V$-category $\Und\calK$,
with the same objects and morphisms given by 
\[
\Und\calK(X,Y) = \duo(e,\calK(X,Y)),\ X,Y \in \calK.
\] 
This gives a $2$-functor 
$$\Und:\Cat(\duo)\rightarrow \Cat(V).$$
This is actually a lax-monoidal $2$-functor. To see this we have to specify a transformation
\[
\Und\calK\times_V \Und\calL \to \Und(\calK\times_1\calL).
\] 
On the object level this is an identity and on the morphisms level we have 
$$(\Und\calK\times_V \Und\calL)((X,Y),(Z,W)) = \Und\calK(X,Z)\otimes_V \Und\calL(Y,W) =$$ $$=  \duo(e,\calK(X,Z))\otimes_V \duo(e,\calL(Y,W)) 
\to \duo(e\boxx_1 e, \calK(X,Z)\boxx_1 \calL(Y,W))\to $$ $$ \to
\duo (e,  (\calK\times_1\calL)((X,Y),(Z, W)) =
\Und(\calK\times_1\calL)((X,Y),(Z,W)).$$

In this calculation we used the fact that $\boxx_1$ is a $V$-functor
and that $e$ is a comonoid with respect to $\boxx_1.$ The unit
constraint $$1\to { \mathbf 1_v}$$ amounts to a morphism $I_V\to
\duo(e,v)$ which corresponds to the canonical morphism $e\to v$ in
$\duo.$ We leave the verification of coherence conditions to the
reader.

\begin{proposition}
\label{under} 
The $2$-functor $\Und$ maps monoidal $\duo$-categories to monoidal
$V$-categories.

\end{proposition}

\begin{proof}
This is a direct consequence of lax-monoidality of $\Und.$
\end{proof}

\begin{definition} 
A {\em lax-monoidal functor\/} $F$ from a monoidal $\duo$-category
$(\calK,\odot,\eta)$ to a monoidal $\duo$-category
$(\calL,\diamond,\iota)$ is a $\duo$-functor $F:\calK\to\calL$
equipped with a $\duo$-natural transformation:
\[
F(X)\diamond F(Y)\to F(X\odot Y)
\]
and a morphism
\[
\iota\to F(\eta)
\]
which make $F$ a lax-monoidal functor between underlying monoidal
$V$-categories and which satisfy the following additional coherence
condition:

\begin{center}
{\unitlength=.9mm
\begin{picture}(100,60)(33,-12)
\put(115,30){\makebox(0,0){\makebox(0,0)
{$\calL(FX\diamond FZ,F(X\odot Z))\boxx_0 \calL(F(X\odot Z),F(Y\odot W))$}}}
\put(115,15){\makebox(0,0){\makebox(0,0){$\calL(FX\diamond FZ,F(Y\odot W))$}}}
\put(115,0){\makebox(0,0){\makebox(0,0){$
\calL(FX\diamond FZ,FY\diamond FW)\boxx_0 \calL(FY\diamond FW,F(Y\odot W))$}}}
\put(72.5,45){\makebox(0,0){\makebox(0,0){$ e\boxx_0 \calK(X\odot Z,Y\odot W)$}}}
\put(72.5,-15){\makebox(0,0)
{\makebox(0,0){$(\calL(FX,FY)\boxx_1\calL(FZ,FW))\boxx_0e$}}}
\put(30,30){\makebox(0,0){\makebox(0,0){$\calK(X\odot Z,Y\odot W)$}}}
\put(30,15){\makebox(0,0){\makebox(0,0){$\calK(X,Y)\boxx_1\calK(Z,W)$}}}
\put(30,0){\makebox(0,0){\makebox(0,0){$\calL(FX,FY)\boxx_1\calL(FZ,FW)$}}}
\thicklines
\put(115,4){\vector(0,1){8}}
\put(30,12){\vector(0,-1){8}}
\put(115,27){\vector(0,-1){8}}
\put(30,19){\vector(0,1){8}}
\put(95,41.5){\vector(2,-1){16}}
\put(33,33.5){\vector(2,1){18}}
\put(95,-11){\vector(2,1){16}}
\put(33,-3){\vector(2,-1){16}}
\end{picture}}
\end{center}
\end{definition}

As usual, we call a lax-monoidal $\duo$-functor {\em strong\/} 
({\em strict}) if its coherence constrains are isomorphisms (identities). 
 
\begin{definition} 
A {\em monoidal $\duo$-transformation\/} between two lax-monoidal
$\duo$-functors is a $\duo$-natural transformation which is a monoidal
transformation between their underlying lax-monoidal $V$-functors.
\end{definition}
 
Monoidal $\duo$-categories, lax-monoidal (strong, strict)
$\duo$-functors and their monoidal $\duo$-transformations form a
$2$-category $1\Cat_{lax}(\duo)$ ($1\Cat(\duo),1\Cat_{strict}(\duo)$).
Every lax-monoidal $\duo$-functor between monoidal $\duo$-categories
induces a lax-monoidal $V$-functor between the underlying monoidal
$V$-categories. The same is true for $\duo$-transformations.
 
\begin{remark}  
The $2$-category $1\Cat(\duo)$ is not a monoidal $2$-category. To make
it monoidal, we need one more tensor product on $\duo$ which would make
it a trioidal category. If this is the case, the underlying
$V$-category functor
\[
\Und:1\Cat(\duo)\to 1\Cat(V)
\] 
would be even a monoidal $2$-functor.
\end{remark}

\begin{theorem}\label{Kcoherence}  Every monoidal $\duo$-category is equivalent in $1\Cat(\duo)$ to a strict monoidal $\duo$-category.
\end{theorem}

\begin{proof} 
This follows from a general theorem of S.~Lack about strictification
of pseudomonoids in a Gray-monoid \cite{Lack}. In our situation, the
monoidal $2$-category $(\Cat(\duo),\times_1,1_v)$ is not a~Gray-monoid
but can be replaced by an equivalent Gray-monoid due to the
tricategorical coherence theorem of Gordon-Power-Street \cite{GPS}.
\end{proof}

Due to this coherence theorem we will assume that all objects of
$1\Cat_{lax}(\duo)$, $1\Cat(\duo)$ and $1\Cat_{strict}(\duo)$ are strict
monoidal $\duo$-categories.

\begin{definition} 
A $\duo$-enriched $2$-category is a category enriched over
$(\Cat(\duo),\times_1,1_v)$.
\end{definition} 

\begin{remark}\label{loop} 
As in the classical situation, we can identify a strict monoidal
$\duo$-category with a $\duo$-enriched $2$-category with one object.
On the other hand, if $\calC$ is a $\duo$-enriched $2$-category and
$x$ is an object of $\calC$, then $\calC(x,x)\in \Cat(\duo)$ is a
monoidal $\duo$-category.
\end{remark}

\subsection{ Monoids in monoidal $\duo$-categories} 
Let $(\calK,\odot,\eta)$ be a monoidal $\duo$-category. The category
${\mathbf 1_v}$ is canonically a monoidal $\duo$-category (the tensor
product is given by the structure isomorphism $v\boxx_1v\to v $).

\begin{definition}
\label{na-plaz} 
A {\em monoid in $\calK$\/} is a lax-monoidal $\duo$-functor
\[
\ttM:{\mathbf 1_v}\to \calK.
\]
A morphism $f:\ttM\to\ttN$ of monoids is a monoidal transformation.
More explicitly, a monoid in $\calK$ is given by an object
$\ttM\in\calK$ together with:
\begin{itemize}
\item[(i)] 
a morphism (neutral element) in $\calK$
$$i:\eta \to  \ttM,$$
which is, by definition,
a morphism $\nubar:e\to \calK(\eta,\ttM)$ in $\duo$,
\item[(ii)] 
a morphism in $\calK$ (multiplication)
$$m: \ttM\odot \ttM \to \ttM$$
that is, a morphism
 $\mubar:e\to \calK(\ttM\odot \ttM,\ttM)$ in $\duo$ and
\item[(iii)] 
a morphism in $\duo$ (the {\em unit\/})
 $$u:v\to \calK(\ttM,\ttM).$$ 
\end{itemize}
This last unusual piece of 
data comes from the requirements that $\ttM$ is a $\duo$-functor.
These data should satisfy the following axioms:
\begin{itemize}
\item[($\star$)] 
$i$ and $m$ make $\ttM$ a monoid in $\Und\calK$,
\item[($\star\star$)] 
$u$ is a monoid morphism in $\duo$, and
\item[($\star\!\star\!\star$)] 
the following diagram commutes:

\begin{center}
\begin{picture}(300,123)(50,-8)
\thicklines
\unitlength=1mm
\put(15,35){\makebox(0,0){\mbox{$v$}}}
\put(40,35){\makebox(0,0){\mbox{$e\boxx_0 v$}}}
\put(1,20){\vector(1,1){12}}
\put(0,14){\vector(1,-2){6}}
\put(19,35){\vector(1,0){12}}
\put(48,35){\vector(1,0){34}}
\put(77,-2){\vector(1,0){10}}
\put(24,-2){\vector(1,0){8}}
\put(108,32){\vector(1,-1){12}}
\put(108,2){\vector(1,1){12}}
\put(105,35){\makebox(0,0){\mbox{$\calK(\ttM\odot\ttM,\ttM)\boxx_0\calK(\ttM,\ttM)$}}}
\put(0,17){\makebox(0,0){\mbox{$v\boxx_1v$}}}
\put(124,17){\makebox(0,0){\mbox{$\calK(\ttM\odot \ttM,\ttM)$}}}
\put(5,-2){\makebox(0,0){\mbox{$\calK(\ttM,\ttM)\boxx_1\calK(\ttM,\ttM)$}}}
\put(55,-2){\makebox(0,0){\mbox{$(\calK(\ttM,\ttM)\boxx_1\calK(\ttM,\ttM))\boxx_0 e$}}}
\put(115,-2){\makebox(0,0){\mbox{$\calK(\ttM\odot \ttM,\ttM\odot \ttM)\boxx_0 \calK(\ttM\odot \ttM,\ttM)$}}}
\put(25,36){\makebox(0,0)[b]{\mbox{\scriptsize $\cong$}}}
\put(65,36){\makebox(0,0)[b]{\mbox{\scriptsize $\overline{\mu}\boxx_0 u$}}}
\put(5,8){\makebox(0,0)[lb]{\mbox{\scriptsize $u \boxx_1 u$}}}
\put(81,0){\makebox(0,0)[b]{\mbox{\scriptsize $\odot\boxx_0 u$}}}
\end{picture}
\end{center}
\end{itemize}
A {\em monoid morphism\/} is a morphism $f: \ttM\to \ttN$ in $\calK$
(i.e.~a morphism $\bar{\phi}: e\to \calK(\ttM,\ttN)$ in~$\duo$) which
satisfies the usual requirements for monoids morphism, and

\begin{itemize}
\item[($\star\!\star\!\star \hskip 0.05em \star$)]  
the following diagram commutes:
\begin{center}
\begin{picture}(300,100)(50,-8)
\thicklines
\unitlength=.8mm
\put(40,35){\makebox(0,0){\mbox{$v\boxx_0 e$}}}
\put(40,20){\vector(0,1){12}}
\put(40,14){\vector(0,-1){12}}
\put(48,35){\vector(1,0){34}}
\put(48,-2){\vector(1,0){34}}
\put(108,32){\vector(1,-1){11}}
\put(108,2){\vector(1,1){11}}
\put(103,35){\makebox(0,0){\mbox{$\calK(\ttM,\ttM)\boxx_0\calK(\ttM,\ttN)$}}}
\put(40,17){\makebox(0,0){\mbox{$v$}}}
\put(124,17){\makebox(0,0){\mbox{$\calK( \ttM,\ttN)$.}}}
\put(40,-2){\makebox(0,0){\mbox{$e \boxx_0 v$}}}
\put(103,-2){\makebox(0,0){\mbox{$\calK(\ttM,\ttN)\boxx_0 \calK(\ttN,\ttN)$}}}
\put(35,25){\makebox(0,0)[b]{\mbox{\scriptsize $\cong$}}}
\put(35,6){\makebox(0,0)[b]{\mbox{\scriptsize $\cong$}}}
\end{picture}
\end{center}
\end{itemize}
Monoids and their morphisms form the category $\Mon(\calK).$
\end{definition}

\subsection{$\calK$-enriched categories.}
Classically, a monoid in a monoidal category $\calC$ is the same as a one
object $\calC$-enriched category.  We now introduce
$\calK$-enriched categories in a way that this property is preserved.

\begin{definition} 
A {\em $\calK$-enriched category\/} $\ttM$ consists of a set of objects
$\ttM_0$ and, for each two objects $x,y\in \ttM_0$, an object
$\ttM(x,y)\in \calK.$ The structure morphisms are:
\begin{itemize} 
\item[(i)] 
for each object $x\in \ttM_0$, a morphism in $\calK$
\[
i(x): \eta \to \ttM(x,x),
\]
\item[(ii)] 
for any $x,y,z\in \ttM_0$ a morphism
\[
m(x,y,z):\ttM(x,y)\odot \ttM(y,z)\to \ttM(x,z),
\]
\item[(iii)]
and, for any two objects $x,y\in \ttM_0$, a morphism in $\duo$
\[
u(x,y):v\to \calK(\ttM(x,y),\ttM(x,y)).
\] 
\end{itemize}
These data satisfy the obvious analogs of the axioms for monoids
where, in the monoid  coherence condition  ($\star\!\star\!\star$), 
we have to replace $\calK(\ttM,\ttM)\boxx_1\calK(\ttM,\ttM)$ by 
\[
\calK(\ttM(x,y),\ttM(x,y))\boxx_1\calK(\ttM(y,z),\ttM(y,z)).
\]
The rest of the diagram is clear. 
\end{definition}

Analogously, one can define {\em $\calK$-functors\/} 
and {\em $\calK$-natural transformations.\/}
So, a $\calK$-functor $F:\ttM\rightarrow \ttN$ 
is given by a map of objects and an effect in $\duo$ on 
morphisms expressed as a~structure morphism
\begin{equation}
\label{etov}f_e(x,y): e \to \calK(\ttM(x,y),\ttN(F(x),F(y)), \ x,y \in \ttM_0,
\end{equation}
satisfying the usual conditions 
and an obvious analogue of the extra coherence diagram 
(\hbox{$\star\!\star\!\star \hskip 0.05em \star$}) in which 
we have to replace $\ttM$ by $\ttM(x,y)$ and $\ttN$ by $\ttN(F(x),F(y))$.

\begin{remark}\label{replace} 
We  can replace  the structure morphism (\ref{etov}) by the morphism  
\begin{equation}
\label{vtov} f_v(x,y): v \to \calK(\ttM(x,y),\ttN(F(x),F(y)), \ x,y \in \ttM_0,
\end{equation}
defined  as the composite
\begin{align*}
v\simeq e\boxx_0 v \stackrel{f_e\boxx_0 u}{\vlra}& \hskip .5em
\calK\big(\ttM(x,y),\ttN(F(x),F(y)\big)\boxx_0
\calK\big(\ttN(F(x),F(y)),\ttN(F(x),F(y)\big)\longrightarrow
\\
&\hskip 1em \longrightarrow
\calK\big(\ttM(x,y),\ttN(F(x),F(y)\big).
\end{align*}
We can reconstruct $f_e$ from $f_v$ by precomposing with $e\to v.$ 
\end{remark}

It is not difficult to check that $\calK$-enriched categories,
their $\calK$-functors and $\calK$-natural transformations form a
$2$-category which we will denote $\Cat(\calK)$. By abusing the
notation, we will also denote by $\Cat(\calK)$ the $1$-truncation of
$\Cat(\calK)$, i.e.~the ordinary category of $\calK$-categories and
$\calK$-functors when it does not lead to confusion.
  
With the definitions above we can identify a 
monoid in $\calK$ with a $\calK$-category with one object. 
This identification gives a functor 
\[
\Sigma: \Mon(\calK)  \to \Cat(\calK).
\]

\section{Operads in duoidal categories}

It is customary and convenient in the classical operad theory to assume
that the base symmetric monoidal category is strict.  This is possible
due to MacLane coherence theorem.  We follow this tradition and assume
that our base duoidal category $\duo$ is strict.
Theorem~\ref{semistrict} justifies this assumption.

The notion of a $2$-fold operad in a $2$-fold monoidal category was
introduced by Forcey, Siehler and Sowers in~\cite{Forcey2}.  Our
notion of a duoidal category is weaker than Forcey's $2$-fold monoidal
category, so we need a slight modification of their definition.

\begin{definition}
\label{Forcey-operad}
A collection $A = \{A(n)\}_{n \geq 0}$ of objects of $\duo$ is a
{\it Forcey $1$-operad} 
if, for each integers $n \geq 1$, $\Rada k1n
\geq 0$, one is given a morphism
\begin{equation}
\label{eq:5}
\gamma : \ll A(k_1) \boxx_1 \cdots \boxx_1 A(k_n)\rr \boxx_0 A(n) \to
A(k_1 + \cdots + k_n),
\end{equation}
fulfilling
the obvious version of the associativity for a nonsymmetric
operad. One also requires
a $\duo$-map $j : e \to A(1)$ (the {\em unit\/}) 
such that the diagrams
\[
\hskip 4em
{\unitlength=1.1mm
\begin{picture}(40,8)(10,15)
\thicklines
\put(13,25){\makebox(0,0){\mbox{$(\boxx_1^k e) \boxx_0 A(k)$}}}
\put(13,21){\vector(0,-1){10}}
\put(37.5,25){\vector(-1,0){12}}
\put(46,25){\makebox(0,0){\mbox{$e\boxx_0 A(k)$}}}
\put(45,21){\vector(0,-1){10}}
\put(13,7){\makebox(0,0){\mbox{$(\boxx_1^k A(1)) \boxx_0 A(k)$}}}
\put(28,7){\vector(1,0){13}}
\put(46,7){\makebox(0,0){\mbox{$A(k)$}}}
\put(12,16){\makebox(0,0)[r]{\mbox{\scriptsize $\boxx^k_1 j\boxx_0\id$}}}
\put(31,8.5){\makebox(0,0)[b]{\mbox{\scriptsize $\gamma$}}}
\end{picture}}
\hskip 2em \mbox { and }
%
\jarca{A(k) \boxx_0 e}{A(k)}{A(k) \boxx_0 A(1)}{\id \boxx_0j}
\]
commute for each $k \geq 0$. 
{\em Morphisms\/} of  Forcey $1$-operads are
morphism of the underlying collections compatible with all
structure operations.
\end{definition}

We, therefore, have a category of Forcey $1$-operads in $\duo.$ 
Every Forcey $1$-operad determines a right action of $A$ on $A(0)$ in
the sense that there are morphisms $(\boxx_1^k A(0))\boxx_0 A(k)
\rightarrow A(0), k\ge 1$, which satisfy the usual conditions for operad
action.

\begin{definition}
\label{1-operad} 
A {\em $1$-operad\/} in a duoidal category $\duo$ is a Forcey
$1$-operad in $\duo$ equipped with a left $v$-module structure $v
\boxx_0 A(0) \to A(0)$ with respect to $\boxx_0$ on $A(0)$ such that
it makes $A(0)$ a $(v,A)$-bimodule in the sense that the following
diagram commutes:

{\unitlength=1.5mm
\begin{picture}(40,30)(-40,2)
\thicklines
\put(0,26){\makebox(0,0)[b]{\scriptsize $\cong$}}
\put(17,25){\makebox(0,0){$v\boxx_0 (\boxx_1^k A(0)) \boxx_0 A(k)$}}
\put(-20,25){\makebox(0,0){{$(\boxx_1^kv)\boxx_0(\boxx_1^kA(0))\boxx_0A(k)$}}}
\put(-20,7){\makebox(0,0){{$\big(\boxx_1^k(v\boxx_0 A(0))\big) \boxx_0 A(k)$}}}
\put(-20,21){\vector(0,-1){10}}
\put(31,25){\vector(1,0){8}}
\put(-3,25){\vector(1,0){6}}
\put(-5,7){\vector(1,0){8}}
\put(46,25){\makebox(0,0){\mbox{$v\boxx_0 A(0)$}}}
\put(45,21){\vector(0,-1){10}}
\put(15,7){\makebox(0,0){\mbox{$(\boxx_1^k A(0)) \boxx_0 A(k)$}}}
\put(26,7){\vector(1,0){15}}
\put(45,7){\makebox(0,0){\mbox{\hphantom{.}$A(0)$.}}}
\end{picture}}
\end{definition}

\begin{remark} 
As it is clear from these definitions, $1$-operads and Forcey
$1$-operads differ only in the treatment of composition~(\ref{eq:5}) for
$n=0$. If we agree that $\boxx_1^0 = v$ and add the $n=0$ case of
$\gamma$ to Definition~\ref{Forcey-operad}, we will obtain exactly 
Definition~\ref{1-operad}.
\end{remark}

\begin{example}\label{Ass} 
The associativity $1$-operad $\fAss$ in $\duo$
is given by $\fAss(n) = v$, with the unit and multiplication
given by canonical morphisms $e\rightarrow v$ and $v\boxx_0 v
\rightarrow v.$ \end{example}

\begin{example}
\label{weak}
The Forcey $1$-operad $\eAss$ is given by $\eAss(n) = e$.
The unit is obvious. The multiplication is defined as the composition
\begin{align*}
(\underbrace{e \boxx_1\ldots\boxx_1e}_n)& \boxx_0 e
\simeq(\underbrace{e \boxx_1\ldots\boxx_1e}_n))\boxx_0 ((\underbrace{v
\boxx_1\ldots\boxx_1v}_{n-1}\boxx_1 e) \rightarrow
\\
&\rightarrow \underbrace{(e\boxx_0 v)\boxx_1\ldots \boxx_1 (e\boxx_0
v)}_{n-1} \boxx_1 (e\boxx_0 e) \simeq \underbrace{v\boxx_1\ldots
\boxx_1 v}_{n-1} \boxx_1 e \simeq e. 
\end{align*}
For $n=1$ we have just the canonical isomorphism
$$e\boxx_0 e \to e.$$

\end{example}

\begin{example}
If $\duo$ is a braided monoidal category, then a $1$-operad in $\duo$
is a classical nonsymmetric operad. In this case there is no
difference between Forcey operads and $1$-operads.
\end{example}

\subsection{Endomorphism operads and algebras of operads}\label{algebras} 

Let $(\calK,\odot,\eta)$ be a (strict) monoidal $\duo$-category.

\begin{definition} 
\label{sec:endom-oper-algebr}
The {\em endomorphism $1$-operad\/} of an object $X\in \calK$ is
determined by:
\[
\End_X(n) = {\calK}(\odot^n X,X)
\]
with the obvious multiplication and unit data. 
The only unusual data
is the action of $v$ on $\End_X(0) = {\calK}(\eta,X)$ given
by the composition of the structure maps
\[
v \boxx_0 \End_X(0) \to \calK(\eta,\eta) \boxx_0 \calK(\eta,X) \to
\calK(\eta,X) = \End_X(0).
\]
Since the monoid ${\calK}(\eta,\eta)$ acts on ${\calK}(\eta,X)$, we
have also an action of $v$ via the morphism of monoids
(\ref{vaction}).
\end{definition}

\begin{definition} An {\em algebra\/} of
a $1$-operad $A$ is
an object $X$ of a duoidal category $\calK$ equipped  with a $1$-operad morphism $k: A \to \End_X$.
\end{definition}

We define now the $V$-enriched category of algebras for a $1$-operad $A$. 
To shorten the notations, we denote 
for two objects $X,Y\in \calK$ and $n \geq 0$,
\[
\EE_X(n) = \End_X(n), \ \EE_Y(n) = \End_Y(n), \ \EE_{X,Y}(n) = 
{\calK}(\odot^n X,Y).
\]
Let now $X,Y$ be two $A$-algebras and let 
$k_X^n: A(n) \rightarrow \EE_X(n)$, $k_Y^n : A(n) 
\rightarrow \EE_Y(n)$ be components of their structure morphisms.
By definition, they are  morphisms 
\[
k_X^n: I_V\rightarrow 
{\duo}(A(n), \EE_X(n)), \ k_Y^n: I_V\rightarrow 
{\duo}(A(n) ,\EE_Y(n))
\]
in $V$. For any $n\ge 0$, we have obvious actions
\[
\EE_X(n)\boxx_0 \EE_{X,Y}(1) \stackrel{a_0}\longrightarrow
\EE_{X,Y}(n) \stackrel{a_1}\longleftarrow \EE_{X,Y}(1)\boxx_0 \EE_Y(n).
\]
Now we define, using $a_0$, for each $n$ a morphism $d_n^0 :
{\duo}(e,\EE_{X,Y}(1)) \rightarrow {\duo}(A(n),
\EE_{X,Y}(n))$ in~$V$
\begin{gather*}
{\duo}(e,\EE_{X,Y}(1)) \rightarrow I_V\otimes_V
{\duo}(e,\EE_{X,Y}(1)) \rightarrow {\duo}(A(n),
\EE_X(n))\otimes_V {\duo}(e,\EE_{X,Y}(1)) \rightarrow
\\
\rightarrow {\duo}(A(n)\boxx_0 e, \EE_X(n))\boxx_0 \EE_{X,Y}(1))
\rightarrow {\duo}(A(n), \EE_{X,Y}(n)).
\end{gather*}
Similarly we define $d_n^1 : {\duo}(e,\EE_{X,Y}(1)) \rightarrow
{\duo}(A(n), \EE_{X,Y}(n))$ using $a_1$. Finally, 
we define the $V$-enriched Hom
from $X$ to $Y$ as the equalizer of the products of $d_n^0$ and
$d_n^1$,
$$
{\duo}(e,\EE_{X,Y}(1))
\eqv
\prod_{n\geq 0}{\duo}(A(n), \EE_{X,Y}(n)).
$$
It is easy to see that the above construction defines a $V$-enriched category   of
algebras of~$A$. As usual, the category of algebras of $A$ is just the
underlying category  of this $V$-category.
Analogously one can define $V$-category of algebras of any Forcey operad.

\begin{proposition}
\label{sec:spn-mono-categ-2}The category of algebras of the Forcey $1$-operad $\eAss$ is isomorphic to the category of 
 monoids in the underlying category $\Und\calK.$  The category of algebras of the
$1$-operad $\fAss$ is isomorphic to the category of monoids
 in $\calK.$ 
 
\end{proposition}

\begin{proof}

The first statement of the proposition is classical. Let us prove the
second statement.  Let $\ttM$ be an algebra of $\fAss.$ It is obvious
that the structure algebra map $\fAss \to \End_{\ttM}$ is given by the
following three maps in $\duo$: \hfill\break \hphantom{iiiii}(i)~the
`neutral element' $\nu : v \to \calK(\eta,\ttM)$,\hfill\break
\hphantom{iiii}(ii)~the `unit' $u : v \to \calK(\ttM,\ttM)$, and
\hfill\break \hphantom{iii}(iii)~the `multiplication' $\mu : v \to
\calK(\ttM \odot \ttM,\ttM)$, \hfill\break such that the
diagrams~(\ref{d1})--(\ref{d5}) of maps in $\duo$ below are
commutative.

\begin{subequations}
\begin{equation}
\posunm{2}
\label{d1}
\unitlength 4mm 
\begin{picture}(19,4.5)(2,8)
\thicklines
\put(5,13){\vector(1,0){9}}
\put(23,12.25){\vector(0,-1){1.5}}
\put(23,9){\vector(0,-1){1.25}}\put(23,8){\vector(0,1){1.25}}
\put(23,4.75){\vector(0,1){1.5}}
\put(5,4){\vector(1,0){9}}
\put(-2,9.25){\vector(0,1){3}}
\put(-2,7.75){\vector(0,-1){3}}
\put(-5,13){\makebox(0,0)[cl]{$v\boxx_0v\cong(v\boxx_1v)\boxx_0 v$}}
\put(-5,8.5){\makebox(0,0)[cl]{$v\boxx_0 e \cong v$}}
\put(-5,4){\makebox(0,0)[cl]{$v\boxx_0v\cong(v\boxx_1v)\boxx_0v$}}
\put(14.5,13){\makebox(0,0)[lc]{
$\calK(\ttM\odot(\ttM\odot\ttM),\ttM\odot\ttM)\boxx_0\ \calK(\ttM\odot\ttM,\ttM)$}}
\put(23,10){\makebox(0,0)[cc]{$\calK(\ttM\odot(\ttM\odot\ttM),\ttM)$}}
\put(23,7){\makebox(0,0)[cc]{$\calK((\ttM\odot\ttM)\odot\ttM,\ttM)$}}
\put(14.5,4){\makebox(0,0)[lc]{
$\calK((\ttM\odot\ttM)\odot\ttM,\ttM\odot\ttM)\boxx_0\ \calK(\ttM\odot \ttM,\ttM)$}}
\put(9,13.25){\makebox(0,0)[bc]{\scriptsize $\odot(u\boxx_1 \mu)\boxx_0 \mu$}}
\put(23.5,11.5){\makebox(0,0)[lc]{\scriptsize $\circ$}}
\put(23.5,8.5){\makebox(0,0)[lc]{\scriptsize $\cong$}}
\put(23.5,5.5){\makebox(0,0)[lc]{\scriptsize $\circ$}}
\put(9,3.75){\makebox(0,0)[tc]{\scriptsize $\odot(\mu\boxx_1 u)\boxx_0 \mu$}}
\end{picture}
\end{equation}
\begin{equation}
\label{d2}
\posunm{2}
\unitlength 4mm
\begin{picture}(19,4.75)(5,8)
\thicklines
\put(7.25,12){\vector(1,0){5}}
\put(20.5,11){\vector(0,-1){2}}
\put(26.5,5){\vector(0,1){2}}
\put(7.25,4){\vector(1,0){11}}
\put(2,11){\vector(0,-1){2}}
\put(2,5){\vector(0,1){2}}
\put(2,12){\makebox(0,0)[cc]{$v\boxx_0 v \cong (v\boxx_1v)\boxx_0v$}}
\put(12.5,12){\makebox(0,0)[cl]{
$\calK(\tte \odot \ttM,\ttM\odot \ttM)\boxx_0\ \calK(\ttM \odot \ttM,\ttM)$}}
\put(13,8){\makebox(0,0)[cl]{$\calK(\ttM,\ttM) \cong \calK(\tte \odot
    \ttM,\ttM) \cong \calK(\ttM \odot \tte,\ttM)$}}
\put(18.5,4){\makebox(0,0)[cl]{
$\calK(\ttM \odot \tte,\ttM\odot \ttM)\boxx_0\ \calK(\ttM \odot \ttM,\ttM)$}}
\put(2,4){\makebox(0,0)[cc]{$v\boxx_0 v \cong (v\boxx_1v)\boxx_0v$}}
\put(2,8){\makebox(0,0)[cc]{$v$}}
\put(9.5,12.25){\makebox(0,0)[bc]{\scriptsize $\odot(\nu \boxx_1u)\boxx_0 \mu$}}
\put(21,10){\makebox(0,0)[lc]{\scriptsize $\circ$}}
\put(27,6){\makebox(0,0)[lc]{\scriptsize $\circ$}}
\put(12.5,3.75){\makebox(0,0)[tc]{\scriptsize $\odot(u \boxx_1 \nu)\boxx_0 \mu$}}
\put(8.5,8.25){\makebox(0,0)[cb]{\scriptsize $u$}}
\put(3,8){\vector(1,0){9}}
\end{picture}
\end{equation}
\begin{equation}
\label{d3}
\unitlength 4mm
\posunm{1.7}
\begin{picture}(19,5)(5,7.5)
\thicklines
\put(6,12){\vector(1,0){9}}
\put(20.5,11){\vector(0,-1){2}}
\put(20.5,5){\vector(0,1){2}}
\put(7.25,4){\vector(1,0){5.75}}
\put(4,11){\vector(0,-1){2}}
\put(4,5){\vector(0,1){2}}
\put(4,12){\makebox(0,0)[cc]{$v\boxx_0v$}}
\put(20.5,12){\makebox(0,0){$\calK(\ttM \odot \ttM,\ttM)\boxx_0\ \calK(\ttM,\ttM)$}}
\put(20.5,8){\makebox(0,0){$\calK(\ttM \odot\ttM,\ttM)$}}
\put(20.5,4){\makebox(0,0){$\calK(\ttM \odot \ttM,\ttM \odot
    \ttM)\boxx_0 \ \calK(\ttM \odot\ttM,\ttM)$}}
\put(4,4){\makebox(0,0)[cc]{$(v\boxx_1v)\boxx_0v$}}
\put(4,8){\makebox(0,0)[cc]{$v$}}
\put(10.5,12.25){\makebox(0,0)[bc]{\scriptsize $\mu \boxx_0u$}}
\put(21,10){\makebox(0,0)[lc]{\scriptsize $\circ$}}
\put(21,6){\makebox(0,0)[lc]{\scriptsize $\circ$}}
\put(10,3.75){\makebox(0,0)[tc]{\scriptsize $\odot(u \boxx_1u)\boxx_0\mu$}}
\put(10.5,8.25){\makebox(0,0)[cb]{\scriptsize $\mu$}}
\put(5,8){\vector(1,0){12.5}}
\end{picture}
\end{equation}
\begin{equation}
\label{d4}
\unitlength 4mm
\posunm{1}
\begin{picture}(19,2.5)(5,10)
\thicklines
\put(6,12){\vector(1,0){9.75}}
\put(20.5,11){\vector(0,-1){2}}
\put(4,11){\vector(0,-1){2}}
\put(4,12){\makebox(0,0)[cc]{$v\boxx_0v$}}
\put(20.5,12){\makebox(0,0){$\calK(\tte,\ttM)\boxx_0\ \calK(\ttM,\ttM)$}}
\put(20.5,8){\makebox(0,0){$\calK(\tte,\ttM)$}}
\put(4,8){\makebox(0,0)[cc]{$v$}}
\put(10.5,12.25){\makebox(0,0)[bc]{\scriptsize $\nu \boxx_0u$}}
\put(21,10){\makebox(0,0)[lc]{\scriptsize $\circ$}}
\put(10.5,8.25){\makebox(0,0)[cb]{\scriptsize $u$}}
\put(4.75,8){\vector(1,0){13.5}}
\end{picture}
\end{equation}
\begin{equation}
\label{d5}
\unitlength 4mm
\posunm{1}
\begin{picture}(19,2.5)(5,10)
\thicklines
\put(6,12){\vector(1,0){9.75}}
\put(20.5,11){\vector(0,-1){2}}
\put(4,11){\vector(0,-1){2}}
\put(4,12){\makebox(0,0)[cc]{$v\boxx_0v$}}
\put(20.5,12){\makebox(0,0){$\calK(\ttM,\ttM)\boxx_0\ \calK(\ttM,\ttM)$}}
\put(20.5,8){\makebox(0,0){$\calK(\ttM,\ttM)$}}
\put(4,8){\makebox(0,0)[cc]{$v$}}
\put(10.5,12.25){\makebox(0,0)[bc]{\scriptsize $u \boxx_0u$}}
\put(21,10){\makebox(0,0)[lc]{\scriptsize $\circ$}}
\put(10.5,8.25){\makebox(0,0)[cb]{\scriptsize $u$}}
\put(4.75,8){\vector(1,0){13.5}}
\end{picture}
\end{equation}
\end{subequations}

Using the map $e \to v$ of~Definition~\ref{sec:spn-mono-categ-1}(v), 
one defines 
the morphisms $\nubar \in \Und\calK(\tte,\ttM)$, $\mubar \in
\Und\calK(\ttM\odot \ttM,\ttM)$ and  $\ubar \in \Und\calK(\ttM,\ttM)$ of the
underlying category as the compositions
\[
\nubar := e \to v \stackrel \nu\to \calK(\tte,\ttM),\
\mubar := e \to v \stackrel \mu\to \calK(\ttM\odot \ttM,\ttM)\ 
\mbox{ and }\
\ubar := e \to v \stackrel u\to \calK(\ttM,\ttM).
\]
Diagram~(\ref{d4}) extends into
\[
\unitlength 4mm
\posunm{1.2}
\begin{picture}(19,2.75)(2,10)
\thicklines
\put(6,12){\vector(1,0){9.75}}
\put(20.5,11){\vector(0,-1){2}}
\put(4,11){\vector(0,-1){2}}
\put(4,12){\makebox(0,0)[cc]{$v\boxx_0v$}}
\put(-3,12){\makebox(0,0)[cc]{$e\boxx_0v$}}
\put(-1.5,12){\vector(1,0){3.5}}
\put(3,8.4){\vector(-2,1){5}}
\put(1.5,10){\makebox(0,0)[cb]{\scriptsize $\cong$}}
\put(20.5,12){\makebox(0,0){$\calK(\tte,\ttM)\boxx_0\ \calK(\ttM,\ttM)$}}
\put(20.5,8){\makebox(0,0){$\calK(\tte,\ttM)$}}
\put(4,8){\makebox(0,0)[cc]{$v$}}
\put(10.5,12.25){\makebox(0,0)[bc]{\scriptsize $\nu \boxx_0u$}}
\put(21,10){\makebox(0,0)[lc]{\scriptsize $\circ$}}
\put(10.5,8.25){\makebox(0,0)[cb]{\scriptsize $\nu$}}
\put(4.75,8){\vector(1,0){13.5}}
\end{picture}
\]
which shows that $\nu$ is determined by the composition $\nubar
\boxx_0 u$ of the top two horizontal maps, i.e.~by $\nubar$ and $u$.
Similarly one proves, using~(\ref{d3}), that $\mu$ is determined by
$\mubar$ and $u$. Finally,~(\ref{d5}) implies that $\ubar$
equals the unit map of the underlying category.

{}From~(\ref{d1}) and~(\ref{d2}) one concludes that $(\ttM,\mubar,\nubar)$
is a unital monoid in the underlying category
$\Und\calK$.  Diagram~(\ref{d4}) asserts that $u$ is a monoid morphism. 
From~(\ref{d3}) one gets the last coherence condition:

\begin{equation}
\label{eq:11}
\unitlength 4mm
\posunm{1.7}
\begin{picture}(19,4)(5,7.5)
\thicklines
\put(6,12){\vector(1,0){9}}
\put(20.5,11){\vector(0,-1){2}}
\put(20.5,5){\vector(0,1){2}}
\put(7.25,4){\vector(1,0){5.75}}
\put(4,8){\vector(0,-1){3}}
\put(4,8){\vector(0,1){3}}
\put(4,12){\makebox(0,0)[cc]{$e\boxx_0v$}}
\put(20.5,12){\makebox(0,0){$\calK(\ttM \odot \ttM,\ttM)\boxx_0\ \calK(\ttM,\ttM)$}}
\put(20.5,8){\makebox(0,0){$\calK(\ttM \odot\ttM,\ttM)$}}
\put(20.5,4){\makebox(0,0){$\calK(\ttM \odot \ttM,\ttM \odot
    \ttM)\boxx_0 \ \calK(\ttM \odot \ttM,\ttM)$}}
\put(4,4){\makebox(0,0)[cc]{$(v\boxx_1v)\boxx_0e$}}
\put(10.5,12.25){\makebox(0,0)[bc]{\scriptsize $\mubar \boxx_0u$}}
\put(21,10){\makebox(0,0)[lc]{\scriptsize $\circ$}}
\put(21,6){\makebox(0,0)[lc]{\scriptsize $\circ$}}
\put(10,3.75){\makebox(0,0)[tc]{\scriptsize $\odot(u \boxx_1u)\boxx_0\mubar$}}
\put(5,8){\makebox(0,0)[l]{\scriptsize $\cong$}}
\end{picture}
\end{equation}

This shows that $\fAss$-algebras determine a monoid in $\calK$. The
opposite implication is now obvious as well as the statement about
the isomorphism of categories.
\end{proof}

\section{Center, $\delta$-center  and homotopy center of a monoid}

\subsection{Multiplicative $1$-operads in duoidal categories.}
In the following definition, $\fAss$ is the $1$-operad in $\duo =
(\duo, \boxx_0,\boxx_1,e,v)$ introduced in Example~\ref{Ass}.

\begin{definition}
\label{sec:mutipl-1-oper}
A $1$-operad $A$ in a duoidal category $\duo$ is {\em
multiplicative\/} if it is equipped with a $1$-operad map $\fAss \to A$.
\end{definition}

By definition, a multiplicative structure on $A$ is given by a system
$v \to A(n)$, $n \geq 0$, of $\duo$-morphisms satisfying
appropriate compatibility conditions.
Let  $\Delta$ be the classical simplicial
category whose objects are finite ordinals 
$[n] = \{0,\ldots,n\}$, $n \geq 1$, and morphisms in $\Delta(m,n)$ are
order-preserving set maps $\{0,\ldots,m\} \to \{0,\ldots,n\}$.
\label{page14}

\begin{proposition}
\label{Lei}
Each multiplicative operad $A$ determines a cosimplicial object $A :
\Delta \to \duo$ in $\duo$ whose value at $n \in \Delta$ is $A(n)$.
\end{proposition}

\begin{proof}
Assume that the multiplicative structure of $A$ is given by a
system $m_n : v \to A(n)$, $n \geq 0$, of $\duo$-maps.
We define the cosimplical structure on $A$ by specifying the actions
of the standard generating maps of $\Delta$. The coboundary
$d_0 : A(n) \to A(n+1)$ is the composition
\[
A(n) \!\cong\! \big(v \boxx_1 A(n)\big) \boxx_0 e
\!\to \!
\big(v \boxx_1 A(n)\big) \boxx_0 v 
\stackrel {(m_1 \boxx_1 \id) \boxx_0 m_2}\vvlra 
\big(A(1) \boxx_1 A(n)\big) \boxx_0 A(2) \stackrel \gamma \to A(n+1),
\]
where the second map uses the canonical morphism $c: e \to v$. 
Likewise, $d_{n+1} : A(n) \to A(n+1)$ is the
composition
\[
A(n)\! \cong\! \big(A(n) \boxx_1 v\big) \boxx_0 e
\!\to\! \big(A(n) \boxx_1 v\big) \boxx_0 v 
\stackrel {(\id \boxx_1 m_1) \boxx_0 m_2}\vvlra
\big(A(n) \boxx_1 A(1)\big) \boxx_0 A(2) \stackrel \gamma \to A(n+1).
\]
For $1\leq i \leq n$, the map $d_i: A(n) \to A(n+1)$ is the composition
\begin{align*}
A(n) \cong e \boxx_0 A(n)& \stackrel{c \boxx_0 \id}\vlra v
\boxx_0 A(n) \cong (\boxx_1^n v ) \boxx_0 A(n)
\stackrel{f_i \boxx_0 \id}\vlra
\\
&\hskip 1em
\stackrel{f_i  \boxx_0 \id}\vlra \big(\boxx_1^{i-1} A(1) 
\boxx_1 A(2) \boxx_1^{n-i} A(1)\big)\boxx_0 A(n) \stackrel\gamma\to A(n+1),
\end{align*}
where $f_i := \boxx_1^{i-1} m_1 \boxx_1 m_2 \boxx_1^{n-i} m_1$.
The cosimplicial degeneracies $s_i : A(n+1) \to A(n)$, $0
\leq i \leq n$, are the compositions
\begin{align*}
A(n+1) \cong e \boxx_0 A(n+1)& \stackrel{c \boxx_0 \id}\vlra v
\boxx_0 A(n+1) \cong (\boxx_1^{n+1} v ) \boxx_0 A(n+1)
\stackrel{g_i  \boxx_0 \id}\vlra
\\
&\hskip 1em
\stackrel{g_i  \boxx_0 \id}\vlra \big(\boxx_1^{i} A(1) 
\boxx_1 A(0) \boxx_1^{n-i} A(1)\big)\boxx_0 A(n+1) \stackrel\gamma\to A(n),
\end{align*}
where $g_i := \boxx_1^{i} m_1 \boxx_1 m_0 \boxx_1^{n-i} m_1$.
\end{proof}

\subsection{Hochschild object, center and homotopy center}
\label{sec:hochsch-object-homot}

Assume that $\duo$ is complete as a $V$-category. By
\cite[Theorem 3.73]{Kelly} this is equivalent to $\duo$ having small
conical limits and $V$-cotensors. The last condition  means that, for any $a\in V$
and $y\in \duo$, there exists an object $y^a\in\duo$ such that
\[
{\duo}(x,y^a)\simeq
{V}\big(a,{\duo}(x,y)\big)
\] 
naturally for all $x\in\duo.$
We also fix a cosimplicial object $\delta :
\Delta \to V$
in $V$. 
Then one defines the $\delta$-{\em totalization\/} 
of a cosimplicial object $\phi : \Delta
\to \duo$ as the $V$-enriched end
\[
\Tot_\delta(\phi) := \int_{n \in \Delta} \phi(n)^{\delta(n)} \in \duo.
\]

Let $A$ be a multiplicative $1$-operad in $\duo$. By
Proposition~\ref{Lei}, it determines a cosimplicial object in $\duo$
(denoted again by $A$).

\begin{definition}
The {\em Hochschild $\delta$-object\/} of a multiplicative $1$-operad
$A$ is defined as
\[
\CH_\delta(A) := \Tot_\delta(A).
\] 
\end{definition}

The endomorphism operad $\End_\ttM$ of a  monoid $\ttM$ in a
$\duo$-monoidal category $\calK$ is, by Proposition~\ref{sec:spn-mono-categ-2}, a
multiplicative $1$-operad in $\duo$. 

\begin{definition} 
The {\em $\delta$-center\/} 
of a monoid $\ttM$ is defined as
\[
\CH_\delta(\ttM,\ttM) := \CH_\delta(\End_\ttM).
\]
\end{definition}

If $\delta = I$ is the constant cosimplicial object that equals $I\in
V$ for all $n$, then the $\delta$-center $Z(\ttM) :=
\CH_I(\ttM,\ttM)$ will be called the {\em center} of $\ttM.$
Notice that, in general, the center of a monoid in $\calK$ lives in
the category $\duo$, not in $\calK.$ It is not difficult to see that
the center of a monoid $\ttM$ is given by the following equalizer in
$\duo$:
\begin{equation}
\label{equalizer}
Z(\ttM)\to {\calK}(\eta,\ttM)
\eqv
{\calK}(\ttM, \ttM).
\end{equation}

\begin{example}
\label{centerofunit} 
A trivial example of a monoid in $\calK$ is the unit object $\eta$.
It is obvious from (\ref{equalizer}) that $Z(\eta) =\calK(\eta,\eta)$ and
that $\calK(\eta,\eta)$ is a duoid in $\duo.$
\end{example}

\begin{theorem}
\label{Za_pet_dni_stehovani.} 
Let $A$ be a multiplicative operad and $\delta = I$. Then the
Hochschild object $\CH_I(A)$ has the canonical structure of a duoid in
$\duo$. In particular, the center of a monoid $\ttM\in \calK$ has a
canonical structure of a duoid in $\duo.$
\end{theorem}

\begin{proof}
We describe the structure morphisms for the duoid $\CH_I(A).$ 
To construct a morphism
\[
\CH_I(A)\boxx_0\CH_I(A)\to \CH_I(A)
\]
it is enough to construct a morphism  $\CH_I(A)\boxx_0\CH_I(A)\to
A(0)$ which equalizes the coboundary operators
$d_0,d_1: A(0)\to A(1)$. One can take the composite
\[
\CH_I(A)\boxx_0\CH_I(A)\stackrel{\pi\boxx_0 \pi}{\vlra} A(0)\boxx_0 A(0)
\stackrel{d_0\boxx_0 d_0}{\vlra}
A(1)\boxx_0 A(1) \to A(1) \stackrel{s_0}{\to} A(0),
\]
in which $\pi: \CH_I(A) \to A(0)$ is the canonical map. Observe that,
since $\pi$ equalizes $d_0$ and $d_1$, instead of $d_0 \boxx_0 d_0$,
one could have taken, in the above composition, $d_i \boxx_0 d_j$ with
arbitrary $i,j \in \{0,1\}$.
Analogously, we construct a  morphism
\[
\CH_I(A)\boxx_1\CH_I(A)\to \CH_I(A)
\]
using the composite
\begin{align*}
\CH_I(A)\boxx_1\CH_I(A)&\to A(0)\boxx_1 A(0) \simeq \big(A(0)\boxx_1
A(0)\big)\boxx_0 e \to
\\
 \to& \big(A(0)\boxx_1 A(0)\big)\boxx_0 v \to
\big(A(0)\boxx_1 A(0)\big)\boxx_0 A(2) \to A(0).
\end{align*}
 We define the unit 
\[
v\rightarrow \CH_I(A)
\] 
for the second product using the composite
\[
v\to A(1)\stackrel{s_0}{\to} A(0)
\]
and the unit for the first product by  composing
\[
e\to v\to \CH_I(A).
\] 
We leave a long, tedious, but straightforward verification of the
correctness of our definitions, as well as the verification of the duoid
axioms to the reader.
\end{proof}

C.~Barwick developed in \cite{barwick}  a notion of a model category
enriched in a monoidal model category. We can easily adapt his
definition to the situation of an enrichment over a duoidal category.
So, let us assume that $V$ is a monoidal model category, the category
$\duo$ is a model category which is a monoidal model $V$-category for
each of the monoidal structures on $\duo$, and $\calK$ is a
$\duo$-monoidal model category.  In this case one can speak about a
{\it standard system of simplices} for $V$ in the sense of
\cite[Definition~A.6]{BergerMoerdijk} (see also \cite{batanin-berger},
Section 3.5).  Let $\delta$ be such a standard system of simplices for
$V.$ We also assume that there is a model structure on the category of
monoids in $\calK$ and $Fb(\ttM)$ is a fibrant replacement for a
monoid $\ttM$.

\begin{definition} 
The $\delta$-center $\CH_\delta\big(Fb(\ttM),Fb(\ttM)\big)$ will be called the
{\em homotopy center} of $\ttM$ and will be denoted $\CH(\ttM,\ttM)$.
\end{definition}

\begin{remark} 
The homotopy aspects of the theory of the center are out of the scope of this
paper and will be considered in the sequel
\cite{batanin-berger-markl:homotopycenter}.  We will show there that,
under some, not very restrictive, technical conditions the notion of
homotopy center does not depend (up to homotopy) on the standard
system of simplices we use (see Example~\ref{JArca} for
illustration). This justifies our terminology.
\end{remark}

\begin{example}
Let $V=\Set$ and let $\duo$ be a closed braided monoidal category,
i.e~$\duo$ is enriched over itself. Then the center of a monoid $\ttM$ in
$\duo$ is the equalizer
\[
Z(\ttM)\to\ttM \simeq {\duo}(e,\ttM)
\eqv
{\duo}(\ttM, \ttM).
\]
where the two arrows are induced by the left and right multiplication in
$\ttM.$ Therefore, $Z(\ttM)$ is the classical center of $\ttM.$
\end{example}

\begin{example}
\label{JScenter} 
Let $V= \duo = \calK = \Cat.$ A monoid $\ttM$ in $\Cat$ is a strict
monoidal category. Let $\delta$ be the cosimplicial object in $\Cat$
whose $n$-th space is equal to the chaotic groupoid with $n+1$
objects. This is a standard system of simplices in $\Cat$ if we equip
$\Cat$ with a Joyal-Tirney model structure for which weak equivalences
are categorical equivalences. In this model structure all objects in
$Cat$ are fibrant, hence the $\delta$-center of $\ttM$ is its
homotopy center and is equal to the Joyal-Street center of $\ttM$
\cite{JS}.

\end{example}
 
\begin{example}\label{laxcenter} 
If we use, in the previous example, $\delta$ which in dimension $n$
equals to the free category on a linear graph with $n+1$ objects
$${\bullet_0 \to \bullet_1 \to \ldots \to \bullet_{n-1} \to
\bullet_n}$$ then the $\delta$-center of $\ttM$ is its lax-center (or
colax if we reverse the orientation in $\delta$), see~\cite{DayStreet}.
\end{example} 

\begin{example} 
\label{JArca} 
Let $k$ be a commutative ring and let $V = \duo = \calK = \Chain$ the
category of chain complexes of $k$-modules.  Let $\delta :=
C_*(\Delta^\bullet)$, the complex of normalized simplicial chains on
the standard simplicial simplex $\Delta^\bullet$.  This $\delta$ is a
standard system of simplices.  The $\delta$-center of a monoid $\ttM$
(i.e.\ a unital differential graded algebra) is its normalized
Hochschild complex. It is the homotopy center of $\ttM$.

Instead of $\delta$ given by normalized chains we can take
$\tilde{\delta}=Lan_i(\delta_{un})$, the left Kan extension of $\delta_{un}:
\Delta_{in}\to \Chain$ given by un-normalized chains. Here
$i:\Delta_{in}\subset \Delta$ is the subcategory of injections.  As
follows from \cite[Proposition~A.6]{batanin-berger-markl} and the
discussion in the appendix to that paper, the $\tilde{\delta}$-center of
$\ttM$ will be the unnormalized Hochschild complex of $\ttM.$ It is
classical that $CH_{\delta}(\ttM)$ is weakly equivalent to
$CH_{\tilde{\delta}}(\ttM)$.
\end{example}

\section{The duoidal category $\funny(\cat,\duo)$}
\label{sec:spn-categories}

Let us fix a small (with respect to some universe $\bbU$) category
$\cat$ with the set of objects $\cat_0$, set of arrows $\cat_1$,
source and target maps $s_\cat,t_\cat : \cat_1 \to \cat_0$, and the
identity map $i_\cat : \cat_0 \to \cat_1$. We also fix a duoidal
$V$-category $\duo = (\duo,\boxx_0,\boxx_1,e,v)$ as in
Definition~\ref{sec:spn-mono-categ-1}. We assume, in addition, that
$\duo$ has small products and coproducts and both $\boxx_0$ and
$\boxx_1$ commute with coproducts in each variable.

\begin{remark}
\label{sec:first-mono-struct-1} 
In our applications we often assume that $\cat$ is a large category
with respect to $\bbU$ such as the category of small categories.  This
is not, however, a big obstacle. Let $\bbU' \supset \bbU$ be a bigger
universe with respect to which the set of objects of $\cat$ is a small set
i.e.~$\cat_0 \in \bbU'$.  There is a standard procedure in enriched
category theory described in \cite[Section 3.11]{Kelly} known as the
{\em universe enlargement\/} which allows to embed a symmetric
monoidal category~$V$ in an essentially unique manner to a larger
symmetric monoidal category $V'$ in a way that this embedding
preserves all limits and colimits which exist in $V$, but $V'$ also
admits large (with respect to $\bbU$) limits and colimits which are
small with respect to $\bbU'.$ 

This embedding is based on the argument of Day \cite{Day} which uses
the convolution tensor product (Day convolution) on the presheaf
category $\mathcal{SET}^{V^{\it op}}.$ Here $\mathcal{SET}$ is a
version of the category of sets based on the universe $\bbU'$. It is
not difficult to check that Day's argument works equally well for a
duoidal category $\duo$ so we can embed $\duo$ to a larger duoidal
category $\duo'$ which admits all necessary limits and colimits.  Due
to this consideration, we can always assume that $V$ and $\duo$ are
large enough to form limits and colimits we need.
\end{remark}

\begin{definition}
A {\em globe\/} (in $\cat$) is a diagram
\begin{equation}
\label{globe}
\glb ABfg := \globe ABfg, 
\end{equation}
where $A,B \in \cat_0$ are objects of $\cat$ and $f,g : A \to B$
their morphisms.
\end{definition}

For $\sfG $ as in~(\ref{globe}) we set $s(\sfG):= f$, $t(\sfG):= g$, 
$S(\sfG):= A$ and $T(\sfG):= B$ (the {\em source, target, supersource and
supertarget of $\sfG$\/}, respectively). We will often need the 
`trivial' globes
\begin{subequations}
\begin{eqnarray}
\label{triv-globe}
\posun{-1.4}\posun{1.7}
\sfG(A)&:=& \glb AA{\id_A}{\id_A} = \globe AA{\id_A}{\id_A},\ A \in
\cat_0, \ \mbox { and}
\\
\label{triv-globe1}
\sfG(f)&:=& \glb ABff = \globe ABff,\ f: A \to B \in
\cat_1.
\end{eqnarray}
\end{subequations}

\begin{definition}
\label{funny-object}
Let $\duo$ be a duoidal $V$-category as above. A {\em \spn\
$\duo$-object over $\cat$\/} (or simply a {\em \spn\ object\/}) is a
system $Y = \{Y_\sfG\}$ of objects of $\duo$ indexed by globes in
$\cat$. A~{\em morphism\/} $F : Y' \to Y''$ of \spn\ objects is a
system of morphisms $\{F_\sfG \in \duo(Y'_\sfG,Y''_\sfG)\}$ indexed by
globes in $\cat$.
\end{definition}

We denote by $\funny(\cat,\duo)$ or simply by $\funny$ if $\cat$
and $\duo$ are understood, the $V$-category of \spn\ $\duo$-objects over $\cat$
and their morphisms.

\begin{example}
If $\duo = V$ is the category $\Set$ of sets, a \spn\
$\Set$-object, or a {\em \spn\ set\/} for short, is the same as a
diagram of sets
\begin{equation}
\label{funny-diagram}
\begin{array}{c}
Y
\\
\kategorie st
\\
\cat_1
\\
\kategorie {s_\cat}{t_\cat}
\\
\cat_0
\end{array}
\end{equation}
in which $s_\cat s = s_\cat t$ and $t_\cat t = t_\cat s$. 
Indeed, for $Y$ is as in the above diagram, 
we define the {\em fiber\/} over a globe $\sfG $~as
\begin{equation}
\label{eq:1}
Y_\sfG := \{y \in Y;\ s(y) = s(\sfG),\ t(y) = t(\sfG)\}. 
\end{equation}
The system $\{Y_\sfG\}$ of fibers is then a \spn\ set in the sense of
Definition~\ref{funny-object}. 

On the other hand, any collection $Y_\sfG$ of sets indexed by globes
in $\cat$ assembles into the disjoint union $Y :=
\bigcup_{\sfG}Y_\sfG$. The maps $s,t : Y \to \cat_1$
defined by $s(y) := s(\sfG)$, $t(y) := t(\sfG)$ for $y \in Y_{\sfG}$,
are then as in diagram~(\ref{funny-diagram}).

We call the composition $S := s_\cat
s: Y \to \cat_0$ (resp.~$T := t_\cat t: Y \to \cat_0$) the {\em
supersource\/} (resp.~the {\em supertarget\/}) map.
\end{example}

\begin{convention}
\label{conv}
We will visualize \spn\ $\duo$-objects as
diagrams~(\ref{funny-diagram}) even when $\duo$ is a general duoidal
category so the disjoint union $Y := \bigcup_{\sfG}Y_\sfG$ does not
have a formal sense. We can then think of $Y$ as of a set fibered over
the globes in $\cat$, with the fibers objects of $\duo$.
\end{convention}

\subsection{The first monoidal structure}
\label{sec:first-mono-struct}\def\tmes{\times}
Let $\funny = \funny(\cat,\duo)$ be as in
Definition~\ref{funny-object}. For \spn\ objects $Y_i =
\{Y_{i,\sfG}\}\in \funny$, $i=1,2$, define $Y_1 \tmes_0 Y_2 = 
\{(Y_1 \tmes_0 Y_2)_\sfG\} \in \funny$ by  
\begin{equation}
\label{mon1}
(Y_1 \tmes_0 Y_2)_\sfG:=\coprod_{\sfG_1,\sfG_2}Y_{1,\sfG_1}\boxx_0 Y_{2,\sfG_2}, 
\end{equation}
where the coproduct is taken over all globes $\sfG_1,\sfG_2 $ that
decompose $\sfG$ in the sense that $T(\sfG_1) = S(\sfG_2)$ and
\[
s(\sfG) = s(\sfG_2)s(\sfG_1),\ t(\sfG) = t(\sfG_2)t(\sfG_1)\
\mbox { (the composition in $\cat$)}.
\]
If we think of $Y_1$ and $Y_2$ in terms of
diagrams~(\ref{funny-diagram}), then $Y_1 \tmes_0 Y_2$ is the pullback
\begin{equation}
\label{eq:10}
\pullback{\cat_o}{tt}{ss}{Y_1}{Y_2}{Y_1 \tmes_0 Y_2}.
\end{equation}

The above construction clearly extends into a functor $\tmes_0 : \funny
\times \funny \to \funny$.  Let $0 \in \duo$ be the initial object and
recall that $e\in \duo$ is the unit for $\boxx_0$. Denote by $\II0
= \{\II0_\sfG\} \in \funny$ the object defined by
\begin{equation*}
\label{u1}
\II0_\sfG := \cases {e}{if $\sfG$ is the globe $\sfG(A)$ in~(\ref{triv-globe})
  for some $A \in \cat_0$, and}{0}{otherwise.}
\end{equation*}
It is easy to see that $\II0$ is a two-sided unit for $\tmes_0$.
Observe that, if $\duo = V = \Set$, then $\II0$ is
given by the diagram
\[
\begin{array}{c}
\cat_0
\\
\kategorie {i_\cat}{i_\cat}
\\
\hphantom{.} \hskip 1em \cat_1 \hskip 1em .
\\
\kategorie {s_\cat}{t_\cat}
\\
\cat_0
\end{array}
\] 

\subsection{The second monoidal structure}
For \spn\ $V$-objects $Y_i =
\{Y_{i,\sfG}\}\in \funny$, $i=1,2$, define $Y_1 \tmes_1 Y_2 = 
\{(Y_1 \tmes_1 Y_2)_\sfG\} \in \funny$ by  
\begin{equation}
\label{mon2}
(Y_1 \tmes_1 Y_2)_\sfG :=\coprod_{\sfG_1,\sfG_2} Y_{1,\sfG_1}\boxx_1 Y_{2,\sfG_2}, 
\end{equation}
where $\boxx_1$ is the second monoidal structure of $\duo$ and the coproduct is
taken over all globes $\sfG_1,\sfG_2 $ such that
\begin{equation}
\label{eq:13}
s(\sfG) = s(\sfG_1),\
t(\sfG_1) =
s(\sfG_2) \mbox { and } t(\sfG) = t(\sfG_2).
\end{equation}
In terms of diagrams~(\ref{funny-diagram}), $Y_1 \tmes_1 Y_2$ is the pullback
\[
\posun{3.75}
\pullback{\cat_1}ts{Y_1}{Y_2}{Y_1 \tmes_1 Y_2}
\]
which has to be compared to the pullback~(\ref{eq:10}) defining the
$\tmes_0$-product.  Let $\II1 = \{\II1_\sfG\} \in \funny$ be the object
with
\begin{equation*}
\label{u2}
\II1_\sfG := \cases {v}{if $\sfG$ is the globe $\sfG(f)$ of~(\ref{triv-globe1})
  for some $A \stackrel f\to B  \in \cat_1$, and}{0}{otherwise.}
\end{equation*}
In the diagrammatic language, $\II1$ is the diagram
\[
\begin{array}{c}
\cat_1
\\
\kategorie {\id}{\id}
\\
\cat_1
\\
\kategorie {s_\cat}{t_\cat}
\\
\cat_0
\end{array}  \hskip .5em.
\] 
It is clear that $\II1$ is a two-sided unit for $\tmes_1$. To construct the
canonical map $\II1\tmes_0 \II1 \to \II1$, observe that $(\II1\tmes_0 \II1)_\sfG
\not=0$ only if $\sfG = \sfG(f)$ for some morphism $f$ in $\cat$, in
which case one has the composition
\begin{equation}
\label{eq:15}
(\II1\tmes_0 \II1)_{\sfG(f)} = \coprod_{f = f_2f_1} \II1_{\sfG(f_1)} \boxx_0
\II1_{\sfG(f_2)} =  \coprod_{f = f_2f_1} v \boxx_0 v  \longrightarrow
\coprod_{f = f_2f_1} v,
\end{equation}
in which the last arrow is the map (iv) of
Definition~\ref{sec:spn-mono-categ-1}.  We define the component
$(\II1\tmes_0 \II1)_{\sfG} \to \II1_{\sfG}$ of the structure map $\II1\tmes_0 \II1
\to \II1$ as the composition of the map~(\ref{eq:15}) with the folding
map $\coprod_{f = f_2f_1} v \to v = \II1_\sfG$ if $\sfG = \sfG(f)$
for some $f$, and as the unique map $0 \to 0$ in the remaining cases.

It is clear that, for the object $\II0$
defined in Subsection~\ref{sec:first-mono-struct},
$(\II0\tmes_1 \II0)_\sfG \not = 0$ only if $\sfG = \sfG(A)$ for some $A \in
\cat_0$, in which case
\[
(\II0\tmes_1\II0)_{\sfG(A)} = e \boxx_1 e.
\]
We define the
structure map $\II0 \to \II0\tmes_1 \II0$ to be the map induced, in the obvious
way, by (iii) of Definition~\ref{sec:spn-mono-categ-1}.

Let us describe the interchange law
$(A \tmes_1 B) \tmes_0 (C \tmes_1 D) \to (A \tmes_0 C) \tmes_1 (B
\tmes_0 D)$. {}From the definitions~(\ref{mon1}) and~(\ref{mon2}) of the
products $\tmes_0$ and $\tmes_1$  we get that
\[
(A \tmes_1 B) \tmes_0 (C \tmes_1 D)_{\sfG} =
\coprod_{(\sfG^u_1,\sfG^u_2,\sfG^d_1,\sfG^d_2) \in {\sf L_\sfG}}
(A_{\sfG^u_1} \boxx_1 B_{\sfG^d_1}) \boxx_0(C_{\sfG^u_2}
\boxx_1 D_{\sfG^d_2}),
\]
where ${\sf L_\sfG}$ is the set of all globes
$\sfG^u_1,\sfG^u_2,\sfG^d_1,\sfG^d_2$ such that
\begin{subequations}
\begin{align}
\label{p1}
T(\sfG^u_1) = S(\sfG^u_2),\ T(\sfG^d_1) = S(\sfG^d_2),&\
s(\sfG) = s(\sfG^u_2)s(\sfG^u_1),\ t(\sfG) = t(\sfG^d_2)t(\sfG^d_1)\
\mbox { and}
\\
\label{p2}
t(\sfG^u_1) = s(\sfG^d_1),& \ t(\sfG^u_2) = s(\sfG^d_2).  
\end{align}
\end{subequations}
The `configuration' of the globes
$(\sfG^u_1,\sfG^u_2,\sfG^d_1,\sfG^d_2) \in {\sf L_\sfG}$ is schematically
depicted as 
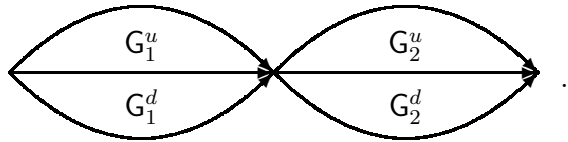
\begin{center}
{
\unitlength=1.000000pt
\begin{picture}(230.00,60.00)(0.00,0.00)
\thicklines
\put(150.00,17.00){\makebox(0.00,0.00){$\sfG^d_2$}}
\put(150.00,40.00){\makebox(0.00,0.00){$\sfG^u_2$}}
\put(50.00,17.00){\makebox(0.00,0.00){$\sfG^d_1$}}
\put(50.00,40.00){\makebox(0.00,0.00){$\sfG^u_1$}}
\put(210.00,25.00){\makebox(0.00,0.00){.}}
\put(200.00,30.00){\vector(1,1){0.00}}
\put(200.00,30.00){\vector(1,-1){0.00}}
\put(100.00,30.00){\vector(1,1){0.00}}
\put(100.00,30.00){\vector(1,-1){0.00}}
\put(100.00,30.00){\vector(1,0){100.00}}
\put(0.00,30.00){\vector(1,0){100.00}}
\qbezier(100.00,30.00)(150.00,-20.00)(200.00,30.00)
\qbezier(100.00,30.00)(150.00,80.00)(200.00,30.00)
\qbezier(0.00,30.00)(50.00,-20.00)(100.00,30.00)
\qbezier(0.00,30.00)(50.00,80.00)(100.00,30.00)
\end{picture}}
\end{center}
It is equally clear that
\[
(A \tmes_0 C) \tmes_1 (B \tmes_0 D)_{\sfG} =
\coprod_{(\sfG^u_1,\sfG^u_2,\sfG^d_1,\sfG^d_2) \in {\sf R_\sfG}}
(A_{\sfG^u_1} \boxx_0 C_{\sfG^u_2}) \boxx_1 
(B_{\sfG^d_1} \boxx_1  D_{\sfG^d_2}),  
\]
where ${\sf R_\sfG}$ is the set of all globes
$\sfG^u_1,\sfG^u_2,\sfG^d_1,\sfG^d_2$ satisfying~(\ref{p1}), but instead
of~(\ref{p2}), a weaker condition
\[
s(\sfG^u_2)s(\sfG^u_1) = t(\sfG^d_2)t(\sfG^d_1).
\]
It is clear that ${\sf L_\sfG} \subset {\sf R_\sfG}$. A `configuration' 
that belongs to ${\sf R_\sfG}$ but not to ${\sf L_\sfG}$ is portrayed below:
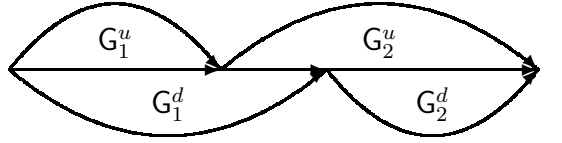
\begin{center}
{
\unitlength=1.000000pt
\begin{picture}(230.00,63.50)(0.00,0.00)
\thicklines
\put(80.00,30.00){\vector(1,-1){0.00}}
\put(120.00,30.00){\vector(3,2){0.00}}
\put(80.00,30.00){\vector(1,0){40.00}}
\put(0.00,30.00){\vector(1,0){80.00}}
\qbezier(0.00,30.00)(40.00,80.00)(80.00,30.00)
\qbezier(120.00,30.00)(160.00,-20.00)(200.00,30.00)
\qbezier(0.00,30.00)(60.00,-20.00)(120.00,30.00)
\qbezier(80.00,30.00)(140.00,80.00)(200.00,30.00)
\put(160.00,17.00){\makebox(0.00,0.00){$\sfG^d_2$}}
\put(140.00,40.00){\makebox(0.00,0.00){$\sfG^u_2$}}
\put(60.00,17.00){\makebox(0.00,0.00){$\sfG^d_1$}}
\put(40.00,40.00){\makebox(0.00,0.00){$\sfG^u_1$}}
\put(210.00,25.00){\makebox(0.00,0.00){.}}
\put(200.00,30.00){\vector(1,1){0.00}}
\put(200.00,30.00){\vector(1,-1){0.00}}
\put(100.00,30.00){\vector(1,0){100.00}}
\end{picture}}
\end{center}
The $\sfG$-component of the interchange law is defined as the map of
coproducts induced by the inclusion ${\sf L_\sfG} \hookrightarrow 
{\sf R_\sfG}$ of the indexing sets, precomposed with the map
\[
\coprod_{(\sfG^u_1,\sfG^u_2,\sfG^d_1,\sfG^d_2) \in {\sf L_\sfG}}
(A_{\sfG^u_1} \boxx_1 B_{\sfG^d_1}) \boxx_0(C_{\sfG^u_2}
\boxx_1 D_{\sfG^d_2}) \longrightarrow
\coprod_{(\sfG^u_1,\sfG^u_2,\sfG^d_1,\sfG^d_2) \in {\sf L_\sfG}}
(A_{\sfG^u_1} \boxx_0 C_{\sfG^u_2}) \boxx_1 
(B_{\sfG^d_1} \boxx_1  D_{\sfG^d_2}) 
\]
induced by the interchange law in $\duo$.

\begin{theorem}
\label{sec:second-mono-struct}
The object $\funny(\cat,\duo) = 
(\funny(\cat,\duo),\tmes_0,\tmes_1,\II0,\II1)$ constructed above 
is a duoidal $V$-category in the sense of
Definition~\ref{sec:spn-mono-categ-1}. Suppose moreover that $\duo$
is $V$-complete. Then $\funny(\cat,\duo)$ is also $V$-complete. 
\end{theorem}

\begin{proof}
We leave the proof of the first part as an exercise. If $V$ is
complete, it is clear that $\funny(\cat,\duo)$ has all conical limits.
Let us prove that $\funny(\cat,\duo)$ has $V$-cotensors if $\duo$ has.
For $v \in V$ and $Y = \{Y_\sfG\}\in \funny(\cat,\duo)$ put $Y^v :=
\{Y^v_\sfG\} \in \funny(\cat,\duo)$, where $Y^v_\sfG$ is the
$V$-cotensor of $Y_\sfG \in \duo$.  Let us verify that this formula
defines a cotensor in $\funny(\cat,\duo)$. For $X = \{X_\sfG\} \in
\funny(\cat,\duo)$ one has
\begin{align*}
{\funny}(X,Y^v) & \cong
\textstyle\prod_\sfG \duo(X_\sfG,Y^v_\sfG) 
\cong \prod_\sfG V\big(v,\duo(X_\sfG,Y_\sfG)\big) 
\cong  V\big(v,\prod_\sfG\duo(X_\sfG,Y_\sfG)\big) 
\\
&\cong
V\big(v,{\funny}(X,Y)\big)
\end{align*}
as required.
\end{proof}

Observe that, if $\duo = V$, i.e.~if $\boxx_0 = \boxx_1 = \ot$ and
$e=v=I$, then  the structure map $\II0 \to \II0
\tmes_1 \II0$ of \ $\funny(\cat,\duo)$ is an isomorphism. Also, the
assumption of the second part of
Theorem~\ref{sec:second-mono-struct} is satisfied, therefore 
the category $\funny(\cat,\duo)$ has cotensors.

\begin{example}
If $\cat$ is the one-object, one-morphism category,
the duoidal category $\funny(\cat,\duo)$ is isomorphic to the basic duoidal
category $\duo$.
\end{example}

\begin{example} 
If $\duo = \Set$, then $\funny(\cat,\duo)$ is the category of
derivation schemes introduced by R.~Street in \cite{StH}. If, in addition,
$\cat$ is the free category on a graph $G$, then $\funny(\cat,\duo)$ is
the category of $2$-computads in the sense of Street, whose
$1$-truncation is $G$.
\end{example}

\section{Span categories and span operads}

{}From now on we will assume that $\duo = (\duo,\boxx_0,\boxx_1,e,v)$ is
a duoidal category which is $V$-complete, has coproducts, and both
monoidal structures in $\duo$ preserve coproducts in each variable.

\begin{definition}
\label{span-cat}
A {\em \spn\ $\duo$-category\/} is a category enriched over the
$V$-monoidal category $\funny =
(\funny(\cat,\duo),\tmes_0,\II0)$.
\end{definition}

By expanding the above definition, one sees that 
a \spn\ category $\calA$ consists of a class $\Ob(\calA)$ of objects
and of \spn\ $\duo$-sets $\calA(\ttE,\ttF)$ given for any $\ttE,\ttF \in
\Ob(\calA)$, equipped with the composition
\begin{subequations}
\begin{equation}
\label{eq:2}
\circ: \calA(\ttE,\ttF)_\sfG \boxx_0 \calA(\ttF,\ttG)_\sfF \to 
\calA(\ttE,\ttG)_{\sfF \sfG} \  \mbox { (a map in $\duo$)}
\end{equation}
defined for all objects $\ttE,\ttF,\ttG \in \Ob(\calA)$ and globes
$\sfG, \sfF $ satisfying $S(\sfF) = T(\sfG)$. In~(\ref{eq:2}),
$\sfF \sfG$ denotes the globe $\glb
{S(\sfG)}{T(\sfF)}{s(\sfF)s(\sfG)}{t(\sfF)t(\sfG)}$.  Still more explicitly,
the compositions are $\duo$-maps
\begin{equation}
\label{eq:3bis}
\circ:
\calA(\ttE,\ttF)_{\ssglobe ABfg} \boxx_0 \calA(\ttE,\ttF)_{\ssglobe BChl} 
\to \calA(\ttE,\ttF)_{\ssglobe AB{hf}{lg}},
\end{equation}
\end{subequations}
defined for arbitrary $A,B,C \in \cat_0$ and $f,g,h,l \in \cat_0$ for
which the globes in the above display make sense.

The operation $\circ$ is assumed to fulfill the standard associativity
whenever the iterated composition is defined. We also require, for
each $\ttE \in \Ob(\calA)$, the {\em unit map\/} $i_\ttE \in
\calA(\ttE,\ttE)$ having the standard unitality property with respect
to the composition $\circ$.

\begin{convention}
\label{sec:span-categories-span}
As usual, by a {\em map\/} in an enriched category we understand a map in the
underlying category (recalled below). So $i_\ttE$ is in fact a map in
\[
\funny\big(\II0,\calA(\ttE,\ttE)\big) \cong \prod_{A \in \cat_0}
\duo(e,\calA(\ttE,\ttE)_{\sfG(A)}),\ 
\mbox { (cartesian product in $V$)}
\]
where $\sfG(A)$ is the trivial globe~(\ref{triv-globe}).
Because $\duo$ is $V$-enriched, we still have to descent one more step and
interpret $i_\ttE$ as an element of the set  
\[
{\mathcal U}V\big(I, \prod_{A \in \cat_0}
\duo(e,\calA(\ttE,\ttE)_{\sfG(A)})\big),\ 
\mbox { (cartesian product of sets)}
\]
where  ${\mathcal U}V$ is the underlying category of $V$. This convention
will be used throughout the rest of the paper.
\end{convention}

If we interpret \spn\ objects as diagrams~(\ref{funny-diagram}), a
\spn\ category appears as a `partial' category, in which the
categorial composition $\phi\circ \psi$ of $\psi \in \calA(\ttE,\ttF)$ and $\phi
\in \calA(\ttF,\ttG)$ is defined only if $T(\psi) = S(\phi)$. One then has
\[
s(\phi\circ \psi) = s(\phi) s(\psi) \mbox { and } 
t(\phi\circ \psi) = t(\phi)t(\psi),
\] 
which implies $T(\phi\circ \psi) = T(\phi)$ and 
$S(\phi\circ \psi) = S(\psi)$.  The unit $i_\ttE$
is then represented by a~map $i_\ttE : \cat_0
\to \calA(\ttE,\ttE)$ with $si_\ttE = ti_\ttE = i_\cat$ such that
\[
\phi \circ i_\ttE(S(\phi)) = i_\ttF(T(\phi)) \circ \phi = \phi,
\]
for all $\phi \in \calA(\ttE,\ttF)$ and $\ttE,\ttF \in \Ob(\calA)$.

The {\em underlying category\/}  of a \spn\ (=
$\funny$-enriched) category $\calA$ is defined in the usual manner
as the $V$-enriched category $\calUV$
with the same set of objects, and morphism  $\calUV(\ttE,\ttF) :=
\funny(\II0,\calA(\ttE,\ttF))$. It follows from definition that
\[
\calUV(\ttE,\ttF) = \prod_{A \in \cat_0}
\duo\big(e,\calA(\ttE,\ttF)_{\sfG(A)}\big), \ \mbox {(the product in $V$)}
\] 
where $\sfG(A)$ is the trivial $A$-globe~(\ref{triv-globe}). 
In the diagrammatic interpretation~(\ref{funny-diagram}) 
of \spn\ objects one has
\[
\calUV(\ttE,\ttF) := \{\lambda : \cat_0 \to \calA(\ttE,\ttF);\
s\lambda = t\lambda = i_\cat\}.
\]

So, the underlying category $\calUV$ of a span $W$-category $\calA$ is a
$V$-enriched category. It therefore has its own underlying category 
${\mathcal U}^2 \calA := {\mathcal U}(\calUV)$, which is this time an
ordinary category (no enrichment). 
Objects of a \spn\ category $\calA$ are
{\em isomorphic\/} if and only if they are isomorphic as objects of
${\mathcal U}^2 \calA$.

\begin{example}
If $\cat$ is the initial one-object, one-morphism category,
then \spn\ $\duo$-categories over $\cat$ are ordinary
$(\duo,\boxx_0,e)$-enriched categories.
\end{example}

\subsection{Monoidal  span-categories.}
\label{fmc}

Monoidal categories over a duoidal category were introduced in
Section~\ref{sec:duo-categ-mono}. Here we address the particular case
of the duoidal category $\funny(\cat,\duo)$.
The {\em $1$-product\/} of \spn\ $\duo$-categories $\calA_1$ and
$\calA_2$ over $\cat$ is the \spn\ $V$-category \hbox{$\calA_1\times_1
  \calA_2$} over $\cat$ whose class of objects is the cartesian
product $\Ob(\calA_1)\times \Ob(\calA_2)$.  
The morphisms~are
\begin{subequations}
\begin{equation}
(\calA_1\times_1 \calA_2)(\ttE_1 \times \ttE_2,\ttF_1 \times \ttF) :=
  \calA_1(\ttE_1,\ttF_1) \tmes_1 \calA_2(\ttE_2,\ttF_2).
\end{equation}
Explicitly, 
for a globe $\sfG $
and objects $\ttE_i,\ttF_i \in \calA_i$, $i=1,2$, we have
\begin{equation}
\label{eq:3}
(\calA_1\times_1 \calA_2)(\ttE_1 \times \ttE_2,\ttF_1 \times
\ttF)_\sfG 
:= 
\coprod_{\sfG_1,\sfG_2}
 \calA_1(\ttE_1,\ttF_1)_{\sfG_1} 
\boxx_1  \calA_2(\ttE_2,\ttF_2)_{\sfG_2},
\end{equation}
\end{subequations}
with the coproduct over all globes $\sfG_1,\sfG_2$ as in~(\ref{eq:13}).

Loosely speaking, the set of morphisms 
$(\calA_1\times_1 \calA_2)(\ttE_1 \times
\ttE_2,\ttF_1 \times \ttF_2)$ 
is generated by the products  $\phi_1 \boxx_1 \phi_2 \in
\calA_1(\ttE_1,\ttF_1) 
\boxx_1 \calA_2(\ttE_2,\ttF_2)$ satisfying $t(\phi_1) = s(\phi_2)$, see
Figure~\ref{obr2}. 
\begin{figure}
{
\unitlength=1.400000pt
\begin{picture}(120.00,40.00)(0.00,0.00)
\thicklines
\put(95.00,10.00){\makebox(0.00,0.00){$\phi_2$}}
\put(25.00,10.00){\makebox(0.00,0.00){$\phi_1$}}
\put(60.00,20.00){\makebox(0.00,0.00){$t(\phi_1)=s(\phi_2)$}}
\put(70.00,0.00){\line(0,1){10.00}}
\put(120.00,0.00){\line(-1,0){50.00}}
\put(120.00,40.00){\line(0,-1){40.00}}
\put(70.00,40.00){\line(1,0){50.00}}
\put(70.00,30.00){\line(0,1){10.00}}
\put(50.00,0.00){\line(0,1){10.00}}
\put(0.00,0.00){\line(1,0){50.00}}
\put(0.00,40.00){\line(0,-1){40.00}}
\put(50.00,40.00){\line(0,-1){10.00}}
\put(0.00,40.00){\line(1,0){50.00}}
\end{picture}}
\caption{\label{obr2}
The source-target conditions for generators $\phi_1 \boxx_1 \phi_2 \in
\calA_1(\ttE_1,\ttF_1) 
\boxx_1 \calA_2(\ttE_2,\ttF_2)$.}
\end{figure}
The categorical composition is defined componentwise in the obvious
manner.  One clearly has, for \spn\ categories $\calA_1$, $\calA_2$
and $\calA_3$, an isomorphism
\[
(\calA_1 \times_1 \calA_2) \times \calA_3
\cong \calA_1 \times_1 (\calA_2 \times \calA_3),
\] 
but, in general, 
$\calA_1 \times_1 \calA_2 \not\cong (\calA_2 \times_1 \calA_1)$. The
category ${\mathbf 1}_v$ with one object $1$ and the \spn\ set of
morphisms ${\mathbf 1}_v(1,1) := v$
is the unit for the multiplication $\times_1$.

\begin{definition}
\label{mono1}
A {\em \spn\ monoidal $\duo$-category\/} is a
$\funny(\cat,\duo)$-monoidal category in the sense of
Definition~\ref{mono}.
\end{definition}

\begin{observation}
\label{sec:spn-mono-categ}
The functor $\eta: {\mathbf 1}_v \to \calK$ in Definition~\ref{mono} is
specified by an object $\tte := \eta(1)$ together with a
$\cat_1$-family of $\duo$-morphisms $v \to \calK(\tte,\tte)_{\sfG(f)}$,
with $\sfG(f)$ as in~(\ref{triv-globe1}). 
\end{observation}

\subsection{Span operads}

\begin{definition}

A {\em span operad\/} is a $1$-operad, in the sense of
Definition~\ref{1-operad}, in the duoidal category $\funny(\cat,\duo)$. 
\end{definition}

Expanding the above definition, we see that a 
\spn\ operad is an $\funny$-collection 
$X = \{X(n)\}_{n \geq 0}$ such that, for $n \geq 1$, $\Rada k1n \geq 0$, and
globes $\sfG$, $\sfG_i$, $1 \leq i \leq n$, that satisfy 
\[
S(\sfG) = T(\sfG_1) = \cdots = T(\sfG_n) 
\]
and
\[
t(\sfG_1) = s(\sfG_2),\ t(\sfG_2) = s(\sfG_3), \ldots,
 t(\sfG_{n-1}) = s(\sfG_n),
\]
one has a $\duo$-map 
\begin{equation}
\label{eq:19}
\gamma :  \big(X(k_1)_{\sfG_1} \boxx_1 \cdots \boxx_1
X(k_n)_{\sfG_n}\big) \boxx_0 X(n)_\sfG \longrightarrow
X(k_1 + \cdots + k_n)_{\sfG(\Rada \sfG1n)},  
\end{equation}
where $\sfG(\Rada \sfG1n) := 
\glb{S(\sfG)}{T(\sfG_1)}{s(\sfG)s(\sfG_1)}{t(\sfG)t(\sfG_n)}$,
satisfying May's associativity (which, in this case, includes also the
distributivity law in $\duo$) whenever the
iterated compositions are defined.

An operad unit is given
by a map\footnote{In the sense of Convention~\ref{sec:span-categories-span}.}
$j \in \funny\big(e,X(1)\big)$, i.e.~ by maps $j_{\sfG(A)} \in
\duo\big(e, X(1)_{\sfG(A)}\big)$, $A \in \cat_0$, such that
\[
\gamma(x,j_{\sfG(B)},\ldots,j_{\sfG(B)}) = x\ \mbox { and }\
\gamma(j_{\sfG(A)},x_1) = x_1,
\] 
for each $x \in X(n)$, $x_1 \in X(1)$ and $A,B \in \cat_0$ such that
$S(x) = B$ and $T(x_1)=A$.

Informally, a \spn\ operad is a `partial' operad with the composition
$\gamma(x,\Rada xn1)$ defined only if the {\em source and target conditions}
\begin{subequations}
\begin{equation}
\nonumber 
S(x) = T(x_1) = \cdots = T(x_n)
\end{equation}
and
\begin{equation}
\nonumber 
t(x_1) = s(x_2),\ t(x_2) = s(x_3), \ldots,\ t(x_{n-1}) = s(x_n),
\end{equation}
\end{subequations}
are satisfied, see Figure~\ref{obr1}.

\begin{example}
If $\cat$ is the one object, one morphism category,
then a span operad is a $1$-operad in the duoidal category $\duo$ in
the sense of Definition~\ref{1-operad}.

\end{example}

\begin{example}
If $\duo = V=\Set$, we immediately see from the above explicit
description that a span operad is exactly a ${\mathbf f}{\mathbf
c}$-operad of Leinster \cite{Leinster} with discrete graph of colors 
$$
\begin{array}{c}
\cat_0
\\
\kategorie {\id}{\id}
\\
\cat_0
\end{array}
$$
whose $0$-truncation is ${\mathbf I}{\mathbf d}$-operad $\cat.$ 

\end{example}

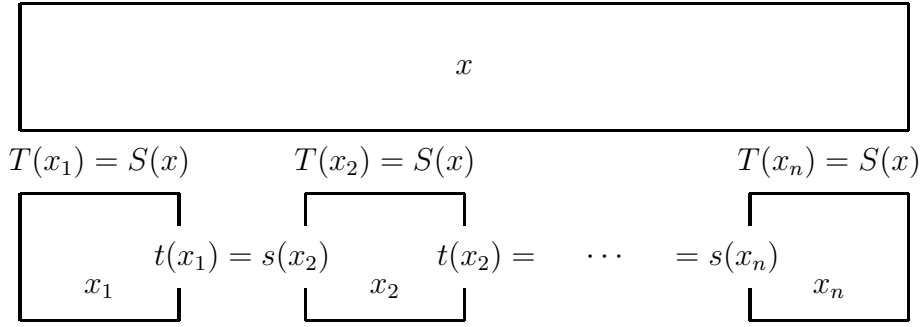
\begin{figure}
{
\unitlength=1.2pt
\begin{picture}(280.00,100.00)(0.00,0.00)
\thicklines
\put(50.00,10.00){\line(0,-1){10.00}}
\put(90.00,10.00){\line(0,-1){10.00}}
\put(140.00,10.00){\line(0,-1){10.00}}
\put(230.00,10.00){\line(0,-1){10.00}}
\put(230.00,40.00){\line(0,-1){10.00}}
\put(140.00,40.00){\line(0,-1){10.00}}
\put(90.00,40.00){\line(0,-1){10.00}}
\put(50.00,40.00){\line(0,-1){10.00}}
\put(0.00,40.00){\line(0,-1){40.00}}
\put(255.00,10.00){\makebox(0.00,0.00){$x_n$}}
\put(115.00,10.00){\makebox(0.00,0.00){$x_2$}}
\put(25.00,10.00){\makebox(0.00,0.00){$x_1$}}
\put(140.00,80.00){\makebox(0.00,0.00){$x$}}
\put(255.00,50.00){\makebox(0.00,0.00){$T(x_n)=S(x)$}}
\put(115.00,50.00){\makebox(0.00,0.00){$T(x_2)=S(x)$}}
\put(25.00,50.00){\makebox(0.00,0.00){$T(x_1)=S(x)$}}
\put(280.00,60.00){\line(-1,0){280.00}}
\put(280.00,100.00){\line(0,-1){40.00}}
\put(0.00,100.00){\line(1,0){280.00}}
\put(0.00,60.00){\line(0,1){40.00}}
\put(185.00,20.00){\makebox(0.00,0.00){$\cdots$}}
\put(223.00,20.00){\makebox(0.00,0.00){$=s(x_n)$}}
\put(147.00,20.00){\makebox(0.00,0.00){$t(x_2)=$}}
\put(70.00,20.00){\makebox(0.00,0.00){$t(x_1)=s(x_2)$}}
\put(280.00,0.00){\line(-1,0){50.00}}
\put(280.00,40.00){\line(0,-1){40.00}}
\put(230.00,40.00){\line(1,0){50.00}}
\put(140.00,40.00){\line(-1,0){50.00}}
\put(90.00,0.00){\line(1,0){50.00}}
\put(50.00,40.00){\line(-1,0){50.00}}
\put(0.00,0.00){\line(1,0){50.00}}
\end{picture}}
\caption{\label{obr1}
The source-target conditions of elements for
which $\gamma(x,\Rada x1n)$ is defined. The composition is assumed
from the bottom to the top.}
\end{figure}

\begin{example}
\label{fAss}
According to Example~\ref{Ass}, one has a span operad 
$\fAss$ with $\fAss(n) := \II1$
for $n \geq 0$. There is also a Forcey span-operad
$\eAss$ with $\eAss(n) := \II0$. If $\cat$ is the initial one-object one-arrow
category and $\duo = V$, then $\fAss$ and $\eAss$ coincide and
are the non-$\Sigma$ operads for unital associative $V$-algebras.
\end{example}

\begin{example}
Each object $\ttE$ of a \spn\ monoidal category $(\calA,\odot,\tte)$
determines the \spn\ {\em endomorphism operad\/} $\End_\ttE =
\coll{\End_\ttE}$ with $\End_\ttE(n) := \calA(\ttE^{\odot
n},\ttE)$. We put, by definition, $\ttE^{\odot 0} := \tte$ so
$\End_\ttE(0) = \calA(\tte,E)$.  The structure operations are given in
an obvious way. The operadic unit $j : e \to \End_E(1) =
\calA(\ttE,\ttE)$ is the unit map of the category $\calA$.
\end{example}

\section{The monoidal span category $\Jarka(O,\calK) $} 
\label{Jarka}

In this section we describe a construction providing examples of span
categories.  More precisely,  we us fix a monoidal $\duo$-category
$\calK = (\calK, \odot, \eta)$ which admits coproducts and $\odot$
preserves them in each variable. Let us fix also a functor $O:\cat\to
\Set$. We will construct  a monoidal span category
$\Jarka(O,\calK)$ which will play the r\^ole of the basis monoidal category
for the construction of the Tamarkin complex.

Let $\Gl$ be a category of globes in $\cat$, that is, the
category whose set of objects is $\cat_0$ and whose set of arrows is
the set of globes in $\cat$. The composition and identities are
obvious.

Let $F:\Gl\rightarrow \Cat(\duo)$ be a
functor.  The enriched version of the Grothendieck construction
$\int{F}$ is the category whose objects are pairs $(A,X)$ where $A\in
\cat_0$ and $X$ is an object of $F(A)$.  The enriched hom
$\int{F}\big((A,X),(B,Y)\big)$ is defined as
\[
\coprod_{\sfG \in \Gl(A,B)}{F(B)}\big(F(\sfG)(X),Y\big).
\]

There is a projection $\pi_0: (\int F)_0 \rightarrow \cat_0.$ 
One therefore has a span category $\Gamma(F)$ whose objects are 
sections of the map 
$\pi_0: (\int F)_0 \rightarrow \cat_0$.
Given two such sections $\ttE,\ttF$, one puts
\[
\Gamma(F)(\ttE,\ttF)
:= \coprod_{A,B \in \cat_0} {\textstyle\int{F}}\big(\ttE(A),\ttF(B)\big)
\]
with the obvious structure of a span $\duo$-object.
The composition and identities are obvious.

Now we will construct a canonical
spanification functor $\Sp(O,\calK): \Gl \rightarrow \Cat(\duo)$ out of a
functor $O:\cat\rightarrow \Set$ and a monoidal $\duo$-category $\calK$.
The objects of the category
$\Sp(O,\calK)(A)$ are  $\calK$-spans of the form 
\begin{equation}\label{spanE}\posun{1.5}
\0-span(O(A), O(A),X) \hskip 1em ,\ \  A \in \cat_0,
\end{equation}
i.e.~collections $X = \{X(a',a'')\}$ of objects of $\calK$ indexed by
elements $a',a'' \in O(A)$. The $\duo$-enriched homs in $\Sp(O,\calK)(A)$ are given by
\[
\Sp(O,\calK)(A)(X,Y) :=   
\prod_{a',a'' \in O(A)} \calK\big(X(a',a''),Y(a',a'')\big).
\] 
It is easy to see that the monoidal
$\duo$-structure $\odot$ of $\calK$ induces a monoidal
$\duo$-structure $\tens$ on
$\Sp(O,\calK)(A)$ by the formula
\begin{equation}\label{tens}
(X \tens Y)(a',a'') := \coprod_{a \in O(A)} X(a',a)\odot Y(a,a'').
\end{equation}
For a globe as in~(\ref{globe}), the functor $\Sp(O,\calK)(\sfG)$ maps a
span~(\ref{spanE}) to the span
\[
O(B)\stackrel{O(f)}{\longleftarrow} O(A)\stackrel{s}{\leftarrow} X 
\stackrel{t}{\rightarrow} O(A)\stackrel{O(g)}{\longrightarrow} O(B).
\]
So we have a span category $\Gamma(\Sp(O,\calK))$.  This category is the main
ingredient for the construction of the Tamarkin complex, so
we describe it explicitly.

To simplify the notation, we will denote 
$O$ by $\wt{(-)}: \cat \to \Set$ and $\Gamma(\Sp(O,\calK))$ 
will be denoted $\Jarka(O,\calK)$.
Objects of $\Jarka(O,\calK)$ are $\cat_0$-families of $\calK$-enriched spans
\begin{equation}
\label{eq:6}
\posun{1.5}
\hphantom{\  A \in \cat_0,}
\0-span(\wt A, \wt A,\ttE_A) \hskip 1em ,\ \  A \in \cat_0,
\end{equation}
i.e.~families
$\ttE = \{\ttE_A(a',a'')\}$, $a',a''\in \wt A$, $A \in \cat_0$, 
$\ttE(a',a'') \in \calK$. 
The \spn\ $\duo$-sets of morphisms $\Jarka(O,\calK)(\ttE,\ttF)
= \{\Jarka(O,\calK)(\ttE,\ttF)_\sfG\}$ have fibers the products
\begin{equation}
\label{eq:8}
\Jarka(O,\calK)(\ttE,\ttF)_{\ssglobe ABfg} := \prod_{a',a'' \in \wt{A}}
\calK\left(\posun{.4} \ttE_A(a',a''), \ttF_B(\Of(a'), \Og(a'')\right)
\end{equation}
of $\duo$-enriched homs in $\calK$. Less formally,~(\ref{eq:8}) is the set of
dashed arrows in the commutative diagram   
\[
\jaruska{{\wt A}}{{\wt B}}{\ttE_A}{\ttF_B}{\Of}{\Og}
\]

The structure maps~(\ref{eq:3bis}) are then 
compositions of the dashed arrows in 
\[
\jaruskaext{{\wt A}}{{\wt B}}{\ttF_A}{\ttE_B}{\Of}{\Og}
\]
In terms of the fibers, the structure maps are compositions of the
maps in the following display:
\[
\begin{aligned}
\prod_{a',a'' \in {\wt A}}
\calK\left(\posun{.4}\right. \hskip -.3em  \ttE_A(a',a''), 
\ttF_B(\Of(a'), \Og(a''))\left.\posun{.4}\hskip -.3em \right)& \boxx_0
\prod_{b',b'' \in {\wt B}}
\calK\left(\posun{.4}\ttF_B(b',b''), \ttG_C(\Oh(b'), \Ol(b''))\right)
\\
&\downarrow
\\
\prod_{a',a'' \in {\wt A}}  
\calK\left(\posun{.4}\right. \hskip -.3em \ttE_A(a',a''), \ttF_B(\Of(a'), 
\Og(a''))\left.\posun{.4}\hskip -.3em \right) & \boxx_0 
\calK \left(\posun{.4} \right. \hskip -.3em
\ttF_B(\Of(a'),\Og(a'')), \ttG_C(\Ohf(a')), \Olg(a'')) 
\hskip -.3em  \left.\posun{.4}\right)
\\
&\downarrow
\\
\prod_{a',a'' \in {\wt A}}
\calK\left(\posun{.4}\ttE_A(a',a''),\right.\! &
\ttG_C(\Ohf(a'),\Olg(a'')) \hskip -.3em \left. \posun{.4}\right).
\end{aligned}
\]
The upper map above is the canonical one and the lower map is the
categorial composition in $\calK$.  The unit map $i_\ttE$ in
$\funny(e,\Jarka(O,\calK)(\ttE,\ttE)$ is the the product in
\[
\prod_{A \in \cat_0}
\duo\big(\posun{.4} e,\Jarka(O,\calK)(\ttE,\ttE)_{\sfG(A)}\big)
=
\prod_{A \in \cat_0} \ \prod_{a',a'' \in {\wt A}} 
\duo\big(\posun{.4}e,\calK(\ttE_A(a',a''),\ttE_A(a',a''))\big) 
\]
of the enriched units $i_{\ttE(a',a'')} \in
\duo\big(\posun{.4}e,\calK(\ttE_A(a',a''),\ttE_A(a',a''))\big)$ of the
category~$\calK$.

\begin{example}
\label{Jaruska} 
The following particular case will be 
relevant to our interpretation of the Tamarkin construction 
addressed in Section~\ref{sec:tamark-compl-funct-1}.  
Let $\cat$ be  the category
of small
dg-categories, $\calK = \duo = V = \Chain$ and $O : \cat \to \Set$ be the object
functor. The objects of the corresponding category $\Jarka(O,\calK)$ will
then be collections of chain complexes $\ttE = \{\ttE_A(a',a'')\}$,
indexed by objects $a',a'' \in A$ of dg-categories $A \in \cat_0$.
\end{example}

We will often drop the indices $A,B,C,... \in \cat_0$ and write simply
$\{\ttE(a',a'')\}$ instead of $\{\ttE_A(a',a'')\}$,~\&c.

Let us prove that the category $\Jarka(O,\calK)$ constructed above has
a natural monoidal structure.  The functor $\tens : \Jarka(O,\calK)
\times_1 \Jarka(O,\calK) \to \Jarka(O,\calK)$ assigns to objects
$\ttE_1 = \{\ttE_1(a',a'')\}$ and $\ttE_2 = \{\ttE_2(a',a'')\}$ of
$\Jarka(O,\calK)$ the object $\ttE_1 \tens \ttE_2 = \{(\ttE_1 \tens
\ttE_2)(a',a'')\} \in \Jarka(O,\calK)$ where in the right hand side we
use the
`local' product defined by (\ref{tens}).
Informally, $\ttE_1 \tens \ttE_2$ is the $\cat_0$-family of the pull-backs
\[
\posun{1.5}
\unitlength 1.5mm
\linethickness{0.4pt}
\begin{picture}(19,13)(10,11)
\thicklines
\put(11,13){\vector(-1,-1){3.2}}
\put(15,13){\vector(1,-1){3.2}}
\put(7.7,8){\makebox(0,0)[r]{${\wt A}$}}
\put(21.5,8){\makebox(0,0)[c]{${\wt A}$}}
\put(13,14){\makebox(0,0)[b]{$\ttE_1$}}
\put(9,12.5){\makebox(0,0)[r]{\scriptsize $s$}}
\put(18,12.5){\makebox(0,0)[l]{\scriptsize $t$}}
\put(17,0){
\put(11,13){\vector(-1,-1){3.2}}
\put(15,13){\vector(1,-1){3.2}}
\put(18.3,8){\makebox(0,0)[l]{${\wt A}$}}
\put(13,14){\makebox(0,0)[b]{\hphantom{.}\hskip 1em $\ttE_2$ \hskip 1em .}}
\put(9,12.5){\makebox(0,0)[r]{\scriptsize $s$}}
\put(18,12.5){\makebox(0,0)[l]{\scriptsize $t$}}
}
\put(8.5,7){
\put(9,13){\vector(-1,-1){3.2}}
\put(17,13){\vector(1,-1){3.2}}
\put(13,14){\makebox(0,0)[b]{$\ttE_1 \tens \ttE_2$}}
}
\end{picture}
\]

Before we explain how the functor $\tens$ acts on morphisms, we need to
expand some definitions.
For objects $\ttE_1, \ttE_2, \ttF_1, \ttF_2 \in \Jarka(O,\calK)$ and a globe
$\sfG = \glb ABfg \in \Gl$, one sees that the fiber 
$(\Jarka(O,\calK) \times_1 \Jarka(O,\calK))(\ttE_1 \times \ttE_2, \ttF_1
\times\ttF_2)_\sfG$ of the mapping space in $\Jarka(O,\calK) \times_1 \Jarka(O,\calK)$ equals
\[
\coprod_{l \in \cat(A,B)}
\left(
\prod_{a'_1,a''_1 \in {\wt A}} \hskip -.2em
\calK\ll\ttE_1(a'_1,a''_1),\ttF_1(\Of(a'_1),\Ol(a''_1))\rr
\boxx_1
\prod_{a'_2,a''_2 \in {\wt A}}  \hskip -.2em
\calK\ll\ttE_2(a'_2,a''_2),\ttF_2(\Ol(a'_2),\Og(a''_2))\rr
\right),
\]
since all globes $\sfG_1$, $\sfG_2$ such that $ t(\sfG_1) =
s(\sfG_2)$ as in~(\ref{eq:3}) have the form 
\[
\sfG_1 = \glb ABfl,\ \sfG_2 = \glb ABlg,
\] 
for some $l : A \to B \in \cat_1$. On the other
hand, the fiber $\Jarka(O,\calK)(\ttE_1 \tens \ttE_2, \ttF_1
\tens\ttF_2)_\sfG$ of the hom space in $\Jarka(O,\calK)$ equals
\[
\prod_{a',a'' \in {\wt A}}
\calK\ll\coprod_{a \in {\wt A}} \ttE_1(a',a) \odot \ttE_2(a,a''),
\coprod_{a \in {\wt A}}    \ttF_1(\Of(a'),a) \odot \ttF_2(a,\Og(a''))\rr.
\]

To define the functor $\tens$ on morphisms, one needs to specify, for
objects $\ttE_1, \ttE_2, \ttF_1, \ttF_2 \in \Jarka(O,\calK)$ and a globe $\sfG$
as above, a $\duo$-morphism
\[
\tens:
(\Jarka(O,\calK) \times_1 \Jarka(O,\calK))(\ttF_1 \times \ttF_2,\ttE_1 \times \ttE_2)_\sfG \to 
\Jarka(O,\calK)(\ttF_1 \tens \ttF_2,\ttE_1 \tens \ttE_2)_\sfG.
\]
One defines this $\duo$-morphism as the composition of the canonical maps
\begin{align*}
\coprod_{l \in \cat(A,B)}
\left(
\prod_{a'_1,a''_1 \in {\wt A}} \hskip -.2em
\calK\ll\ttE_1(a'_1,a''_1),\ttF_1(\Of(a'_1),\Ol(a''_1))\rr
\right.\boxx_1&\left.
\prod_{a'_2,a''_2 \in {\wt A}}  \hskip -.2em
\calK\ll\ttE_2(a'_2,a''_2),\ttF_2(\Ol(a'_2),\Og(a''_2))\rr
\right)
\\
\downarrow&
\\
\coprod_{l \in \cat(A,B)} \
\prod_{a'_1,a''_1,a'_2,a''_2 \in {\wt A}} \hskip -.5em  \left(
\calK\ll\ttE_1(a'_1,a''_1),\ttF_1(\Of(a'_1),\Ol(\right.\hskip -.1em&a''_1))\rr
\boxx_1\left.
\calK\ll\ttE_2(a'_2,a''_2),\ttF_2(\Ol(a'_2),\Og(a''_2))\rr
\right)
\\
\downarrow&
\\
\coprod_{l \in \cat(A,B)} \
\prod_{a'_1,a''_1,a'_2,a''_2 \in {\wt A}} \hskip -0.2em 
\calK\ll\ttE_1(a'_1,a''_1) \odot \ttE_2(a'_2,a''_2),  & \
\ttF_1(\Of(a'_1),\Ol(a''_1)) \odot \ttF_2(\Ol(a'_2),\Og(a''_2))\rr
\\
\downarrow&
\\
\coprod_{l \in \cat(A,B)} \
\prod_{a'_1,a,a''_2 \in {\wt A}} \hskip -.1em 
\calK\ll\ttE_1(a'_1,a) \odot \ttE_2(a,a''_2),  & \
\ttF_1(\Of(a'_1),\Ol(a)) \odot \ttF_2(\Ol(a),\Og(a''_2))\rr
\\
\downarrow&
\\
\prod_{a'_1,a,a''_2 \in {\wt A}} \hskip -.1em 
\calK\ll\ttE_1(a'_1,a) \odot \ttE_2(a,a''_2),  & \
\coprod_{a' \in {\wt A}}    \ttF_1(\Of(a'_1),a') \odot
\ttF_2(a',\Og(a''_2))\rr
\\
\downarrow&
\\
\prod_{a',a'' \in {\wt A}}
\calK\ll\coprod_{a \in {\wt A}} \ttE_1(a',a) \odot \ttE_2(a,a''),  & \
\coprod_{a \in {\wt A}}    \ttF_1(\Of(a'),a) \odot \ttF_2(a,\Og(a''))\rr.
\end{align*}

Observe the necessity of the source-target condition $t(\sfG_1) = l =
s(\sfG_2)$ for the existence of the above composed map.  

The first piece of data specifying the unit functor $\eta: {\mathbf 1} \to
\Jarka(O,\calK)$ as in Observation~\ref{sec:spn-mono-categ} is the object
$\tte = \{\tte(a',a'')\} \in \Jarka(O,\calK)$ defined by
\begin{equation}
\label{eq:4}
\tte(a',a'') := \cases{\eta}{if $a'=a''$, and}0{otherwise.}
\end{equation}
In the diagrammatic language, $\tte$ is the span
\[
\posun{1.4}
\spanext({\wt A},{\wt A},{\wt A},\id,\id).
\]
To define $v \to \Jarka(O,\calK)(\tte,\tte)_{\sfG(f)}$, notice that, for $f
\in \cat_1$,
\[
\Jarka(O,\calK)(\tte,\tte)_{\sfG(f)} = \prod_{a',a'' \in {\wt A}}
\calK(\tte(a',a''), \tte(\Of(a'),\Of(a'')) \cong \prod_{a \in {\wt A}} 
\calK(\eta,\eta).
\] 
With this identification, the $\duo$-morphism $v \to
\Jarka(O,\calK)(\tte,\tte)_{\sfG(f)}$ is the product of the
$\duo$-morphisms $v \to \calK(\eta,\eta)$ of~(\ref{vaction}).

The underlying category $\UJarka(O,\calK)$ has the same objects as
$\Jarka(O,\calK)$, i.e.~families of $\calK$-spans $\ttE =
\{\ttE_A\}$, $A \in \cat_0$, as in~(\ref{eq:6}). We leave as an
exercise to prove that $\UJarka(O,\calK)(\ttE,\ttF)$ consists
$\cat_0$-families $\{\varphi_A : \ttE_A \to \ttF_A\}$ of morphisms of
$\calK$-enriched spans, with the componentwise composition.  By
Proposition \ref{under}, $\UJarka(O,\calK)$, is a monoidal $\duo$-category.

\section{Factorization of functors, and monoids in $\Jarka(O,\calK)$}

In this section we analyze a correspondence between monoids in the
\spn\ monoidal category $\Jarka(O,\calK)$ introduced and further
studied in Section~\ref{Jarka}, and factorizations of the defining
functor $O= \wt{\hskip .5em} : \cat \to \Set$. 
Let $\Gr(\calK)_0$ be the set
of $\calK$-enriched graphs, i.e.~objects
\[
\posun{1.4}
\0-span(S,S,A)
\]
in which $S$ is a set and $A$ a collection of objects of $\calK$ indexed
by $S \times S$. Suppose that the object map $\wt{(-)}_0 : \cat_0 \to
\Set_0$  of the functor $\wt{(-)} : \cat \to \Set$ factorizes as
\begin{subequations}
\begin{equation}
\label{fact1}
\posun{3}
\unitlength 5mm
\linethickness{0.4pt}
\begin{picture}(14,3.2)(2,4)
\put(7,2){\makebox(0,0)[cc]{$\cat_0$}}
\put(13,2){\makebox(0,0)[cc]{$\Set_0$}}
\put(13,6.2){\makebox(0,0)[bc]{$\Gr(\calK)_0$}}
\thicklines
\put(13,6){\vector(0,-1){3}}
\put(7.75,2){\vector(1,0){4}}
\put(10,2.3){\makebox(0,0)[b]{\scriptsize $\wt{(-)}_0$}}
\put(13.3,4.5){\makebox(0,0)[l]{\scriptsize ${\it vrt}_0$}}
\put(8,3){\vector(4,3){4}}
\put(9.5,4.5){\makebox(0,0)[rb]{\scriptsize $L_0$}}
\end{picture} 
\end{equation}
where ${\it vrt\/}_0$ assigns to each $\calK$-graph its set of vertices.
This factorization determines a {\em distinguished object\/} of $\Jarka(O,\calK)$,
namely the $\cat_0$-family $\ttM = \{\ttM_A\}$, 
with $\ttM_A := L_0(A)$, for $A \in \cat_0$.

Suppose that there is a map $F_0: \cat_0 \to \Cat(\calK)_0$ assigning to
each object $A\in \cat_0$ a small $\calK$-category $F_0A$ such that
$\wt{(-)}_0 : \cat_0 \to \Set_0$ further factorizes as
\begin{equation}
\label{fact2}
\posun{3}
\unitlength 5mm
\linethickness{0.4pt}
\begin{picture}(14,3.5)(2,4)
\put(7,2){\makebox(0,0)[cc]{$\cat_0$}}
\put(13,2){\makebox(0,0)[cc]{$\Set_0$}}
\put(7,6.5){\makebox(0,0)[cc]{$\Cat(\calK)_0$}}
\put(13,6.5){\makebox(0,0)[cc]{$\Gr(\calK)_0$}}
\thicklines
\put(7,2.75){\vector(0,1){3}}
\put(8.5,6.5){\vector(1,0){2.25}}
\put(13,5.75){\vector(0,-1){3}}
\put(7.75,2){\vector(1,0){4}}
\put(10,2.3){\makebox(0,0)[b]{\scriptsize $\wt{(-)}_0$}}
\put(13.3,4.25){\makebox(0,0)[l]{\scriptsize${\it vrt}_0$}}
\put(9.75,6.8){\makebox(0,0)[b]{\scriptsize${\it gr}_0$}}
\put(6.75,4.25){\makebox(0,0)[cr]{\scriptsize $F_0$}}
\end{picture} 
\end{equation}
where ${\it gr}_0$ is the underlying graph map.

\begin{proposition}
\label{sec:factorizations}
Suppose that the map $\wt{(-)}_0 : \cat_0 \to \Set_0$ factorizes as
in~(\ref{fact2}). Then the distinguished object $\ttM \in \Jarka(O,\calK)$
constructed above is a monoid in the underling category~$\UJarka(O,\calK)$.
\end{proposition}

\begin{proof}
A monoid structure on $\ttM$ is given by $\funny$-maps $\mubar : e \to
\Jarka(O,\calK)(\ttM\tens \ttM,\ttM)$ and $\nubar : e \to
\Jarka(O,\calK)(\tte,\ttM)$. It is an exercise on definitions that
these maps are given by specifying, for each $A \in \cat_0$, elements
\[
\mubar_{\sfG(A)} \in  \hskip -.3em \prod_{a',a,a'' \in \wt A} \hskip -.3em
\calK\ll F_0A(a',a) \odot F_0A(a,a''),F_0A(a',a'')\rr \ \mbox { and }\
\nubar_{\sfG(A)} \in \prod_{a \in \wt A}
\calK\ll \eta,F_0A(a,a)\rr,
\] 
where $\sfG(A)$ is as in~(\ref{triv-globe}).  Since $F_0A$ is a
$\calK$-category with the set of objects $\wt A$, one can take as
$\mubar_{\sfG(A)}$ the element determined by the $\calK$-category
composition of $F_0A$ and as $\nubar_{\sfG(A)}$ the element determined
by the $\calK$-category identities of $F_0A$. One easily verifies that
this choice gives a monoid in $\UJarka(O,\calK)$.
\end{proof}

Assume that the factorization~(\ref{fact2}) is induced by a
factorization
\begin{equation}
\label{fact3}
\posun{3}
\unitlength 5mm
\linethickness{0.4pt}
\begin{picture}(14,3.5)(2,4)
\put(7,2){\makebox(0,0)[cc]{$\cat$}}
\put(13,2){\makebox(0,0)[cc]{$\Set$}}
\put(7,6.5){\makebox(0,0)[cc]{$\Cat(\calK)$}}
\put(13,6.5){\makebox(0,0)[cc]{$\Gr(\calK)$}}
\thicklines
\put(7,2.75){\vector(0,1){3}}
\put(8.5,6.5){\vector(1,0){2.25}}
\put(13,5.75){\vector(0,-1){3}}
\put(7.75,2){\vector(1,0){4}}
\put(10,2.3){\makebox(0,0)[b]{\scriptsize $\wt{\hskip .5em}$}}
\put(13.3,4.25){\makebox(0,0)[l]{\scriptsize${\it vrt}$}}
\put(10,6.8){\makebox(0,0)[b]{\scriptsize${\it gr}$}}
\put(6.75,4.25){\makebox(0,0)[cr]{\scriptsize $F$}}
\end{picture} 
\end{equation}
\end{subequations}
of the {\em functor\/} $\wt {\hskip .5em} : \cat \to \Set$ via {\em functors\/}.
One then has

\begin{proposition}
\label{Jarca1}
Suppose that the functor $\wt{\hskip .5em} : \cat \to \Set$ factorizes
via functors as in~(\ref{fact3}). Then the object $\ttM \in \Jarka(O,\calK)$ is
a monoid, in the sense of Definition~\ref{na-plaz}, in the
\spn-monoidal category~$\Jarka(O,\calK)$.
\end{proposition}

\begin{proof}
Since factorization~(\ref{fact3}) implies factorization~(\ref{fact2})
$\ttM$ has, by Proposition~\ref{sec:factorizations}, an induced
structure of a monoid in $\UJarka(O,\calK)$. By
Proposition~\ref{sec:spn-mono-categ-2}, it remains to specify an
$\funny$-map $u : v \to \Jarka(O,\calK)(\ttM,\ttM)$. Such a map is determined
by a choice, for each $A \stackrel f\to B \in \cat_1$, of an element
\[
u_{\sfG(f)} \in \prod_{a',a'' \in \wt A} 
\duo \big(v, \calK \ll FA(a',a''),FB(f(a'),f(a''))\rr\big),
\]
where $\sfG(f)$ is as in~(\ref{triv-globe1}).
We take as $u_{\sfG(f)}$ the product of the $\duo$-maps  $F(f)_v$ 
of~(\ref{vtov}) determining the functor $F(f): FA \to FB$.
\end{proof}

The correspondences described above can be organized into the scheme:
\[
\def\arraystretch{1.4}
\begin{array}{rcl}
\mbox{functor $\wt{\hskip .5em}:\cat \to \Set$}&\longmapsto&
\mbox {category $\Jarka(O,\calK)$} 
\\
\mbox {factorization~(\ref{fact1})}&\longmapsto&
\mbox {object of $\Jarka(O,\calK)$}
\\
\mbox {factorization~(\ref{fact2})}&\longmapsto&
\mbox {monoid in the underlying category $\UJarka(O,\calK)$}
\\
\mbox {factorization~(\ref{fact3})}&\longmapsto&
\mbox {monoid in $\Jarka(O,\calK)$}
\end{array}
\]

The table above can be `categorified' as follows. Let us denote by
$\Fact_1$ the category whose objects are factorizations $L_0$ as
in~(\ref{fact1}) and whose morphisms $L'_0 \to L''_0$ are
$\cat_0$-families $\{\alpha_A : L'_0(A) \to L''_0(A)\}$ of graph
morphisms. Likewise, let $\Fact_2$ be the category whose objects are
factorizations $F_0$ as in~(\ref{fact2}) and morphisms $F'_0 \to
F''_0$ are $\cat_0$-families $\{\beta_A : F'_0(A) \to F''_0(A)\}$ of
$\calK$-functors. Finally, let $\Fact_3$ be the category of functors
$F$ as in~(\ref{fact3}), and their natural transformations.

\begin{proposition}
\label{sec:fact-funct-mono}
One has the following natural isomorphisms of categories:
\begin{align*}
\Fact_1 &\cong \mbox{ the underlying category } \UJarka(O,\calK) \mbox { of
} \Jarka(O,\calK),
\\
\Fact_2 &\cong \mbox {the category of monoids in }  \UJarka(O,\calK),
\\
\Fact_3 &\cong \mbox {the category of monoids in }  \Jarka(O,\calK).
\end{align*}
\end{proposition}

The correspondence $O \to \Jarka(O,\calK)$ behaves functorially as well:

\begin{proposition}
The correspondence $O \mapsto \Jarka(O,\calK)$ extends to a
contravariant functor $\Jarka(\calK)$ from the category of functors
$[\cat,\Set]$ and their natural transformations to the category of
monoidal span $\duo$-categories and their span $\duo$-functors.
\end{proposition}

\begin{proof}
To prove the proposition, we need to construct  in a
functorial manner, for an arbitrary natural
transformation $\Phi : O_1 \to O_2$ of functors $O_1,O_2 : \cat \to
\Set$, a span-functor $\Phi^*: \Jarka(O_2) \to \Jarka(O_1)$.

Let $\ttE = \{\ttE_A(a',a'')\}$, $A \in \cat_0$, $a',a'' \in O_2(A)$,
be an object of $\Jarka(O_2,\calK)$. We then define $\Phi^*\ttE \in
\Jarka(O_1,\calK)$ to be the object  $\Phi^*\ttE =
\{\Phi^*\ttE_A(b',b'')\}$,  $A \in \cat_0$, $b',b'' \in O_1(A)$, with 
\[
\Phi^*\ttE_A(b',b'') := \ttE_A(\Phi_A(b'),\Phi_A(b'')),
\]
where $\Phi_A : O_1(A) \to O_2(A)$ is the set map induced by the
transformation $\Phi$. To finish the definition of $\Phi^*$ we need to
specify, for each globe in $\cat$ and each $\ttE, \ttF \in
\Jarka(O_2,\calK)$, a $\duo$-map
\[
\Jarka(O_2,\calK)(\ttE,\ttF)_{\ssglobe ABfg} \longrightarrow
\Jarka(O_1,\calK)(\Phi^*\ttE,\Phi^*\ttF)_{\ssglobe ABfg}.
\]
Expanding definitions, we see that we need to construct a $\duo$-map
from the product
\begin{equation}
\label{eq:7}
\prod_{a',a'' \in O_2(A)}
\duo\left(\posun{.4} \ttE_A(a',a''), \ttF_B(O_2(f)(a'), O_2(g)(a'')\right)
\end{equation}
to the product
\[
\prod_{b',b'' \in O_1(A)}
\duo\left(\posun{.4} \ttE_A(\Phi_A(b'),\Phi_A(b''), 
\ttF_B(\Phi_B O_1(f)(b'), \Phi_B O_1(g)(b'')\right).
\]
Since, of course, $\Phi_B O_1(f)(b') = O_2(f)( \Phi_A(b'))$ and
$\Phi_B O_1(g)(b'') = O_2(g)( \Phi_A(b''))$, the product in the last
display equals
\begin{equation}
\label{eq:9}
\prod_{b',b'' \in O_1(A)}
\duo\left(\posun{.4} \ttE_A(\Phi_A(b'),\Phi_A(b''), 
\ttF_B(O_2(f)( \Phi_A(b')),  O_2(g)( \Phi_A(b''))\right),  
\end{equation}
so we need to construct a map from the product~(\ref{eq:7}) to the
product~(\ref{eq:9}). We take the map induced by the map
\[
\Phi_A \times \Phi_A : O_1(A) \times O_1(A) \to O_2(A) \times O_2(A)
\]
of the indexing sets. It is not difficult to verify that the above
constructions indeed assemble into a span $\duo$-functor $\Phi^*:
\Jarka(O_2,\calK) \to \Jarka(O_1,\calK)$.
\end{proof}

\section{Tamarkin complex of a functor $F:\cat \to \Cat(\calK)$} 
\label{sec:tamark-compl-funct-1}

Let $F:\cat \to \Cat(\calK)$ be a functor. Then we have factorization
(\ref{fact3}) and therefore a~distinguished monoid $\ttM(F) \in
\Jarka(O,\calK).$ Let $\delta:\Delta \to V$ be a fixed cosimplicial object
in $V.$

\begin{definition} 
The {\em Tamarkin complex\/} of a functor $F: \cat \to \Cat(\calK)$ relative
to $\delta$ is the $\delta$-center $\CH_\delta(F,F)
:=\CH_\delta(\ttM(F),\ttM(F))$.
\end{definition}

\begin{example}\label{IdId}
Let $\cat = 1$. A functor $\calC:1\to \Cat(\calK)$ picks up a
$\calK$-category $\calC$. Let $\delta = I$ so the $\delta$-center of a
monoid is its center. Then $\CH_{\delta}(\calC,\calC)$ can be
identified with the duoid of $\calK$-natural transformations of the
identity functor $\Id:\calC\to \calC$. If $\calK=\duo = V$, we get the
classical center of $\calC$.
\end{example}

\begin{example}\label{classicalcenter} 
Take, in the previous example, $\calK = \duo = V=\Chain$ and $\delta$
as in Example~\ref{JArca}.  Then $\CH_{\delta}(\calC,\calC)$ is the
classical Hochschild complex of the $dg$-category $\calC$~\cite{Keller}.
\end{example}

The following result shows that the Tamarkin complex is a powerful tool
for constructing enrichments.

\begin{theorem} 
\label{enrichment}
Let $\cat = \Cat(\calK)$ and $F = \Id: \Cat(\calK)\to \Cat(\calK).$
Let again $\delta = I.$ Then $\CH_{\delta}(\Id,\Id) $ is a duoid in
$\funny(\cat,\duo)$ (i.e.~a $\duo$-enriched $2$-category) with the
property that its underlying duoid $u\big(\CH_{\delta}(\Id,\Id)\big)$
is equal to $\Cat(\calK)$, the $2$-category of $\calK$-categories,
$\calK$-functors and $\calK$-natural transformations.
\end{theorem}

\begin{proof} Direct verification.
\end{proof}

{}From now on we will use the notation $\Cat(\calK)$ for the
$\duo$-enriched $2$-category of $\calK$-categories, $\calK$-functors
and $\calK$-natural transformations.  Theorem \ref{enrichment}
provides a classical interpretation of the center of a monoid as the
object of natural transformations of the identity functor.  Indeed,
if $\ttM$ is a monoid in $\calK$, then $\Sigma(\ttM)$ is a one object
$\calK$-category. Remark \ref{loop} shows that
\[
\Cat(\calK)(\Sigma(\ttM),\Sigma(\ttM))
\]
is a monoidal $\duo$-category with the unit object given by the identity
functor $\Id$. 

\begin{corollary} 
The following four duoids in $\duo$
are naturally isomorphic:
\begin{enumerate}
\item 
the center $Z(\ttM)$ of a monoid $\ttM$ in $\calK$,
\item 
the center $Z(\Id)$ in $\Cat(\calK)\big(\Sigma(\ttM),\Sigma(\ttM)\big)$,
\item   
the duoid 
$\Cat(\calK)\big(\Sigma(\ttM),\Sigma(\ttM)\big)(\Id,\Id)$
in $\duo$ and
\item 
the duoid $\CH_I\big(\Sigma(\ttM),\Sigma(\ttM)\big)$ (see Example \ref{IdId}).
\end{enumerate}
\end{corollary}

\begin{proof}
By Example \ref{centerofunit}, $Z(\Id)\in \duo$ equals
$\Cat(\calK)(\Sigma(\ttM),\Sigma(\ttM))(\Id,\Id).$ On the other hand,
we establish by direct calculation that this object coincides with the
equalizer (\ref{equalizer}) and, therefore, is the center $Z(\ttM).$
\end{proof}

\begin{example}
\label{2cat} 
If $ \calK = \duo = V = \Set$, then $\CH_{I}(\Id,\Id)$ is the
$2$-category of categories $\Cat$ with its cartesian closed structure.
On the other hand, if we define $\delta:\Delta\to \Set$ in dimension $n$
as the set $\{0,\ldots,n\}$ with the obvious coface and codegeneracy
operators, then $\CH_{\delta}(\Id,\Id)$ is the sesquicategory of
categories, functors and their unnatural transformations (so it is
$\Cat$ with its second closed symmetric monoidal structure~\cite{KL}).
\end{example} 

\begin{example} \label{Graycat}
Let $\calK = \duo = V=\Cat$ with its cartesian closed monoidal
structure and $\delta$ be as in Example~\ref{JScenter}. Let $F=\Id:
\Cat(\Cat) \to \Cat(\Cat)$. Then $\CH_{\delta}(\Id,\Id)$ is the
$\mathbf{Gray}$-category $\mathbf{Gray}$ of $2$-categories,
$2$-functors and pseudonatural transformations \cite{GPS}.
\end{example}

\begin{example} 
Replacing $\delta$ from (\ref{JScenter}) by $\delta$ from
(\ref{laxcenter}), we obtain a nonsymmetric version of $\mathbf{Gray}$
which consists of $2$-categories, $2$-functors and lax-natural (or
colax-natural if we change the orientation in $\delta$)
transformations \cite{Gray}.
\end{example}

We end up this section by showing that when $\calK = \duo =V$ is the category
$\Chain$ of chain complexes, $\cat$ is the category $\Cat(\Chain)$ of
small dg-categories, $F = \Id:\Cat(\Chain)\to \Cat(\Chain)$ and
$\delta$ is again as in Example~\ref{JArca}, the resulting
Tamarkin complex indeed coincides with the original Tamarkin's
construction $\bfRhom(-,-)$ 
from~\cite[Definition~3.0.2]{tamarkin:CM07}.  Let us recall its
definition.

For small dg-categories $A$, $B$ and dg-functors  $f,g :A \to B$, one
defines, 
for each $n \geq 0$,  
\[
\bfhom^n(f,g) := \prod_{\Rada a0n \in A} 
\Chain\big(A(a_0,a_1) \ot \cdots \ot A(a_{n-1},a_n),B(f(a_0),g(a_n))\big),
\]
with the product taken over all $(n+1)$-tuples $(\Rada a0n)$ of
objects of $A$. As shown in~\cite{tamarkin:CM07}, the objects
$\bfhom^n(f,g)$ assemble into the cosimplicial chain complex
$\bfhom^*(f,g)$.   

Let $O = \wt{(-)} : \Cat(\Chain) \to \Set$ be the object functor.  As
explained in Example~\ref{Jaruska}, the corresponding category
$\Jarka(O,\Chain)$ consists of collections of chain complexes $\ttE =
\{\ttE_A(a',a'')\}$, indexed by objects $a',a'' \in A$ of small
dg-categories $A$.

 Theorem~\ref{Jarca1} therefore gives a
distinguished monoid $\ttM$ in the span-monoidal category $\Jarka(O,\Chain)$.
The monoid $\ttM = \{\ttM_A(a',a'')\}$ has a simple explicit
description.  For $A \in \cat$, one has
\[
\ttM_A(a',a'') := A(a',a''),
\]
the $\Chain$-enriched hom-functor in $A$. To specify a monoid structure of
$\ttM$, one needs to choose, for any dg-category $A \in \cat$ and any 
dg-functor $f: A \to B\in \cat$, the following three pieces of data:
\begin{align*}
\mubar_{\sfG(A)}& \in  \hskip -.3em \prod_{a',a,a'' \in A} \hskip -.3em
\Chain\ll A(a',a) \otimes A(a,a''),A(a',a'')\rr
\\
\nubar_{\sfG(A)}& \in \ \ \prod_{a \in  A} \
\Chain\ll k,A(a,a)\rr,\ \mbox {and}
\\
u_{\sfG(f)}& \in \prod_{a',a'' \in A}\Chain \ll A(a',a''),B(f(a'),f(a''))\rr.
\end{align*}
The element $\mubar_{\sfG(A)}$ is given by the enriched
composition in $A$, $\nubar_{\sfG(A)}$ by the enriched unit and
$u_{\sfG(f)}$ is part of the definition of the $\Chain$-enriched
functor $f$.

One can consider, as in Definition~\ref{sec:endom-oper-algebr}, 
the endomorphism $1$-operad $\End_\ttM$ of $\ttM$. The monoid
structure of $\ttM$ is, by
Proposition~\ref{sec:spn-mono-categ-2}, equivalent to a $1$-operad map  
$\fAss \to \End_\ttM$, i.e.~$\End_\ttM$ is a multiplicative operad in
the sense of Definition~\ref{sec:mutipl-1-oper}. By
Proposition~\ref{Lei}, $\End_\ttM$ therefore carries a natural
structure of a cosimplicial object in $\Chain$. The 
following  statement of this section is now obvious:

\begin{theorem}
\label{sec:tamark-compl-funct}
Let $f,g: A\to B$ be dg-functors between dg-categories. Then
the cosimplicial hom-functor $\bfhom^\bullet(f,g)$ defined 
in~\cite{tamarkin:CM07} and recalled above, 
is the fiber over the globe
\begin{equation}
\label{eq:14}
\raisebox{1.6em}{\rule{0pt}{0pt}}
\globe ABfg
\end{equation} 
of the cosimplicial span-object 
associated to the multiplicative endomorphism
$1$-operad $\End_\ttM$ of the distinguished monoid $\ttM$ in the 
span-monoidal category $\Jarka(O,\Chain)$. 
\end{theorem}

Tamarkin defined, in~\cite{tamarkin:CM07}, the right
derived hom-functor $\bfRhom(f,g)$ as the totalization
$|\bfhom^\bullet(f,g)|$  
of the cosimplicial  hom-functor $\bfhom^\bullet(f,g)$. In the terminology
of Subsection~\ref{sec:hochsch-object-homot},
$|\bfhom^\bullet(f,g)|$ is therefore the $\sfG$-fiber, where $\sfG$ is
the globe in~(\ref{eq:14}), of $\CH_{\delta}(\Id,\Id)$.

\section{The Deligne conjecture in  monoidal $\duo$-categories}
\label{sec:deligne-conj-mono}

For $n$ a positive integer, $n$-operads are higher analogs of
(nonsymmetric) operads. Their pieces have arities given by trees with
$n$-levels. While ordinary operads live in monoidal categories,
$n$-operads live in augmented monoidal $n$-globular categories. We
begin this section by introducing, for $n=0,1,2$, a simplified version
of $n$-operads tailored for the needs of the present paper. A general
approach can be found in~\cite{batanin:globular}. The relation between our
restricted case and the general one is addressed in
Remark~\ref{Jaruska_je_pusinka}.

\subsection{$2$-operads and their algebras in  duoidal $V$-categories.}

Let us recall the definition of the category $\Omega_k$ of $k$-trees, for $k \leq
2$. The category of {\em $0$-trees $\Omega_0$} is the terminal
category~$1$. Its unique object is denoted $U_0$.

The category of {\em $1$-trees $\Omega_1$} is the category of finite
ordinals $(n) := \{\rada 1n\}$, $n \geq 0$, and their order-preserving
maps.  As usual, we interpret $\{1,\ldots,n\}$  for $n=0$ as the empty~set.
The terminal object of $\Omega_1$ is denoted $U_1 := (1)$. When the
meaning is clear from the context, we will simplify the notation and
denote the object $(n) \in \Omega_1$ simply by $n$.

Notice that $\Omega_1$ is isomorphic to the `algebraic' version
$\Deltaalg$ of the basic simplicial category $\Delta$, i.e.~to
$\Delta$ augmented by the empty set, and can be characterized as the
free strict monoidal category generated by a monoid.  The category
$\Omega_1$ can also be interpreted as the subcategory of open maps of
Joyal's (skeletal) category of intervals $\I$, whereas $\Delta$ is
isomorphic to $\I^{op}$, see~\cite[Section~2]{batanin-markl}.

The definition of the category $\Omega_2$ of $2$-trees is more
involved. A {\it $2$-tree} $T$ is a morphism $t:n\to m$ in
$\Omega_1$.  {\it Leaves of height $2$ } of the tree $T$ are, by
definition, elements from $\{1,\ldots,n\}.$ {\it Leaves of height $1$}
of $T$ are those elements $i\in \{1,\ldots,m\}$ for which
$t^{-1}(i)=\emptyset.$ The set of all leaves of the tree $T$ has a
natural linear order defined by counting leaves when we are going
around the tree in the clockwise direction.

There are exactly two $2$-trees with one leaf. The tree $U_2 := (1 \to
1)$ has one leaf of height $2$ while the tree $zU_1 := (0\to 1)$ has
one leaf of height $1$.
The tree $z^2U_0 := (0 \to 0)$ has no leaves.  A {\em map of
$2$-trees}
\begin{equation}
\label{eq:16}
\sigma:T =(n\to m) \to S= (p\to q)
\end{equation}
is a commutative diagram of maps in $\Set$:

\begin{center}
{\unitlength=.8mm
\begin{picture}(200,28)(-45,10)
\thicklines
\put(40,35){\makebox(0,0){\mbox{$\{1,\ldots,n\}$}}}
\put(40,30){\vector(0,-1){15}}
\put(55,35){\vector(1,0){15}}
\put(55,10){\vector(1,0){15}}
\put(62,37){\makebox(0,0){\makebox(0,0)[b]{\scriptsize $t$}}}
\put(62,12){\makebox(0,0){\makebox(0,0)[b]{\scriptsize $s$}}}
\put(90,22){\makebox(0,0){\mbox{\scriptsize $\sigma_1$}}}
\put(35,22){\makebox(0,0){\mbox{\scriptsize $\sigma_2$}}}
\put(85,35){\makebox(0,0){\mbox{$\{1,\ldots,m\}$}}}
\put(85,30){\vector(0,-1){15}}
\put(40,10){\makebox(0,0){\mbox{$\{1,\ldots,p\}$}}}
\put(85,10){\makebox(0,0){\mbox{$ \{1,\ldots,q\}$}}}
\end{picture}}
\end{center}
such that 
\begin{itemize}
\item[(i)] 
$\sigma_1$ is order preserving and
\item[(ii)]  for any $i\in \{1,\ldots,m\}$, the restriction of
$\sigma_2$ to $t^{-1}(i)$ is order preserving.
\end{itemize}

We denote by $\Omega_2$ the category of {\em $2$-trees\/}. Its terminal object
is the tree $U_2 = (1 \to 1)$.  
The category $\Omega_2$ is a monoidal category with the
structure $+$ given by the fiberwise ordinal sum (gluing the
roots of $2$-trees in geometric terms). The unit of this monoidal
product is~$z^2U_0$.

\begin{remark}
There are obvious truncation functors
\begin{equation}
\label{free2cat}
\Omega_2\stackrel{\partial}{\to}
\Omega_1\stackrel{\partial}{\to} \Omega_0.
\end{equation} If we
consider them as the source-target functors $s=t=\partial$, then (\ref{free2cat})
becomes a strict monoidal $2$-globular category i.e.~the
$2$-categorical object in $\Cat$. It can be
characterized as being the free strict monoidal globular $2$-category generated
by an internal $2$-category \cite{batanin:street}.

Similarly, one can characterize the strict monoidal $1$-globular
category
\[
\Omega_1\stackrel{\partial}{\to} \Omega_0
\]
as the free strict monoidal globular $1$-category generated by an
internal $1$-category. Notice that this universal property is an
`extension' of the universal property of $\Delta_{alg}$ since any
strict monoidal category can be considered as a one object monoidal
globular $1$-category \cite{batanin:globular}.
\end{remark}

Any morphism of $2$-trees $\sigma:T\to S$ as in~(\ref{eq:16}) has its
{\em fibers\/}. Given a leaf $i\in \{1,\ldots,p\}$ of height $2$, the
restriction of $t$ determines an order preserving map
\[
t_i: \sigma_2^{-1}(i)\to \sigma_1^{-1}(s(i))
\] 
which we consider as a $2$-tree and call the {\em fiber\/} over the leaf
$i$. In the case of a leaf $j\in \{1,\ldots,q\}$ of height $1$,
$\sigma^{-1}_1(j)$ is an ordered subset of $\{1,\ldots,m\}$ 
which determines a $1$-tree in $\Omega_1$, the {\em fiber\/} over the leaf
$j$. With a slight abuse of notation, we will denote this fiber by
$\sigma^{-1}_1(j)$ and use the same convention throughout the rest of
this section.

Since the set of leaves of $S$ has a natural linear order, the set of
fibers of $\sigma$ also inherits this order. So, for a $\sigma:T\to S$
we will denote $T_1,\ldots,T_k$ the set of its fibers in this order.
Let us fix $n\in \{0,1,2 \}$.

In item~($\star$) of the next definition we consider the composition
$T\stackrel{\sigma}{\rightarrow} S \stackrel{\omega}{\rightarrow} R$
of maps of $n$-trees, $n \leq 2$. We will use the following
notation. Let $\Rada S1k$ be the fibers of $\omega$ and $T_i :=
\sigma^{-1}(S_i)$, $1 \leq i \leq k$. Denote also by $\sigma_i : T_i
\to S_i$ the restriction $\sigma|_{T_i}$ and
$\rada{T_{i,1}}{T_{i,m_i}}$ the fibers of $\sigma_i$. Clearly $\Rada
T1k$ are precisely the fibers of the composition $\omega\sigma$ and
\[
\{\rada {T_{1,1}}{T_{1,m_1}},\ldots,\rada {T_{k,1}}{T_{k,m_k}}\}
\]
the set of fibers of $\sigma$.

In the following definition where $V = (V,\otimes,I)$ is the basic
monoidal category, we introduce our restricted version of
$n$-operads. The terminology will be justified in Remark~\ref{Jaruska_je_pusinka}.

\begin{definition}
\label{defnoper}  Let $0\le n\le 2.$ 
An {\em $n$-operad in $V^{(n)}$\/} is a collection $A(T), \ T\in \Omega_i \
, \ i \le n$, of objects of $V$ equipped with the following structure:
\begin{itemize}
\item[(i)]
morphisms $\xi_i: I \rightarrow  A(U_i) \ , \ i \le n$ (the units);
\item[(ii)] 
for every morphism $\sigma:T \rightarrow S$ in $\Omega_i,
\ i\le n$, with fibers $\Rada T1k$ a morphism
\[
m_{\sigma}: A(T_1)\otimes \cdots \otimes A(T_k)\otimes A(S)
\rightarrow A(T) \mbox{\ \ (the multiplication}).
\]
\end{itemize}
The structure operations are required to satisfy the following conditions.
\begin{itemize}
\item[($\star$)]
For any composite $T\stackrel{\sigma}{\rightarrow} S
\stackrel{\omega}{\rightarrow} R$, the associativity diagram 
\begin{center}
\unitlength=1mm
\begin{picture}(100,40)(-30,3)
\thicklines
\put(-10,10){{\makebox(0,0)[t]{$\displaystyle\bigotimes_{1 \leq i \leq k}
 A({T_{i,\bullet}})\otimes A({S_{\bullet}}) \otimes A(R)$}}}
\put(-10,31){\vector(0,-1){19}}
\put(-10,15){\vector(0,1){17}}
\put(57,29){\makebox(0,0)[l]{{\scriptsize $m_{\omega\sigma}$}}}
\put(-12,22){\makebox(0,0)[r]{{\scriptsize $\simeq$}}}
\put(57,15){\makebox(0,0)[l]{{\scriptsize $m_\sigma$}}}
\put(27,9){\makebox(0,0)[b]{{\scriptsize $\id \ot m_\omega$}}}
\put(31,39.2){\makebox(0,0)[b]{{\scriptsize $\bigotimes_{i=1}^k m_{\sigma_i}\ot \id$}}}
\put(55,33){\vector(0,-1){6}}
\put(55,11){\vector(0,1){8}}
\put(19,7){\vector(1,0){15}}
\put(20,37.2){\vector(1,0){22}}
\put(-10,35){\makebox(0,0){$\displaystyle\bigotimes_{1 \leq i
      \leq k} \big(A({T_{i,\bullet}})\otimes A({S_{i}})\big)  \otimes A(R)$}}

\put(55,10){\makebox(0,0)[t]{{$\displaystyle\bigotimes_{1 \leq i \leq k}
 A({T_{i,\bullet}}) \otimes A(S)$}}}
\put(55,37){\makebox(0,0){{$A({T_{\bullet}} )\otimes A(R)$}}}
\put(55,22.5){\makebox(0,0){{$A(T)$}}}
\end{picture}
\end{center}
in which
\begin{align*}
A({S_{\bullet}})&: = A({S_1})\otimes \cdots \otimes A({S_k}),
\\
A({T_{i,\bullet}})&: = A({T_{i,1}}) \otimes  \cdots\otimes
A({T_{i,m_i}}),\ 1 \leq i \leq k,\
\mbox { and}
\\
A({T_{\bullet} })&:=  A({T_1})\otimes \cdots \otimes A({T_k});
\end{align*}
commutes.

\item[($\star\star$)]
For the identity $\sigma = id : T\rightarrow T$, the diagram
\begin{center}
{\unitlength=1mm
\begin{picture}(50,18)(60,6)
\thicklines
\put(95,20){\vector(-1,0){6}}
\put(60,17){\vector(0,-1){8}}
\put(60,20){\makebox(0,0){\mbox{
$A({U_{i_0}})\otimes \cdots \otimes A({U_{i_n}}) \otimes A(T)$}}}
\put(114,20){\makebox(0,0){\mbox{${I}\otimes \cdots \otimes {I} \otimes A(T)$}}}
\put(60,5){\makebox(0,0){\mbox{$A(T)$}}}
\put(105,15){\vector(-4,-1){36}}
\put(90,9){\makebox(0,0){\mbox{\scriptsize$\id$}}}
\put(58,12.5){\makebox(0,0)[r]{\mbox{\scriptsize$m_\id$}}}
\end{picture}}
\end{center}
commutes.

\item[($\star\!\star\!\star$)]
For $0\le i \le n$ and the 
unique morphism $T\rightarrow U_i$ in $\Omega_i$, the diagram
\begin{center}
{\unitlength=1mm
\begin{picture}(50,18)(60,6)
\thicklines
\put(87,20){\vector(-1,0){13}}
\put(62,17){\vector(0,-1){8}}
\put(60,20){\makebox(0,0){\mbox{\small$A(T)\otimes A({U_i}) $}}}
\put(98,20){\makebox(0,0){\mbox{\small$A(T)\otimes I$}}}
\put(60,5){\makebox(0,0){\mbox{\small$A(T)$}}}
\put(95,17){\vector(-3,-1){30}}
\put(84,11){\makebox(0,0){\mbox{\scriptsize $\id$}}}
\end{picture}}
\end{center}
commutes.
\end{itemize}
\end{definition} 

\begin{example} 
A $0$-operad in $V^{(0)}$ consists of an object $A(U_0)$. The
structure maps equip it with a monoid structure in $V$.
\end{example} 

\begin{example} 
A $1$-operad $A$ in $V^{(1)}$ is given by a nonsymmetric operad $A'$
in $V$ (which is the same as a $1$-operad in $V$ if we interpret $V$
as a duoidal category) with $A'(k) := A((k))$, $k \geq 0$, and a monoid
$A(U_0)$. The map of $1$-trees $id:(0)\to (0)$ induces an operadic
multiplication
\[
A(U_0)\otimes A((0)) \to A((0))
\]
which equips $A'(0)$ with a $A(U_0)$-module structure. This structure
is compatible with the rest of the operadic structure of $A'$ as in
Definition~\ref{1-operad} with $v$ replaced by $A(U_0)$.
\end{example}

\begin{definition}
The $0$-operad $\fAss_0$ defined as the monoid~$I \in V$. The
classical associativity $1$-operad $\fAss_1$ is such that $\fAss_1(T) :=
I$ for each $n$-tree $T$, $n \leq 1$. Similarly, we define the
$2$-operad $\fAss_2$ with $\fAss_2(T) := I$ for any $n$-tree $T$, $n
\leq 2$, with all structure maps being the canonical isomorphisms.
\end{definition}

\begin{definition}
Let $k<n$. The restriction of an $n$-operad $A$ on $\Omega_i , i\le k$,
is a $k$-operad $tr_k(A)$ in $V^{(k)}$ called the {\it $k$-truncation}  of $A$.
\end{definition}

\begin{definition}
An $n$-operad in $V^{(n)}$ is called {\it $0$-terminal} if $tr_0(A)= \fAss_0.$ 
An $n$-operad in $V^{(n)}$ is called {\it $1$-terminal} if $tr_1(A)= \fAss_1.$ 
\end{definition}

Nonsymmetric operads in $V$ are therefore exactly $0$-terminal
$1$-operads in $V^{(1)}$.

\begin{remark}
\label{Jaruska_je_pusinka}
According to \cite{batanin:globular}, general $n$-operads live in augmented monoidal
$n$-globular categories. The above notion of an $n$-operad in $V^{(n)}$ is the
specialization of this general notion to the augmented monoidal
$n$-globular category $V^{(n)}$ defined, for $n \leq 2$, as follows.

The category $V^{(0)}$ is just the category $V$ with its monoidal structure. 
The category $V^{(1)}$ is the following monoidal $1$-globular
category: in dimension $0$ we have $V,$ in dimension $1$ we have
$V\times V$. The source and target functors coincide and equal to
the projection on the second variable. The functor $z:V\to V\times V$
is defined by $z(x) = (I,x).$ The monoidal structure is induced the by
monoidal structure of $V$ in an obvious manner.

To construct $V^{(2)}$, we add to $V^{(1)}$ the product $V\times
V\times V$ in dimension $2$, with the projection to the second and
third coordinates as its $1$-source and $1$-target functors. We leave
to the interested reader to describe the rest of the augmented
monoidal structure of $V^{(2)}$.
\end{remark}

\begin{remark}
\label{Sigma} 
There is another construction of an augmented monoidal $n$-globular
category associated with a symmetric monoidal category $V.$ This
augmented monoidal $n$-globular category $\Sigma^n V$ has terminal
category $\mathbf1$ in dimensions strictly less than $n$ and $V$ in
dimension~$n$.  An $n$-operad in $\Sigma^n$ was called
{\em $(n-1)$-terminal operad in $V$\/} \cite{batanin:AM08}. The relation
between our terminology here and the terminology of
\cite{batanin:AM08} is following.

There is a globular functor $\Sigma^n (V) \to V^{(n)}$ which in
dimension $k<n$ sends a unique object of $\mathbf1$ to the
$(k+1)$-tuple $(I,\ldots,I)$ and in dimension $n$ it sends an object
$X\in V$ to the tuple $(X,I,\ldots,I).$ It is not hard to check that
this is an augmented monoidal globular inclusion. An $n$-operad in
$V^{(n)}$ is $(n-1)$-terminal in the present terminology if it takes
values in the subcategory $\Sigma^{n}(V).$ Therefore, our terminology
is compatible with the terminology of~\cite{batanin:AM08}.

\end{remark}

A $2$-tree $T = (n\stackrel{t}{\to} m)$ is called {\it pruned} if $t$
is an epimorphism. 
Equivalently, a $2$-tree is pruned if all its leaves are in height $2$.
Any $2$-tree $T$ contains the maximal pruned subtree
$\iota:T^{(p)}\to T$1. It is obvious that the fibers of $\iota$ are
$U_2$ or $z^2U_0$. For any $1$-terminal $2$-operad $A$ one therefore
has the
morphism
\begin{equation}
\label{pruned} 
A(T) \to I\otimes\cdots\otimes I\otimes A(T)\to
A(U_2)\otimes\cdots\otimes A(U_2)\otimes A(T) \to A(T^{(p)}).
\end{equation}

\begin{definition}
\label{prruned}
A $1$-terminal $2$-operad is {\em pruned\/} if (\ref{pruned}) is an
isomorphism for any $T\in \Omega_2$.
 \end{definition}

If $\cat$ is a $V$-category, then with every object $X\in \cat$ we can
associate a $0$-operad $\End_{X0}$, which is just the endomorphism
monoid $\cat(X,X).$

Let now $\euo = (\euo,\boxx,e)$ be a monoidal $V$-category. 
For any object $X\in \euo$
we will define its endomorphism $1$-operad $\End_{X1}$ in $V^{(1)}$. To do this, we
introduce the tensor power of $X$ as follows. 
\begin{itemize} 
\item[(0)]
With a unique $0$-tree $U_0$ we associate its tensor power as
\[
X^{U_0} := e.
\]
\item[(1)] 
With a $1$-tree $n\in \Omega_1$ we associate the tensor power 
\[
X^n := \underbrace{X\boxx\cdots\boxx X}_{n}
\]
with the convention 
\[
X^{0} = X^{zU_0} := e.
\]
\end{itemize}
The definition of $X^n$ above looks tautological, but notice that $n$
abbreviates $(n) \in \Omega_1$.

\begin{definition} 
The endomorphism $1$-operad of $X\in \euo$ is given by
\[
\End_{X1}(T) := \euo(X^T,X^{U_i}),
\]
where $T\in \Omega_i, \ i = 0,1$. 
\end{definition}

Let now $\duo$ be a duoidal $V$-category. For any object $X\in \duo$
we will define its endomorphism $2$-operad $\End_{X2}$ in $V^{(2)}$. To do this, we
define first the tensor power of an object $X$. 
\begin{itemize} 
\item[(0)]
With a unique $0$ tree $U_0$ we associate the tensor power
\[
X^{U_0} := e;
\]
\item[(1)] 
With a $1$-tree $n\in \Omega_1$ we associate the tensor power 
\[
X^n := \underbrace{v\boxx_0\cdots\boxx_0 v}_{n}
\]
with the convention 
\[
X^{0} = X^{zU_0} := e.
\]
\item[(2)] 
With a $2$-tree $T$ we associate the tensor power $X^T$ as
follows.  Let $T = (n\stackrel{t}{\to}m)\ne z^2U_0$ and let $n_i: =
t^{-1}(i)$ for $1 \leq i \leq m$. Then we put
\[
X^T := (X^{\boxx_1^{n_1}})\boxx_0\cdots\boxx_0 (  X^{\boxx_1^{n_m}}).
\]
We use here the convention that $X^{\boxx_1^{0}} := v$. 
We complete the definition by putting  
\[
X^{z^2U_0} := e.
\]  
\end{itemize}
We believe that the `ideological' portrait of $X^T$ in
Figure~\ref{symbol} clarifies our definition.

\begin{figure}
\begin{center}
{
\unitlength=1.2pt
\begin{picture}(150.00,80.00)(0.00,0.00)
\thicklines
\put(85.00,40.00){\makebox(0.00,0.00){$\cdots$}}
\put(130.00,70.00){\makebox(0.00,0.00){\scriptsize $\cdots$}}
\put(40.00,70.00){\makebox(0.00,0.00){\scriptsize $\cdots$}}
\put(110.00,40.00){\makebox(0.00,0.00)[r]{\scriptsize $\boxx_0$}}
\put(63.00,40.00){\makebox(0.00,0.00)[l]{\scriptsize $\boxx_0$}}
\put(43.00,40.00){\makebox(0.00,0.00){\scriptsize $\boxx_0$}}
\put(140.00,70.00){\makebox(0.00,0.00){\scriptsize $\boxx_1$}}
\put(120.00,70.00){\makebox(0.00,0.00){\scriptsize $\boxx_1$}}
\put(100.00,70.00){\makebox(0.00,0.00){\scriptsize $\boxx_1$}}
\put(50.00,70.00){\makebox(0.00,0.00){\scriptsize $\boxx_1$}}
\put(30.00,70.00){\makebox(0.00,0.00){\scriptsize $\boxx_1$}}
\put(10.00,70.00){\makebox(0.00,0.00){\scriptsize $\boxx_1$}}
\put(150.00,72.00){\makebox(0.00,0.00)[b]{$\hphantom{(}X)$}}
\put(110.00,74.00){\makebox(0.00,0.00)[b]{$X$}}
\put(90.00,72.00){\makebox(0.00,0.00)[b]{$(X \hphantom{)}$}}
\put(60.00,72.00){\makebox(0.00,0.00)[b]{$\hphantom{(}X)$}}
\put(20.00,74.00){\makebox(0.00,0.00)[b]{$X$}}
\put(0.00,72.00){\makebox(0.00,0.00)[b]{$(X \hphantom{)}$}}
\put(70.00,0.00){\makebox(0.00,0.00){\scriptsize $\bullet$}}
\put(30.00,40.00){\makebox(0.00,0.00){\scriptsize $\bullet$}}
\put(120.00,40.00){\makebox(0.00,0.00){\scriptsize $\bullet$}}
\put(150.00,70.00){\makebox(0.00,0.00){\scriptsize $\bullet$}}
\put(110.00,70.00){\makebox(0.00,0.00){\scriptsize $\bullet$}}
\put(90.00,70.00){\makebox(0.00,0.00){\scriptsize $\bullet$}}
\put(60.00,70.00){\makebox(0.00,0.00){\scriptsize $\bullet$}}
\put(20.00,70.00){\makebox(0.00,0.00){\scriptsize $\bullet$}}
\put(0.00,70.00){\makebox(0.00,0.00){\scriptsize $\bullet$}}
\put(57.0,40.00){\makebox(0.00,0.00){\scriptsize $\bullet$}}
\put(57.0,44.00){\makebox(0.00,0.00)[b]{$v$}}
\put(70.00,0.00){\line(5,4){50.00}}
\put(30.00,40.00){\line(1,-1){40.00}}
\put(120.00,40.00){\line(1,1){30.00}}
\put(120.00,40.00){\line(-1,3){10.00}}
\put(90.00,70.00){\line(1,-1){30.00}}
\put(30.00,40.00){\line(1,1){30.00}}
\put(20.00,70.00){\line(1,-3){10.00}}
\put(30.00,40.00){\line(-1,1){30.00}}
\put(70.00,0.00){\line(-1,3){13.00}}
\end{picture}}
\caption{\label{symbol}An `ideological' picture of $X^T$. Leaves
of height $2$ (resp.~$1$) are decorated by $X$ (resp.~$v$).  The
decorations of vertices of height $2$ (resp.~$1$) are then multiplied
by $\boxx_1$ (resp.~$\boxx_0$), with the $\boxx_1$-multiplication
performed first.}
\end{center}
\end{figure}
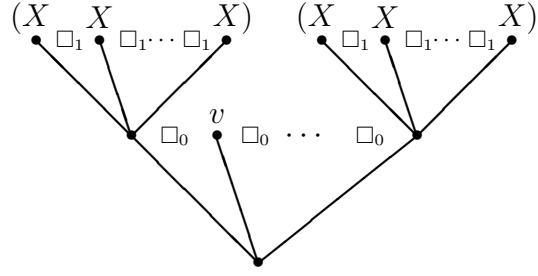

\begin{definition} 
The endomorphism $2$-operad of $X\in \duo$ is given by
\[
\End_{X2}(T) = \duo(X^T,X^{U_i}),
\]
where $T\in \Omega_i, \ i = 0,1,2$. 
\end{definition}

\begin{lemma}  
The collection $\End_{X2}(T), T\in \Omega_i, i = 0,1,2$, has a natural
structure of a $2$-operad in $V^{(2)}.$
\end{lemma}

\begin{proof} 
We construct first the units $\xi_i : I\to \End_{X2}(U_i), \ i=
0,1,2$. We have $\End_{X2}(U_0) = \duo(X^{U_0},X^{U_0}) = \duo(e,e),$
and we define $\xi_0 := \id_e:I\to \duo(e,e)$. Analogously we define
$\xi_1 :=\id_v:I\to \duo(v,v)$ and $\xi_2 := \id_X:I\to \duo(X,X)$.
 
The $0$-truncation of $\End_X$ is clearly the endomorphism monoid of
$e\in\duo.$ The $1$-truncation is the endomorphism operad of the
monoid $v\in \duo$ with the obvious multiplication.

To define the multiplication with respect to morphisms of $2$-trees, we
begin with the special case when the codomain of $\sigma:T\to S$ has the
form $S= (k\to 1).$ We will say that such a tree is a {\it suspension}
of the $1$-tree $k.$ If $k=0$, then $T = (0\to m)$ and the unique fiber
of $\sigma$ is equal to the $1$-tree $m.$ We define the operadic
multiplication as the composite in $\duo$:
\[ 
\duo(\underbrace{v\boxx_0\cdots\boxx_0 v}_m,v)\otimes 
\duo(v,X)\to \duo(\underbrace{v\boxx_0\cdots\boxx_0 v}_m,X).
\]

Suppose $k>0$ and $T= (n\to m).$ Then the fiber over a leaf $i \in
\{\rada 1k\}$ has the
form $T_i = (t_i: n_i\to m)$. 
{}For $j\in \{1,\ldots,m\}$ let $n_{ij}
:= t^{-1}_i(j).$ There is then a canonical morphism
\begin{equation}
\label{general-nterchange}
X^{\sigma}: X^T\to X^{T_1}\boxx_1\cdots\boxx_1 X^{T_{k}}.
\end{equation}
To see it, we  observe that 
\[
X^T \cong
\big( ( X^{\boxx_1^{n_{1,1}}})\boxx_1\cdots\boxx_1 (X^{\boxx_1^{n_{k,1}}})\big)
\boxx_0 \ldots \boxx_0 
\big((X^{\boxx_1^{n_{1,m}}})\boxx_1\cdots\boxx_1(X^{\boxx_1^{n_{k,
      m}}})\big),
\]
and
\[
X^{T_i} \cong 
(X^{\boxx_1^{n_{i,1}}})\boxx_0\cdots\boxx_0(  X^{\boxx_1^{n_{i,m}}}),\
1 \leq i \leq k.
\]
We define  $X^{\sigma}$ as the interchange morphism
\begin{align*}
X^T & \cong 
\big((X^{\boxx_1^{n_{1,1}}})\boxx_1\cdots\boxx_1(X^{\boxx_1^{n_{k,1}}})\big)
\boxx_0 \cdots \boxx_0 
\big((X^{\boxx_1^{n_{1,m}}})\boxx_1\cdots\boxx_1(X^{\boxx_1^{n_{k,m}}})\big)
 \longrightarrow 
\\
& \longrightarrow 
\big((X^{\boxx_1^{n_{1,1}}})\boxx_0\cdots\boxx_0(X^{\boxx_1^{n_{1,m}}})\big)
\boxx_1 \cdots \boxx_1 
\big((
X^{\boxx_1^{n_{k,1}}})\boxx_0\cdots\boxx_0(X^{\boxx_1^{n_{k,m}}})\big) 
\cong
\\
&  \cong X^{T_1}\boxx_1\cdots\boxx_1X^{T_{k}}. 
\end{align*}

The operadic multiplication $m_{\sigma}  : \End_{X2}(T_1) \ot \cdots \ot
\End_{X2}(T_k) \ot \End_{X2}(S) \to  \End_{X2}(T)$ is now 
defined as the composition
\begin{align*}
&\duo(X^{T_1},X)\otimes\cdots\otimes\duo(X^{T_k},X)\otimes \duo(X^S,X)
\to
\\
&\to \duo(X^{T_1}\boxx_1\cdots\boxx_1 X^{T_{k}},X^{\boxx_1^k})
\otimes  \duo(X^{\boxx_1^k},X)\to
\\
&\to \duo(X^T,X^S)\otimes\duo(X^S,X)\to \duo(X^T,X).
\end{align*}
The first map in the above composition exists because
$\boxx_1$ is a $V$-functor, the second map is induced by
$X^\sigma$.

Let now $S$ be a general $2$-tree. If $S= z^2U_0$ then $\sigma =
\id_{z^2U_0}$ and $m_{\sigma}$ is simply the composite
\[
\duo(e,e)\otimes \duo(e,X)\to \duo(e,X).
\]
Let $S\ne z^2U_0.$ Then $S $ is canonically the ordinal sum of trees,  
$S= P_1+\cdots+P_l,$ where $P_i$ is, for $1 \leq i \leq l$ ,
a suspension of a $1$-tree $k_i$. Moreover, there obviously exist
$2$-trees $\Rada Q1l$ such that $T = Q_1+ \cdots + Q_l$ and
$\sigma:T\to S$ is the sum 
$\sigma = \sigma_1 + \cdots +\sigma_l$, for some 
$\sigma_i:Q_i\to P_i$, $1 \leq i \leq l$. 
We denote $T_{i,j}$ the fiber of $\sigma$ over a leaf $j\in P_i$.
  
Observe that $X^T = X^{Q_1}\boxx_0\cdots\boxx_0 X^{Q_l}$ and 
$X^S = X^{P_1}\boxx_0\cdots\boxx_0 X^{P_l}$. We now
define
\[
X^{\sigma}: X^T\to (X^{T_{1,1}}\boxx_1\cdots\boxx_1
X^{T_{1,k_1}})\boxx_0\cdots \boxx_0 (X^{T_{l,1}}\boxx_1\cdots\boxx_1
X^{T_{l,k_l}})
\]
as the product $X^{\sigma} = X^{\sigma_1}\boxx_0\cdots \boxx_0 X^{\sigma_l}$.
Finally, we define the operadic multiplication $m_\sigma$~as 
\begin{align*}
&\duo(X^{T_{1,1}},X)\otimes\cdots
\otimes\duo(X^{T_{l,k_l}},X)\otimes\duo(X^S,X)\to
\\
&\to \duo(X^{T_{1,1}}\boxx_1\cdots\boxx_1 X^{T_{1,k_1}},X^{P_1})\otimes 
\cdots \otimes  \duo(X^{T_{l,1}}\boxx_1\cdots\boxx_1 
X^{T_{l,k_l}},X^{P_l})\otimes \duo(X^S,X) \to
\\
&\to \duo\big((X^{T_{1,1}}\boxx_1\cdots\boxx_1
X^{T_{1,k_1}})\boxx_0\cdots 
\boxx_0 (X^{T_{l,1}}\boxx_1\cdots\boxx_1
X^{T_{l,k_l}}),X^S\big)\otimes 
\duo(X^S,X)\to
\\   
&\to \duo(X^T,X^S)\otimes\duo(X^S,X)\to \duo(X^T,X).
\end{align*}
We used again that $\boxx_0$ and $\boxx_1$ are $V$-functors.  
We leave the tedious but obvious verification of the associativity of thus 
defined operadic multiplication to the reader.
\end{proof}

Observe that the $1$-truncation of the $2$-operad $\End_{X2}$ is the
endomorphism $1$-operad of the monoid $v\in (\duo,\boxx_0,e).$ So we
have a canonical operadic map
\[
k_v:\fAss_1\to tr_1(\End_{X2}).
\] 

\begin{definition}
\label{sec:2-operads-their}
An {\em algebra\/} of a pruned $2$-operad $A$ in $V^{(2)}$ is an object $X\in \duo$ equipped with a map of $2$-operads
\[
k: A\to \End_{X2}
\] 
such that $tr_1(k) = k_v$.
\end{definition}

As in Subsection~\ref{algebras}, one can show that  
$A$-algebras form a $V$-category. Notice also that a more precise name
for algebras in Definition~\ref{sec:2-operads-their} would be {\em
$1$-terminal\/} $A$-algebras, but we opted for a simpler terminology.

\begin{example}\label{algebraass_2}
We leave as an exercise for the reader to show that algebras of
$\fAss_2$ are exactly duoids in $\duo$.
\end{example}

A proof of the following theorem will be given in
\cite{batanin-berger-markl:homotopycenter}:

\begin{theorem}
\label{condensation} 
Let $\delta$ be a fixed cosimplicial object in $V.$ Then there is a
pruned $2$-operad $\Coend_{\Tam_2}(\delta)$ with a canonical action on
$\CH_{\delta}(A)$ for any multiplicative $1$-operad $A$ in $\duo$. In
particular, such an action exists on the $\delta$-center
$\CH_{\delta}(\ttM,\ttM)$ of a monoid $\ttM$ in a monoidal
$\duo$-category.

If $\delta = I$, then $\Coend_{\Tam_2}(\delta) = \fAss_2$ and the action
of $\Coend_{\Tam_2}(\delta)$ recovers the canonical duoid structure on
$\CH_{I}(A)$ constructed in Theorem~\ref{Za_pet_dni_stehovani.}. 
\end{theorem}

\begin{remark} 
The notation $\Coend_{\Tam_2}(\delta)$ comes from
\cite{batanin-berger}. In that paper the authors developed
techniques of condensation of symmetric colored operads. The operad
$\Coend_{\Tam_2}(\delta)$ is also a condensation, but we
condense a colored $2$-operad $\Tam_2$ instead of a
colored symmetric operad. The colored operad
$\Tam_2$ was actually constructed by Tamarkin in
\cite{tamarkin:CM07}. A~different and short description of $\Tam_2$
can be found in \cite[page~25]{batanin-berger}.
\end{remark}

\subsection{Deligne's conjecture for $\delta$-center of a monoid.}

Algebras of contractible $1$-operads in $\Chain$ are known as
$A_{\infty}$-algebras. In fact, we usually replace an action of a
contractible $1$-operad by a minimal cofibrant resolution of $\fAss_1$
to get a canonical notion of an $A_{\infty}$-algebra.  We use the same
philosophy and think about algebras of a contractible $2$-operad in
$V^{(2)}$ as duoids in $\duo$ up to all higher homotopies (see Example
\ref{algebraass_2}).

\begin{definition} 
Let $V$ be a monoidal model category and $n\le 2$. An $(n-1)$-terminal
$n$-operad $A$ equipped with an operad map $A\to \fAss_n$ is called
{\em contractible\/} if, for each $n$-tree~$T$, the map $A(T)\to \fAss_n(T)=I$ 
is a weak equivalence.
\end{definition} 
 
\begin{theorem}
\label{duoidalDeligne} 
Let $V$ be a monoidal model category and $\delta$ be a standard system
of simplices for $V$ such that the lattice path operad is strongly
$\delta$-reductive in the sense of
\cite[Definition~3.7]{batanin-berger}. 
Then the operad $\Coend_{\Tam_2}(\delta)$ is contractible.  In
particular, the $\delta$-center $\CH_{\delta}(\ttM,\ttM)$ of a monoid
$\ttM$ in a monoidal $\duo$-category is an algebra of a contractible
$2$-operad.
\end{theorem}

\begin{corollary}[duoidal Deligne's conjecture]
\label{homotopicalDeligne} 
There is a canonical action of a contractible $2$-operad on the
homotopical center of a monoid $\ttM$ which lifts the duoid structure
on the center of $\ttM$.
\end{corollary}

Theorem~\ref{homotopicalDeligne} has been proved by Tamarkin in \cite{tamarkin:CM07} 
for the particular case of the Tamarkin complex of the functor
\[
\Id:Cat(\Chain)\to Cat(\Chain)
\] 
thus answering the question {\it `what do $DG$-categories form?'} in
the title of that paper.  The
proofs of Theorems \ref{duoidalDeligne} and \ref{homotopicalDeligne}
will be addressed in \cite{batanin-berger-markl:homotopycenter}.

\subsection{Relation to the classical Deligne's conjecture}
Let $\duo$ be a cocomplete symmetric monoidal category. Then $e=v$ in
$\duo$ and, for any $X\in \duo$  and a $2$-tree $T$, one clearly has
$X^T \cong X^{T^{(p)}}$. This implies that 
$\End_{X2}$ satisfies condition (\ref{pruned}). The operad
$\End_{X2}$ is, however, not $1$-terminal, so it is not pruned  
in the sense of Definition~\ref{prruned}. 
But one can still construct a modified
pruned endomorphism $2$-operad $End_{X2}$ together with an operadic
map $End_{X2} \to\End_{X2}$ which is an isomorphism for all trees of
height $2.$ The operad $End_{X2}$ is determined by these conditions
uniquely.  Moreover, any morphism from a pruned $2$-operad $A$ to
$\End_{X2}$ can be factorized through $End_{X2}$.

The operad $End_{X2}$ is the endomorphism $2$-operad of $X$ in the
monoidal $2$-globular category~$\Sigma^2 V$. Therefore, the results of
\cite{batanin:conf,batanin:AM08} are applicable and an action $A \to
End_{X2}$ of $A$ on $X$ is equivalent to a (classical) action of the
symmetric operad $sym_2(A)$ on $X$.

If $V$ is a monoidal model category satisfying the requirements of
Theorem 8.6 or 8.7 of \cite{batanin:conf} and $A$ is a contractible
cofibrant $2$-operad, then the symmetrization $sym_2(A)$ has the homotopy
type of the little $2$-disks operad \cite{batanin:conf}. So we have

\begin{theorem}
\label{Deligne} 
If $\duo$ is a symmetric monoidal $V$-category and the assumptions of
Theorem~\ref{duoidalDeligne} are satisfied, then $\CH_{\delta}(A)$ of a
multiplicative operad $A$ in $\duo$ admits a structure of an algebra of a
$E_2$-operad. In particular, such an action exists on the homotopy
center $\CH(\ttM,\ttM)$ of a monoid $\ttM$ in a monoidal model
$\duo$-category. \end{theorem}

This form of Deligne's conjecture generalizes the classical one. As a
corollary we have

\begin{corollary}
The Hochschild complex of the $dg$-category $\calC$ (see Example
\ref{classicalcenter}) is an algebra of an $E_2$-operad. In particular,
if $\calC = \Sigma A$ for a unital associative algebra $A$, we get the
classical Deligne conjecture for the Hochschild complex of $A$.
\end{corollary}


\def\cprime{$'$}

\end{document}